\documentclass{elsarticle}
\usepackage[margin=1in]{geometry}

\usepackage{bbm}
\usepackage{graphicx}
\usepackage{pstool}
\usepackage{wrapfig}
\usepackage{caption}
\usepackage{subcaption}
\usepackage[section]{placeins} 

\usepackage{fancyhdr} 
\usepackage{lastpage} 
\usepackage{extramarks} 

\usepackage{xcolor} 
\usepackage{enumerate} 
\usepackage{paralist} 
\usepackage{amsmath, amsthm, amssymb, mathtools}
\usepackage{mathabx, pifont, stmaryrd} 
\usepackage[explicit]{titlesec} 
\usepackage{etoolbox} 
\usepackage{bibentry} 
\makeatletter\let\saved@bibitem\@bibitem\makeatother 
\usepackage[colorlinks, bookmarksopen, bookmarksnumbered,
            citecolor=red,urlcolor=red]{hyperref} 
\makeatletter\let\@bibitem\saved@bibitem\makeatother 

\usepackage{algorithm}
\usepackage{algorithmic}

\newtheorem{theorem}{Theorem}
\newtheorem{lemma}{Lemma}

\newtheorem{remark}{Remark}

\theoremstyle{definition}
\newtheorem{assume}{Assumption}

\usepackage{bm} 
\usepackage{amsfonts} 

\newcommand{\func}[3]{\ensuremath{#1 : #2 \rightarrow #3}}

\newcommand{\norm}[1]{\ensuremath{\left\| #1 \right\|}}

\newcommand{\suchthat}{\mathrel{}\middle|\mathrel{}}

\newcommand{\optunc}[2]{\underset{#1}{\text{minimize}} ~~ #2}
\newcommand{\argoptunc}[2]{\underset{#1}{\arg\min} ~~ #2}

\newcommand{\optconOne}[3]{
\begin{aligned}
& \underset{#1}{\text{minimize}}
& & #2 \\
& \text{subject to} & & #3
\end{aligned}}

\newcommand{\pder}[2]{\ensuremath{\frac{\partial #1}{\partial #2}}}

\newcommand{\Ccal}{\ensuremath{\mathcal{C}}}
\newcommand{\Dcal}{\ensuremath{\mathcal{D}}}
\newcommand{\Ecal}{\ensuremath{\mathcal{E}}}

\newcommand{\Hcal}{\ensuremath{\mathcal{H}}}

\newcommand{\Lcal}{\ensuremath{\mathcal{L}}}

\newcommand{\Pcal}{\ensuremath{\mathcal{P}}}

\newcommand{\Rcal}{\ensuremath{\mathcal{R}}}
\newcommand{\Scal}{\ensuremath{\mathcal{S}}}

\newcommand{\Ucal}{\ensuremath{\mathcal{U}}}
\newcommand{\Vcal}{\ensuremath{\mathcal{V}}}

\newcommand{\Xcal}{\ensuremath{\mathcal{X}}}
\newcommand{\Ycal}{\ensuremath{\mathcal{Y}}}
\newcommand{\Zcal}{\ensuremath{\mathcal{Z}}}

\newcommand{\Vboldcal}{\ensuremath{\boldsymbol{\mathcal{V}}}}

\newcommand{\Rbb}{\ensuremath{\mathbb{R} }}

\newcommand\Abm{{\ensuremath{\bm{A}}}}

\newcommand\Dbm{{\ensuremath{\bm{D}}}}

\newcommand\Hbm{{\ensuremath{\bm{H}}}}
\newcommand\Ibm{{\ensuremath{\bm{I}}}}

\newcommand\Kbm{{\ensuremath{\bm{K}}}}

\newcommand\Mbm{{\ensuremath{\bm{M}}}}

\newcommand\Pbm{{\ensuremath{\bm{P}}}}
\newcommand\Qbm{{\ensuremath{\bm{Q}}}}
\newcommand\Rbm{{\ensuremath{\bm{R}}}}

\newcommand\Ubm{{\ensuremath{\bm{U}}}}
\newcommand\Vbm{{\ensuremath{\bm{V}}}}

\newcommand\bbm{{\ensuremath{\bm{b}}}}

\newcommand\gbm{{\ensuremath{\bm{g}}}}

\newcommand\rbm{{\ensuremath{\bm{r}}}}

\newcommand\ubm{{\ensuremath{\bm{u}}}}
\newcommand\vbm{{\ensuremath{\bm{v}}}}
\newcommand\wbm{{\ensuremath{\bm{w}}}}
\newcommand\xbm{{\ensuremath{\bm{x}}}}
\newcommand\ybm{{\ensuremath{\bm{y}}}}
\newcommand\zbm{{\ensuremath{\bm{z}}}}

\newcommand\lambdabold{{\ensuremath{\boldsymbol{\lambda}}}}

\newcommand\deltabold{{\ensuremath{\boldsymbol{\delta}}}}
\newcommand\mubold{{\ensuremath{\boldsymbol{\mu}}}}

\newcommand\rhobold{{\ensuremath{\boldsymbol{\rho}}}}

\newcommand\phibold{{\ensuremath{\boldsymbol{\phi}}}}

\newcommand\varpibold{{\ensuremath{\boldsymbol{\varpi}}}}

\newcommand\Phibold{{\ensuremath{\boldsymbol{\Phi}}}}

\newcommand\Psibold{{\ensuremath{\boldsymbol{\Psi}}}}
\newcommand\Xibold{{\ensuremath{\boldsymbol{\Xi}}}}

\newcommand\zerobold{\ensuremath{\mathbf{0}}}
\newcommand\onebold{\ensuremath{\mathbf{1}}}

\usepackage{tikz}
\usepackage{pgfplots}
\usepackage{pgfplotstable, filecontents, booktabs}
\pgfplotsset{compat=1.9}

\usetikzlibrary{pgfplots.groupplots}
\usepgfplotslibrary{fillbetween}
\usetikzlibrary{calc,fit,matrix,arrows,automata,positioning,shapes}
\usetikzlibrary{arrows.meta}

\pgfplotsset{select coords between index/.style 2 args={
    x filter/.code={
        \ifnum\coordindex<#1\fi
        \ifnum\coordindex>#2\fi
    }
}}

\tikzset{
 invisible/.style={opacity=0},
 visible on/.style={alt={#1{}{invisible}}},
 alt/.code args={<#1>#2#3}{%
   \alt<#1>{\pgfkeysalso{#2}}{\pgfkeysalso{#3}}
 },
}

\theoremstyle{definition}

\newtheorem{corollary}{Corollary}[theorem]

\newcommand{\paren}[1]{\ensuremath{\left( #1 \right)}}
\newcommand{\bracket}[1]{\ensuremath{\left[ #1 \right]}}
\newcommand{\curlyb}[1]{\ensuremath{\left\{ #1 \right\} }}
\newcommand{\sens}[1]{\ensuremath{\pder{#1}{\mubold}}}
\newcommand{\abs}[1]{\ensuremath{\left| #1 \right|}}

\newcommand{\trbm}{\Tilde{\rbm}}

\newcommand{\tybm}{\Tilde{\ybm}}

\newcommand{\tlam}{\Tilde{\lambdabold}}

\newcommand{\hrbm}{\hat{\rbm}}

\newcommand{\hybm}{\hat{\ybm}}

\newcommand{\hlam}{\hat{\lambdabold}}

\newcommand{\hj}{\hat{j}}

\newcommand{\tf}{\Tilde{f}}
\newcommand{\tj}{\Tilde{j}}

\newcommand{\rlam}{\rbm^\lambda}
\newcommand{\glam}{\gbm^\lambda}

\newcommand{\hrlam}{\hat{\rbm}^\lambda}
\newcommand{\hglam}{\hat{\gbm}^\lambda}

\newcommand{\trlam}{\tilde{\rbm}^\lambda}
\newcommand{\tglam}{\tilde{\gbm}^\lambda}

\usepackage{setspace}
\usepackage{enumitem}
\usepackage{ulem}

\newbool{fastcompile}
\setbool{fastcompile}{false}

\begin{document}
\title{A globally convergent method to accelerate large-scale optimization using on-the-fly model hyperreduction: application to shape optimization}

\author[rvt1]{Tianshu Wen\fnref{fn1}}
\ead{twen2@nd.edu}

\author[rvt1]{Matthew J. Zahr\fnref{fn2}\corref{cor1}}
\ead{mzahr@nd.edu}

\address[rvt1]{Department of Aerospace and Mechanical Engineering, University
               of Notre Dame, Notre Dame, IN 46556, United States}
\cortext[cor1]{Corresponding author}

\fntext[fn1]{Graduate Student, Department of Aerospace and Mechanical
             Engineering, University of Notre Dame}
\fntext[fn2]{Assistant Professor, Department of Aerospace and Mechanical
             Engineering, University of Notre Dame}


\begin{abstract}
    We present a numerical method to efficiently solve optimization problems governed by large-scale nonlinear systems of
    equations, including discretized partial differential equations, using projection-based reduced-order models accelerated
    with hyperreduction (empirical quadrature) and embedded in a trust-region framework that guarantees global convergence.
    The proposed framework constructs a hyperreduced model \textit{on-the-fly} during the solution of the optimization problem,
    which completely avoids an offline training phase. This ensures all snapshot information is collected along the optimization
    trajectory, which avoids wasting samples in remote regions of the parameters space that are never visited, and inherently avoids
    the curse of dimensionality of sampling in a high-dimensional parameter space. At each iteration of the proposed algorithm, a
    reduced basis and empirical quadrature weights are constructed precisely to ensure the global convergence criteria of the
    trust-region method are satisfied, ensuring global convergence to a local minimum of the original (unreduced) problem.
    Numerical experiments are performed on two fluid shape optimization problems to verify the global convergence of the method
    and demonstrate its computational efficiency; speedups over $18\times$ (accounting for all computational cost, even cost that is
    traditionally considered ``offline'' such as snapshot collection and data compression) relative to standard optimization
    approaches that do not leverage model reduction are shown.
\end{abstract}
    
\maketitle

\section{Introduction}
\label{sec:intro}
Optimization problems, particularly those constrained by partial differential equations (PDEs), arise in almost every branch of engineering and science. However, they can be computationally expensive to solve because they require numerous
queries to the underlying model simulation, which can be demanding in many applications, particularly those involving fluid
flows. To circumvent the computational cost of optimization problems constrained by PDE simulations, a slew of surrogate-based
approaches have been developed whereby the expensive simulation is replaced with an inexpensive approximation model.
The work by Alexandrov \cite{alexandrov_trust-region_1998, alexandrov_approximation_2001} provides a practical and rigorous framework for leveraging surrogate models to accelerate optimization problems to ensures convergence to a local optimum.
Surrogate models used to accelerate optimization problems come in many forms, e.g.,  adaptive spatial discretizations
\cite{ziems_adaptive_2011}, partially converged solutions \cite{forrester_optimization_2006,zahr_adaptive_2016},
response surfaces \cite{forrester_recent_2009}, projection-based reduced-order models (ROM) \cite{arian_trust-region_2000, qian_certified_2017, zahr_progressive_2015,zahr_efficient_2019, yano_globally_2021, heinkenschloss_reduced_2022}, and deep learning \cite{tao_application_2019, lye_iterative_2021},
to name a few. In this work, we focus solely on projection-based model reduction surrogates.

For nonlinear problems, it is well-known that projection to a lower dimensional space is not sufficient to
realize computational efficiency because there is no opportunity to precompute the expensive projection
or assembly operations. To circumvent this bottleneck, numerous hyperreduction techniques have been developed
\cite{barrault_empirical_2004, chaturantabut_discrete_2009, farhat_dimensional_2014, yano_discontinuous_2019}
to specifically address the cost of evaluating nonlinear terms, usually through empirical interpolation or empirical
quadrature. While many of these methods have proven successful, we focus on the empirical quadrature procedure (EQP)
because it has proven effective for finite-element-based methods and provides direct control of
the output error \cite{yano_discontinuous_2019}. This method restores computational efficiency by only requiring assembly
over a small portion of the mesh, determined by a highly sparse vector of element weights computed as the
solution of an optimization problem. In this work, we aim to leverage hyperreduced models to accelerate optimization problems.

There are two primary approaches to accelerate optimization problems using ROMs:
(i) offline/online approaches and (ii) on-the-fly approaches.
In the offline/online approaches \cite{legresley_airfoil_2000, manzoni_shape_2012, ripepi_reduced-order_2018},
a ROM is constructed in an offline phase by sampling the parameter space, e.g., using a greedy method, and in
the online phase it is used to solve one or more optimization problems.
This approach tends to be the most popular as it allows much of the existing projection-based
model reduction infrastructure to be re-used, and as such they have been used in numerous settings
including shape optimization \cite{legresley_airfoil_2000,manzoni_shape_2012},
reservoir optimization \cite{trehan_trajectory_2016}, and data assimilation \cite{stefanescu_poddeim_2015}, to name a few.
Several hyperreduction approaches have been embedded in this offline/online framework to
accelerate optimization problems with nonlinear PDE constraints, including the Discrete Empirical
Interpolation Method (DEIM ) \cite{amsallem_design_2015,stefanescu_poddeim_2015,yang_efficient_2017,esmaeili_generalized_2020,choi_gradient-based_2020},
piecewise quadratic approximations \cite{trehan_trajectory_2016}, the Energy Conserving Sampling and
Weighting (ECSW) method \cite{scheffold_vibration_2018}, and Koopman-based methods \cite{renganathan_koopman-based_2020}.
However, because a single ROM is built in the offline phase that will be used throughout the optimization trajectory, it
must be accurate in all regions of the parameter space the optimization solver may visit. This requires training
in a potentially high-dimensional parameter space, which suffers from the curse of dimensionality and will likely
lead to unnecessary samples if certain regions of the parameter space are not visited during the optimization iterations.
There has been promising work that aims to overcome the curse of dimensionality in training these models using
active manifolds \cite{boncoraglio_active_2021}. For special classes of optimization problems, these offline/online
methods can be shown to deliver accurate approximations to the optimal solution \cite{dihlmann_certified_2015},
although there are no guarantees for general problems.

On the other hand, on-the-fly approaches \cite{arian_trust-region_2000, yue_accelerating_2013,agarwal_trust-region_2013,zahr_progressive_2015, zahr_adaptive_2016, qian_certified_2017, zahr_efficient_2019, yano_globally_2021,keil_non-conforming_2021,banholzer_trust_2022} have no training phases because they build up a ROM during the optimization procedure.
They often yield more modest speedups because they account for all cost, including the cost of
training simulations and basis construction, but can ensure convergence to a critical point of the
original (unreduced) optimization problem \cite{zahr_adaptive_2016,qian_certified_2017,yano_globally_2021,banholzer_trust_2022}.
These methods are becoming increasingly popular as they have now been developed for optimization under uncertainty \cite{zahr_efficient_2019}, topology optimization \cite{yano_globally_2021}, multiobjective optimization \cite{banholzer_trust_2022}, oil reservoir optimization \cite{suwartadi_adjoint-based_2015}, shape optimization \cite{zahr_progressive_2015,zahr_adaptive_2016,marques_non-intrusive_2021}, multiscale parameter estimation \cite{keil_relaxed_2022}, and general linear PDE-constrained optimization \cite{yue_accelerating_2013,qian_certified_2017}.
Most of these methods either address linear problems \cite{yue_accelerating_2013,qian_certified_2017,yano_globally_2021}
or do not incorporate hyperreduction \cite{zahr_progressive_2015,zahr_efficient_2019}. The work that does directly address
nonlinear problems and incorporate model hyperreduction \cite{suwartadi_adjoint-based_2015,sorek_model_2017,marques_non-intrusive_2021} are not currently equipped with global convergence theory.


The main contribution of this work is a new surrogate-based optimization method that embeds projection-based hyperreduced
(EQP) models in a globally convergent trust-region method that allows for models with inexact gradients and asymptotic error bounds \cite{kouri_trust-region_2013}. Our method constructs a hyperreduced model \textit{on-the-fly} at each trust-region center and uses
it as the trust-region approximation model, which can be rapidly queried many times to solve the trust-region subproblem and
advance toward the optimal solution. This completely avoids an offline training phase and ensures all snapshot information is
collected along the optimization trajectory as opposed to remote regions of the parameter space that are never visited. An on-the-fly
approach also circumvents the curse of dimensionality that comes from trying to sample a high-dimensional parameter space
(we consider problems with up to 18 shape parameters in this work).

In the proposed approach, at each trust-region center, a reduced basis and empirical quadrature weights are constructed to ensure
global convergence of the method. To accomplish this, the EQP method in \cite{yano_discontinuous_2019, yano_lp_2019}
was enhanced with optimization-based constraints, e.g., on the adjoint residual, quantity of interest, and gradient reconstruction,
and the tolerance associated with each constraint is automatically chosen based on the trust-region convergence criteria.
Global convergence to a local solution of the original (unreduced) optimization problem is guaranteed under regularity and
boundedness assumptions on various system operators and quantities
of interest. Lastly, because our primary interest is shape optimization, we introduce an approach for efficient physics-based
mesh motion (e.g., linear elasticity or spring analogy) in the model reduction setting as an alternative to volumetric shape
parametrization \cite{manzoni_shape_2012}. Global convergence of the method is verified for two shape optimization problems, and
speedup up over $18\times$ is demonstrated relative to standard optimization approaches when accounting for all sources of computational
cost. The four main contributions of this work are summarized as:
\begin{enumerate}[label=(\roman*)]
 \item an empirical quadrature procedure with additional constraints tailored to the optimization setting,
 \item an approach to accelerate physics-based mesh motion (e.g., linear elasticity, spring systems) using linear model reduction,
 \item a trust-region method with hyperreduced approximation models (EQP/TR), and
 \item global convergence theory for the proposed EQP/TR method.
\end{enumerate}

The remainder of this paper is organized as follows. Section \ref{sec:pdeopt} introduces the governing nonlinear system,
quantity of interest, an associated optimization problem, and their specialization to the case of PDE-based shape optimization.
Section~\ref{sec:hyperreduction} introduces projection-based model reduction and hyperreduction based on empirical
quadrature, including the proposed optimization-based EQP constraints and relevant error estimates. Section~\ref{sec:trammo}
provides a brief overview of the trust-region method of \cite{kouri_trust-region_2013} that uses inexact gradient evaluations and asymptotic error
bounds, and subsequently introduces the proposed EQP/TR method and establishes its global convergence.
Finally, Section \ref{sec:numexp} presents verifies global convergence of the method and demonstrates its computational efficiency.

\section{Problem formulation}
\label{sec:pdeopt}
In this section, we formulate the governing high-dimensional optimization problem that we aim to accelerate using model hyperreduction in later sections. We consider general optimization problems constrained by a large-scale nonlinear system of equations, although our primary interest is
fully discrete PDE-constrained optimization. To this end, we introduce the governing equations (Section \ref{sec:pdeopt:gov_eqn}), quantities of interest that will be the optimization objective (Section \ref{sec:pdeopt:qoi}), the complete optimization formulation including the adjoint method to compute gradients (Section \ref{sec:pdeopt:disc_pdeopt}), and the specialization of the framework to shape optimization (Section \ref{sec:pdeopt:shape_opt}).

\subsection{Governing equations}
\label{sec:pdeopt:gov_eqn}
We consider a parametrized, large-scale system of nonlinear equations that we will refer to as the high-dimensional model (HDM):
given a collection of system parameters
$\mubold\in\Dcal\subset\Rbb^{N_\mubold}$, find $\ubm^\star\in\Rbb^{N_\ubm}$ such that
\begin{equation}\label{eqn:hdm:res_eqn}
    \rbm(\ubm^\star, \mubold) = \zerobold,
\end{equation}
where $\ubm^\star$ is the primal solution implicitly defined as the solution of (\ref{eqn:hdm:res_eqn}) and  $\rbm: \Rbb^{N_\ubm}\times\Dcal \mapsto \Rbb^{N_\ubm}$ with $\rbm : (\ubm, \mubold) \mapsto \rbm(\ubm,\mubold)$ denotes the residual. We assume that for every $\mubold\in\Dcal$, there exists a unique primal solution $\ubm^\star=\ubm^\star(\mubold)$ satisfying (\ref{eqn:hdm:res_eqn}) and that the implicit map $\mubold \mapsto \ubm^\star(\mubold)$ is continuously differentiable. In most practical applications, the dimension $N_\ubm$ is large, which causes (\ref{eqn:hdm:res_eqn}) to be computationally expensive to solve. We focus on the case where the residual arises from the high-fidelity discretization of a parametrized, partial differential equation (PDE), although the developments in this work generalize to naturally discrete systems. Furthermore, we assume the system ($\Omega$) is comprised on $N_\mathtt{e}$ elements ($\Omega_e$), i.e., $\Omega=\cup_{e=1}^{N_\mathtt{e}} \Omega_e$, and the residual can be written as an assembly of element residuals
\begin{equation}\label{eqn:hdm:res_elem}
    \rbm(\ubm, \mubold)  
    = \sum_{e=1}^{N_\mathtt{e}} \Pbm_e \rbm_e \paren{\ubm_e, \ubm_e', \mubold},
\end{equation}
where $\rbm_e: \Rbb^{N_\ubm^e}\times\Rbb^{N_\ubm^{e'}}\times\Dcal \mapsto \Rbb^{N_\ubm^e}$ with
$\rbm_e : (\ubm_e, \ubm_e', \mubold) \mapsto \rbm_e(\ubm_e,\ubm_e',\mubold)$
is the residual contribution of element $e$, $\ubm_e\in\Rbb^{N_\ubm^e}$ are the DoFs associated with element
$e$, $\ubm_e'\in\Rbb^{N_\ubm^{e'}}$ are the DoFs associated with elements neighboring element $e$ (if applicable),
and $\Pbm_e \in \Rbb^{N_\ubm\times N_\ubm^e}$ is the assembly operator that maps element DoFs for
element $e$ to global DoFs. Furthermore, the element and global DoFs are related via assembly operators
$\ubm_e = \Pbm_e^T \ubm$ and $\ubm_e' = (\Pbm_e')^T\ubm$, where $\Pbm_e' \in \Rbb^{N_\ubm\times N_\ubm^{e'}}$
is an assembly operator that maps element DoFs from neighbors of element $e$ to the corresponding
global DoFs. This structure will be used to facilitate hyperreduction of the nonlinear system of equations.

\begin{remark}
\label{rem:govern}
In the PDE setting, we consider two concrete examples of the framework introduced in this section. Because any of the
quantities at the continuous level can depend on the parameter vector $\mubold$, e.g., the flux function, source term,
or PDE domain, we only include the parameter vector once the continuous terms are converted to their discrete,
algebraic forms for brevity.
\begin{itemize}
 \item Consider the system of $m$ static inviscid
conservation laws on a domain $\Omega \subset \Rbb^d$ subject to appropriate boundary conditions
    \begin{equation}\label{eqn:hdm:static_claw}
        \nabla \cdot F(u) = S(u) \quad \text{in} \quad \Omega
    \end{equation}
where $u: \Omega \rightarrow \Rbb^m$ is the solution of the conservation laws, $F: \Rbb^m \rightarrow \Rbb^{m \times d}$ is the flux function, $S:\Rbb^m\rightarrow\Rbb^m$ is the source term, $\nabla$ is the gradient operator, and $\partial \Omega$ is the boundary of the domain with outward unit normal $n: \partial \Omega \rightarrow \Rbb^d$. Let $\Ecal_h$ be a mesh of $\Omega$, i.e., a collection of non-overlapping, potentially curved elements that cover $\Omega$ such that $K\subset\Omega$ for any $K\in\Ecal_h$, where $|\Ecal_h| = N_\mathtt{e}$. Furthermore, we assume
the elements are ordered, i.e., $\Ecal_h = \{\Omega_1,\Omega_2,\dots,\Omega_{N_\mathtt{e}}\}$. Then, we apply a discontinuous Galerkin
discretization to (\ref{eqn:hdm:static_claw}) to yield the Galerkin form: find $u_h \in \Vcal_h$ such that
\begin{equation}\label{eqn:hdm:disc_dg}
    r_e : (\psi_h^e,u_h) \mapsto \int_{\partial \Omega_e} \psi_h^e \cdot \Hcal(u_h^+, u_h^-, n_h) \, dS -\int_{\Omega_e} F(u_h): \nabla\psi_h^e \, dV - \int_{\Omega_e}  \psi_h^e \cdot S(u_h) \, dV = 0
\end{equation}
holds for all $\psi_h^e \in \Vcal_h^e$ and $e = 1, \dots, N_\mathtt{e}$, where $\Vcal_h^e \subset \Pcal^p(\Omega_e)$ is a polynomial
space over $\Omega_e$ of degree $p$, $\Vcal_h = \bigoplus_{e=1}^{N_\mathtt{e}} \Vcal_h^e \subset L^2(\Omega)$ is the global
trial/test Hilbert space (piecewise polynomial),  $\func{\Hcal}{\Rbb^m\times\Rbb^m\times\Rbb^d}{\Rbb^m}$ is the numerical flux function,
and $w_h^+$ ($w_h^-$) denotes the interior (exterior) trace of $w_h\in\Vcal_h$ to element $\Omega_e$. The global Galerkin form is
defined by summing all element residuals
\begin{equation}\label{eqn:hdm:dg_res_elem}
 r : (\psi_h, u_h) \mapsto \sum_{e=1}^{N_\mathtt{e}} r_e(\left.\psi_h\right|_{\Omega_e}, u_h) = 0,
\end{equation}
for all $\psi_h \in \Vcal_h$. Next, we select a basis for
the local function space ($\Vcal_h^e$) to reduce the element Galerkin form ($r_e$) to algebraic form
$\rbm_e(\ubm_e, \ubm_e', \mubold)$, where $\ubm_e\in\Rbb^{N_\ubm^e}$ are the DoFs of $u_h$ associated
with element $\Omega_e$ and $\ubm_e'\in\Rbb^{N_\ubm^{e'}}$ are the DoFs of $u_h$ associated
to all elements of $\Ecal_h$ that share a face with $\Omega_e$. In the DG setting, the global solution vector $\ubm$
(DoFs of $u_h$) and residual function $\rbm$ (algebraic form of $r$) are simply the concatenation of all
element contributions, i.e.,
\begin{equation}
 \ubm = (\ubm_1, \ubm_2, \dots, \ubm_{N_\mathtt{e}}), \qquad
 \rbm(\ubm,\mubold) = (\rbm_1(\ubm_1,\ubm_1',\mubold), \rbm_2(\ubm_2,\ubm_2',\mubold), \dots, \rbm_{N_\mathtt{e}}(\ubm_{N_\mathtt{e}},\ubm_{N_\mathtt{e}}', \mubold)).
\end{equation}
  \item Next, we consider a system of $m$ static viscous conservation laws on a domain $\Omega \subset \Rbb^d$ subject to
   appropriate boundary conditions
    \begin{equation}\label{eqn:hdm:static_vclaw}
        \nabla \cdot F(u,\nabla u) = S(u, \nabla u) \quad \text{in} \quad \Omega,
    \end{equation}
    where all terms are identical to those in (\ref{eqn:hdm:static_claw}) except the flux function
    $F: \Rbb^m\times\Rbb^{m\times d} \rightarrow \Rbb^{m \times d}$  and source term $S:\Rbb^m\times\Rbb^{m\times d} \rightarrow \Rbb^m$
    also depends on the gradient $\nabla u$. Using the mesh described above with ordered elements, we apply a continuous Galerkin
    discretization to yield the Galerkin form: find $u_h \in \Vcal_h$ such that
 \begin{equation}\label{eqn:hdm:disc_cg}
    r : (\psi_h, u_h) \mapsto
    \int_{\partial \Omega} \psi_h \cdot F_\partial(u_h,\nabla u_h, n_h) \, dS -\int_\Omega F(u_h, \nabla u_h): \nabla\psi_h \, dV - \int_\Omega \psi_h \cdot S(u_h,\nabla u_h) \, dV= 0
\end{equation}
 holds for all $\psi_h \in \Vcal_h$, where $\Vcal_h\subset H^1(\Omega)$ is the trial Hilbert space (continuous, piecewise polynomial) and
 $\func{F_\partial}{\Rbb^m\times\Rbb^{m\times d}\times\Rbb^d}{\Rbb^m}$ is the boundary condition flux. The residual ($r$) in
 (\ref{eqn:hdm:disc_cg}) can be split into element contributions as
  \begin{equation}\label{pde:eqn:disc_cg_elem}
    r(\psi_h, u_h) = \sum_{e=1}^{N_\mathtt{e}} r_e(\psi_h, u_h),
\end{equation}
where the element contributions are defined by leveraging the additive property of integration
 \begin{equation}\label{eqn:hdm:elem_res}
    r_e : (\psi_h, u_h) \mapsto
    \int_{\partial \Omega_e \cap \partial\Omega} \psi_h \cdot F_\partial(u_h,\nabla u_h, n_h) \, dS -\int_{\Omega_e} F(u_h, \nabla u_h): \nabla\psi_h \, dV -\int_{\Omega_e} \psi_h \cdot S(u_h, \nabla u_h) \, dV.
\end{equation}
Next, we select a basis for the Hilbert space ($\Vcal_h$) to reduce the global residual ($r$) in (\ref{eqn:hdm:disc_cg}) to an
algebraic system of nonlinear equations $\rbm(\ubm,\mubold) = \zerobold$, where $\ubm\in\Rbb^{N_\ubm}$ are the global DoFs of $u_h$.
Similarly, the element residual ($r_e$) in (\ref{eqn:hdm:elem_res}) reduces to an algebraic function $\rbm_e(\ubm_e,\mubold)$,
where $\ubm_e\in\Rbb^{N_\ubm^e}$ are the
DoFs of $u_h$ associated with element $\Omega_e$. Unlike the DG setting, the element residual solely depends on its
own DoFs, i.e., $\ubm_e' = \emptyset$, because DoFs at the intersection of elements are shared. In this
setting, assembling the local DoFs or residuals into global vectors is a true assembly operation.
\end{itemize}
There are many other examples of PDE discretizations that fit the setting outlined in this section, including cell-centered finite volume and
other finite-element-based discretizations, as well as naturally discrete systems (e.g., direct stiffness analysis of trusses).
\end{remark}

\subsection{Quantities of interest}
\label{sec:pdeopt:qoi}
Next, we introduce a quantity of interest (QoI), $\func{j}{\Rbb^{N_\ubm}\times\Dcal}{\Rbb}$, that maps the state
vector $\ubm$ and parameter configuration $\mubold$ into a scalar as $j : (\ubm,\mubold) \rightarrow j(\ubm,\mubold)$. Similar
to the residual, we assume the QoI can be written as a summation of element contributions
\begin{equation}\label{eqn:hdm:qoi_elem}
    j(\ubm, \mubold) = \sum^{N_\mathtt{e}}_{e} j_e(\ubm_e, \mubold),
\end{equation}
where $j_e : (\ubm_e,\mubold) \rightarrow j_e(\ubm_e,\mubold)$ with $\func{j_e}{\Rbb^{N_\ubm^\mathtt{e}}\times\Dcal}{\Rbb}$
is the element contribution of element $e$ to the QoI.

\begin{remark}
\label{rem:qoi}
In the PDE setting, quantities of interest are usually integrated quantities over the volume or boundary. For generality, we consider a single integrated quantity with a volume and boundary term
\begin{equation}
 q : u \mapsto \int_\Omega q_\mathtt{v}(u,\nabla u) \, dV + \int_{\partial\Omega} q_\mathtt{b}(u,\nabla u, n) \, dS,
\end{equation}
where $q_\mathtt{v}$ is the volumetric integrand and $q_\mathtt{b}$ is the boundary integrand. For consistency with Remark~\ref{rem:govern},
we omit the dependence on the parameters $\mubold$ at the continuous level.
Assuming the PDE discretization is a finite element method, it is natural to use the same approximations for the quantity of
interest for a consistent discretization \cite{zahr_implicit_2020}, which leads to the discretized quantity of interest $q_h : \Vcal_h \mapsto \Rbb$ with
\begin{equation}
 q_h : u_h \mapsto \int_\Omega q_\mathtt{v}(u_h,\nabla u_h) \, dV + \int_{\partial\Omega} q_\mathtt{b}(u_h,\nabla u_h, n_h) \, dS.
\end{equation}
From this, we define the element contribution to $q_h$ as $q_h^e :\Vcal_h^e \mapsto \Rbb$ with
\begin{equation}
 q_h^e : u_h \mapsto \int_{\Omega_e} q_\mathtt{v}(u_h,\nabla u_h) \, dV + \int_{\partial\Omega_e\cap\partial\Omega} q_\mathtt{b}(u_h,\nabla u_h, n_h) \, dS,
\end{equation}
which relates to the complete quantity of interest as
\begin{equation}
 q_h = \sum_{e=1}^{N_\mathtt{e}} q_h^e.
\end{equation}
By introducing a basis for the local function space $\Vcal_h^e$ as in Remark~\ref{rem:govern}, the continuous QoI ($q_h^e$) becomes an
algebraic expression $j_e(\ubm_e,\mubold)$, where $\ubm_e\in\Rbb^{N_\ubm^e}$ are the DoFs of $u_h$ associated
with element $\Omega_e$. The element contributions ($j_e$) are summed over all elements in the mesh to yield the complete QoI
($j$) in (\ref{eqn:hdm:qoi_elem}).
\end{remark}

\subsection{Optimization formulation}
\label{sec:pdeopt:disc_pdeopt}
Finally, we pose the central constrained optimization problem considered in this work
\begin{equation}\label{eqn:hdm:opt:full_space}
    \optconOne{\ubm \in \Rbb^{N_\ubm}, \mubold \in \Dcal}{j(\ubm, \mubold)}{\rbm(\ubm, \mubold)=\zerobold.}
\end{equation}
In the setting where the residual corresponds to a PDE discretization, this is a \textit{discrete PDE-constrained optimization problem} (discretize-then-optimize). We reduce this to an unconstrained problem by restricting the QoI to the solution manifold, i.e., define $f : \Dcal \rightarrow \Rbb$ such that
\begin{equation}
 f : \mubold \mapsto j(\ubm^\star(\mubold), \mubold),
\end{equation}
which leads to the following reduced space optimization problem
\begin{equation}\label{eqn:hdm:opt:reduc_space}
    \optunc{\mubold \in \Dcal}{f(\mubold)}.
\end{equation}
Because the solution of (\ref{eqn:hdm:opt:reduc_space}) requires potentially many queries to the large-scale HDM, it can be prohibitively expensive
to compute. The main contribution of this work is a procedure to accelerate this using model hyperreduction (Section~\ref{sec:hyperreduction}-\ref{sec:trammo}).

We use the adjoint method to compute the gradient of the QoI because the parameter-space dimension $N_\mubold$ is potentially large.
The HDM adjoint problem is: given $\mubold\in\Dcal$ and the corresponding primal solution $\ubm^\star\in\Rbb^{N_\ubm}$
satisfying (\ref{eqn:hdm:res_eqn}), find the adjoint solution $\lambdabold^\star\in\Rbb^{N_\ubm}$ satisfying
\begin{equation} \label{pde:eqn:res_sys_adj}
 \rlam(\lambdabold^\star, \ubm^\star, \mubold) = \zerobold,
\end{equation}
where the adjoint residual, $\rlam : \Rbb^{N_\ubm}\times\Rbb^{N_\ubm}\times \Dcal \rightarrow \Rbb^{N_\ubm}$, is defined as
\begin{equation} \label{eqn:hdm:res_adj}
 \rlam : (\lambdabold, \ubm, \mubold) \mapsto
 \pder{\rbm}{\ubm}(\ubm,\mubold)^T\lambdabold - \pder{j}{\ubm}(\ubm,\mubold)^T.
\end{equation}
Assuming the residual Jacobian is well-defined and invertible for every $(\ubm^\star(\mubold),\mubold)$ pair with $\mubold\in\Dcal$,
then there exists a unique adjoint solution
\begin{equation}
 \lambdabold^\star(\mubold) = \pder{\rbm}{\ubm}(\ubm^\star(\mubold), \mubold)^{-T}\pder{j}{\ubm}(\ubm^\star(\mubold),\mubold)^T.
\end{equation}
making the implicit map $\mubold\mapsto\lambdabold^\star(\mubold)$ well-defined for all $\mubold\in\Dcal$.
From the primal-adjoint pair $(\ubm^\star,\lambdabold^\star)$ satisfying (\ref{eqn:hdm:res_eqn}) and (\ref{pde:eqn:res_sys_adj}),
the gradient of the reduced QoI ($f$), $\func{\nabla f}{\Dcal}{\Rbb^{N_\mubold}}$, is computed as
\begin{equation}
 \nabla f : \mubold \mapsto \glam(\lambdabold^\star(\mubold), \ubm^\star(\mubold), \mubold)^T = \sens{j}(\ubm^\star(\mubold), \mubold)^T -  \sens{\rbm}(\ubm^\star(\mubold), \mubold)^T\lambdabold^\star(\mubold),
\end{equation}
where the operator that reconstructs the gradient from the adjoint solution,
$\glam : \Rbb^{N_\ubm}\times\Rbb^{N_\ubm}\times \Dcal \rightarrow \Rbb^{1\times N_\mubold}$, is
\begin{equation}
 \glam : (\lambdabold, \ubm, \mubold) \mapsto \sens{j}(\ubm, \mubold) - \lambdabold^T \sens{\rbm}(\ubm, \mubold).
\end{equation}
From the unassembled form of both the residual function (\ref{eqn:hdm:res_elem}) and QoI (\ref{eqn:hdm:qoi_elem}), the adjoint residual and QoI gradient function can be written in unassembled form as
\begin{equation} \label{eqn:hdm:res_qoi_adj_elem}
\begin{aligned}
\glam(\lambdabold, \ubm, \mubold) &= \sum^{N_\mathtt{e}}_{e=1} \bracket{
        \pder{j_e}{\mubold}(\ubm_e, \mubold) - \lambdabold_e^T \pder{\rbm_e}{\mubold}(\ubm_e, \ubm_e', \mubold)} \\
\rlam(\lambdabold, \ubm, \mubold)  &= \sum^{N_\mathtt{e}}_{e=1} \bracket{
        -\Pbm_e\pder{j_e}{\ubm_e}(\ubm_e,\mubold)^T
           + \Pbm_e\pder{\rbm_e}{\ubm_e}(\ubm_e, \ubm_e', \mubold)^T \lambdabold_e
            + \Pbm_e'\pder{\rbm_e}{\ubm_e'}(\ubm_e, \ubm_e',\mubold)^T \lambdabold_e
        },
\end{aligned}
\end{equation}
where $\lambdabold_e = \Pbm_e^T\lambdabold$.
The unassembled form will be used to construct optimization-informed hyperreduced models in the next section.

We close this section by introducing sensitivity analysis terms that mirror the adjoint terms above.
The HDM sensitivity problem is: given $\mubold\in\Dcal$ and the corresponding primal solution $\ubm^\star\in\Rbb^{N_\ubm}$
satisfying (\ref{eqn:hdm:res_eqn}), find the sensitivity solution $\partial_\mubold\ubm^\star\in\Rbb^{N_\ubm\times N_\mubold}$ satisfying
\begin{equation} \label{pde:eqn:res_sys_sens}
 \rbm^\partial(\partial_\mubold\ubm^\star, \ubm^\star, \mubold) = \zerobold,
\end{equation}
where the sensitivity residual, $ \rbm^\partial : \Rbb^{N_\ubm\times N_\mubold}\times\Rbb^{N_\ubm}\times \Dcal \rightarrow \Rbb^{N_\ubm\times N_\mubold}$, is defined as
\begin{equation} \label{pde:eqn:res_sens}
 \rbm^\partial : (\wbm, \ubm, \mubold) \mapsto
 \pder{\rbm}{\ubm}(\ubm,\mubold)\wbm + \pder{\rbm}{\mubold}(\ubm,\mubold).
\end{equation}
Assuming the residual Jacobian is well-defined and invertible for every $(\ubm^\star(\mubold),\mubold)$ pair with $\mubold\in\Dcal$,
then there exists a unique sensitivity solution
\begin{equation}
 \partial_\mubold\ubm^\star(\mubold) = -\pder{\rbm}{\ubm}(\ubm^\star(\mubold), \mubold)^{-1}\pder{\rbm}{\mubold}(\ubm^\star(\mubold),\mubold).
\end{equation}
making the implicit map $\mubold\mapsto\partial_\mubold\ubm^\star(\mubold)$ well-defined for all $\mubold\in\Dcal$.
From the unassembled form of the residual function (\ref{eqn:hdm:res_elem}), the sensitivity residual can be written in unassembled
form as
\begin{equation} \label{eqn:hdm:res_sens_elem}
\rbm^\partial(\wbm, \ubm, \mubold)  = \sum^{N_\mathtt{e}}_{e=1} \bracket{
       \Pbm_e\pder{\rbm_e}{\ubm_e}(\ubm_e, \ubm_e', \mubold) \wbm_e +
       \Pbm_e\pder{\rbm_e}{\ubm_e'}(\ubm_e, \ubm_e', \mubold) \wbm_e' +
       \Pbm_e\pder{\rbm_e}{\mubold}(\ubm_e, \ubm_e', \mubold),
        },
\end{equation}
where $\wbm_e = \Pbm_e^T\wbm$ and $\wbm_e' = (\Pbm_e')^T\wbm$.

\begin{remark}
There are two reasons sensitivity analysis is included in this work. First, the sensitivity method can be used in the place of the
adjoint method in the proposed framework to compute gradients if appropriate for a given application (i.e., few parameters). The
use of sensitivity information to construct reduced-order models has been shown to lead to more robust, predictive ROMs \cite{zahr_adaptive_2016},
although this approach is not practical when there are many parameters. Secondly, our numerical experiments (Section~\ref{sec:numexp})
found it is beneficial to include reduced sensitivity constraints in the EQP training when the adjoint method is used to compute
gradients.
\end{remark}

\subsection{Application: shape optimization}
\label{sec:pdeopt:shape_opt}
To close this section, we specialize the developments thus far to shape optimization as this will be an emphasis
of the numerical experiments (Section~\ref{sec:numexp}). Recall the PDE in (\ref{eqn:hdm:static_claw}) and
(\ref{eqn:hdm:static_vclaw}), and observe that in the discrete setting, the PDE domain $\Omega$ is approximated
using the computational mesh $\Ecal_h$, which we denote
$\Omega_h \coloneqq \cup_{e=1}^{N_\mathtt{e}} \Omega_e \approx \Omega$. Let $\xbm\in\Rbb^{N_x}$ denote the
coordinates of the nodes associated with the mesh $\Ecal_h$, which naturally parametrizes the discrete
domain as $\Omega_h = \Omega_h(\xbm)$. Then the algebraic PDE residual and QoI can be written explicitly
in terms of the PDE state and the nodal coordinates of the mesh. Let $\Rbm : \Rbb^{N_\ubm}\times\Rbb^{N_\xbm} \rightarrow \Rbb^{N_\ubm}$
denote the PDE residual and $J : \Rbb^{N_\ubm}\times\Rbb^{N_\xbm} \rightarrow \Rbb$ denote the QoI with
\begin{equation}
 \Rbm : (\ubm,\xbm) \mapsto \Rbm(\ubm,\xbm), \qquad
 J : (\ubm,\xbm) \mapsto J(\ubm,\xbm).
\end{equation}

In practice, it is rarely convenient to directly use these quantities to pose a shape optimization problem because
directly optimizing the nodal coordinates leads to many parameters and is difficult to enforce smoothness
of the resulting geometry.
In this work, we partition the nodal coordinates as $\xbm = (\xbm_\mathtt{o}, \xbm_\mathtt{c})$,
where $\xbm_\mathtt{o}\in\Rbb^{N_\xbm^\mathtt{o}}$ are the DoFs that will be optimized
and $\xbm_\mathtt{c}\in\Rbb^{N_\xbm^\mathtt{c}}$ are the constrained DoFs that will be determined via linear elasticity
(considering the mesh to be a pseudo-structure with displacement-driven deformation) to obtain
a well-conditioned mesh \cite{farhat1998torsional,lesoinne_linearized_2001}.
That is, we define the constrained DoFs in terms of the optimized DoFs as
\begin{equation}\label{eqn:hdm:shapeopt:mesh_coord}
 \xbm_\mathtt{c} = -\Kbm_{\mathtt{c}\mathtt{c}}^{-1}\Kbm_{\mathtt{c}\mathtt{o}} \xbm_\mathtt{o}
\end{equation}
where $\Kbm\in\Rbb^{N_\xbm\times N_\xbm}$ is the linear elasticity stiffness matrix and
$\Kbm_{\mathtt{c}\mathtt{c}}$, $\Kbm_{\mathtt{c}\mathtt{o}}$ indicate its restriction to the rows
corresponding to constrained DoFs and columns corresponding to the constrained and optimized
DoFs, respectively.

Next, we parametrize the nodal coordinates by introducing a mapping from
the parameter domain
\begin{equation}\label{eqn:hdm:shapeopt:param_mapping}
 \phibold : \Dcal \rightarrow \Rbb^{N_\xbm}, \qquad \mubold \mapsto \begin{bmatrix} \Ibm \\  -\Kbm_{\mathtt{c}\mathtt{c}}^{-1}\Kbm_{\mathtt{c}\mathtt{o}} \end{bmatrix} \phibold_\mathtt{o}(\mubold),
\end{equation}
where $\phibold_\mathtt{o} : \Dcal\rightarrow \Rbb^{N_\xbm^\mathtt{o}}$ directly parametrizes the optimized nodal DoFs.
Then, for a given $\mubold\in\Rbb^{N_\xbm}$, the corresponding nodal coordinates are defined as
$\xbm = \phibold(\mubold)$. In this work, the optimized DoFs correspond to the nodal positions on the surface whose shape is being
optimized and the mapping $\phibold_\mathtt{o}$ is a Bezier parametrization of the nodal coordinates.

From these definitions, we define the $\mubold$-parametrized PDE residual
($\rbm$) in (\ref{eqn:hdm:res_eqn}) and $\mubold$-parametrized QoI ($j$) as
\begin{equation}
\rbm(\ubm,\mubold) = \Rbm(\ubm, \phibold(\mubold)), \qquad
j(\ubm,\mubold) = J(\ubm, \phibold(\mubold)).
\end{equation}
With these definitions, the shape optimization problem exactly fits the form of the optimization (\ref{eqn:hdm:opt:full_space}).
The unassembled form of the residual and quantity of interest are
\begin{equation}
 \rbm(\ubm,\mubold) = \sum_{e=1}^{N_\mathtt{e}} \Rbm_e(\ubm_e,\ubm_e',\phibold_e(\mubold)), \qquad
 j(\ubm,\mubold) = \sum_{e=1}^{N_\mathtt{e}} J_e(\ubm_e,\phibold_e(\mubold)),
\end{equation}
where $\Rbm_e : \Rbb^{N_\ubm^e}\times\Rbb^{N_\ubm^{e'}}\times\Rbb^{N_\xbm^e}$ with
$\Rbm_e : (\ubm_e,\ubm_e',\xbm_e) \mapsto \Rbm_e(\ubm_e,\ubm_e',\xbm_e)$ is the residual
contribution of element $e$, $\ubm_e$ and $\ubm_e'$ are defined in Section~\ref{sec:pdeopt:gov_eqn},
$\xbm_e = \Qbm_e^T\xbm \in \Rbb^{N_\xbm^e}$ are the mesh coordinate DoFs associated with
element $e$, and $\phibold_e : \Dcal \rightarrow \Rbb^{N_\xbm^e}$ with
$\phibold_e : \mubold \mapsto \Qbm_e^T\phibold(\mubold)$ and $\Qbm_e\in\Rbb^{N_\xbm\times N_\xbm^e}$
is the assembly operator that maps element mesh coordinate DoFs to global mesh coordinate DoFs.
The adjoint residual, sensitivity residual, and gradient reconstruction are defined similarly with the
observation that derivatives of $\Rbm_e$ and $J_e$ with respect to $\mubold$ must use the chain
rule, e.g.,
\begin{equation}
  \pder{}{\mubold} \left[\Rbm_e(\ubm_e,\ubm_e',\phibold_e(\mubold))\right] =
  \pder{\Rbm_e}{\xbm_e}(\ubm_e,\ubm_e',\phibold_e(\mubold)) \pder{\phibold_e}{\mubold}(\mubold).
\end{equation}

\section{Hyperreduction}
\label{sec:hyperreduction}
In this section, we enhance the empirical quadrature procedure (EQP) developed by Yano \cite{yano_discontinuous_2019} with additional constraints required in the optimization setting that will be used to prove global convergence when embedded in a trust-region framework. To this end, we introduce standard projection-based model reduction (Section \ref{sec:hyperreduction:rom}), the adapted EQP method (Section \ref{sec:hyperreduction:lp_eqp}), and a technique to accelerate mesh motion in the shape optimization setting (Section \ref{sec:hyperreduction:shapeopt}).

\subsection{Projection-based model reduction}
\label{sec:hyperreduction:rom}
Projection-based model reduction begins with the ansatz that the primal state lies in a low-dimensional affine subspace
$\Vboldcal_\Phibold = \{\Phibold \hat\ybm \mid \hat\ybm\in\Rbb^n\} \subset \Rbb^{N_\ubm}$, where $\Phibold \in \Rbb^{N_\ubm \times n}$ with $n \ll N_\ubm$ is the reduced basis. That is, 
we approximate the primal state $\ubm^\star$ as
\begin{equation} \label{eqn:rom:ansatz}
    \ubm^\star \approx \hat\ubm_\Phibold^\star \coloneqq \Phibold \hat\ybm^\star,
\end{equation}
where $\hat\ubm_\Phibold^\star\in\Vboldcal_\Phibold$ is the subspace approximation of the primal
state $\ubm_\star$ and $\hat\ybm_\Phibold^\star \in \Rbb^n$ contains the corresponding reduced coordinates.
The construction of $\Phibold$ will be deferred to Section~\ref{sec:trammo:eqp_tr}.
The primal reduced coordinates are implicitly defined as the solution of the Galerkin reduced-order model:
given $\mubold\in\Dcal$, find $\hat\ybm_\Phibold^\star\in\Rbb^n$ such that
\begin{equation}\label{eqn:rom:res_eqn}
    \hrbm_\Phibold(\hat\ybm_\Phibold^\star, \mubold) = \zerobold,
\end{equation}
where $\hrbm_\Phibold : \Rbb^{n}\times\Dcal \rightarrow \Rbb^n$ with
\begin{equation}\label{eqn:rom:res}
    \hrbm_\Phibold : (\hat\ybm,\mubold) \mapsto \Phibold^T \rbm(\Phibold\hat\ybm, \mubold),
\end{equation}
which is obtained by substituting the ROM ansatz (\ref{eqn:rom:ansatz}) into the governing equation (\ref{eqn:hdm:res_eqn})
and projecting the resulting residual onto the columnspace of the reduced basis (Galerkin projection). We assume that for
every $\mubold\in\Dcal$, there exists a unique primal solution $\hat\ybm_\Phibold^\star=\hat\ybm_\Phibold^\star(\mubold)$
satisfying (\ref{eqn:rom:res_eqn}) and that the implicit map $\mubold \mapsto \hat\ybm_\Phibold^\star(\mubold)$ is continuously
differentiable. The reduced residual ($\hrbm_\Phibold$) inherits an unassembled structure from the HDM
\begin{equation} \label{eqn:rom:res_elem}
 \hrbm_\Phibold(\hat\ybm, \mubold) =
 \sum_{e=1}^{N_\mathtt{e}} \Phibold_e^T \rbm_\Phibold(\Phibold_e\hat\ybm, \Phibold_e'\hat\ybm, \mubold).
\end{equation}
where $\Phibold_e \coloneqq \Pbm_e^T\Phibold \in \Rbb^{N_\ubm^e \times n}$ and
$\Phibold_e' \coloneqq (\Pbm_e')^T\Phibold \in \Rbb^{N_\ubm^{e'} \times n}$ are the
reduced bases restricted to the elemental degrees of freedom.

The QoI is reduced by substituting the ROM approximation into the original QoI function. Let
$\hat{j}_\Phibold : \Rbb^n\times\Dcal \rightarrow \Rbb$ be the reduced quantity of interest
and $\func{\hat{f}_\Phibold}{\Dcal}{\Rbb}$ be the corresponding quantity restricted
to the ROM solution manifold, where
\begin{equation} \label{eqn:rom:qoi}
 \hat{j}_\Phibold : (\hat\ybm,\mubold) \mapsto j(\Phibold\hat\ybm,\mubold), \qquad
 \hat{f}_\Phibold : \mubold \mapsto \hat{j}_\Phibold(\hat\ybm_\Phibold^\star(\mubold),\mubold).
\end{equation}
The reduced QoI inherits an unassembled structure from the HDM
\begin{equation} \label{eqn:rom:qoi_elem}
 \hat{j}_\Phibold(\hat\ybm,\mubold) = \sum_{e=1}^{N_\mathtt{e}} j_e(\Phibold_e\hat\ybm,\mubold).
\end{equation}

The reduced adjoint residual, $\hrlam_\Phibold : \Rbb^n\times\Rbb^n\times\Dcal \rightarrow \Rbb^n$, is defined as
\begin{equation} \label{eqn:rom:res_adj}
 \hrlam_\Phibold : (\hat\zbm,\hat\ybm,\mubold) \mapsto \pder{\hat\rbm_\Phibold}{\hat\ybm}(\hat\ybm,\mubold)^T\hat\zbm - \pder{\hat{j}_\Phibold}{\hat\ybm}(\hat\ybm,\mubold)^T = \Phibold^T \rlam(\Phibold\hat\zbm, \Phibold\hat\ybm,\mubold),
\end{equation}
where the equality follows directly from the definitions in (\ref{eqn:hdm:res_adj}), (\ref{eqn:rom:res}), and (\ref{eqn:rom:qoi}).
The reduced adjoint problem is: given $\mubold\in\Dcal$ and the corresponding reduced primal solution
$\hat\ybm_\Phibold^\star$ satisfying (\ref{eqn:rom:res_eqn}), find the reduced adjoint solution $\hat\lambdabold_\Phibold^\star\in\Rbb^n$
satisfying
\begin{equation}
 \hrlam_\Phibold(\hat\lambdabold_\Phibold^\star, \hat\ybm_\Phibold^\star, \mubold) = \zerobold.
\end{equation}
Assuming the reduced residual Jacobian is well-defined and invertible for every $(\hat\ybm_\Phibold^\star(\mubold),\mubold)$
pair with $\mubold\in\Dcal$, there exists a unique reduced adjoint solution
\begin{equation}
 \hat\lambdabold_\Phibold^\star(\mubold) = \pder{\hat\rbm_\Phibold}{\hat\ybm}(\hat\ybm_\Phibold^\star(\mubold),\mubold)^{-T}\pder{\hat{j}_\Phibold}{\hat\ybm}(\hat\ybm_\Phibold^\star(\mubold),\mubold)^T,
\end{equation}
making the implicit map $\mubold\mapsto\hat\lambdabold_\Phibold^\star(\mubold)$ well-defined for all $\mubold\in\Dcal$.
The reduced adjoint residual inherits an unassembled structure from the HDM using
the relationships in (\ref{eqn:hdm:res_qoi_adj_elem}) and (\ref{eqn:rom:res_adj})
\begin{equation} \label{eqn:rom:res_adj_elem}
  \hrlam_\Phibold(\hat\zbm,\hat\ybm,\mubold) =
   \sum^{N_\mathtt{e}}_{e=1} \bracket{
        -\Phibold_e^T\pder{j_e}{\ubm_e}(\Phibold_e\hat\ybm,\mubold)^T
           + \Phibold_e^T\pder{\rbm_e}{\ubm_e}(\Phibold_e\hat\ybm, \Phibold_e'\hat\ybm, \mubold)^T \Phibold_e\hat\zbm
            + (\Phibold_e')^T\pder{\rbm_e}{\ubm_e'}(\Phibold_e\hat\ybm, \Phibold_e'\hat\ybm,\mubold)^T \Phibold_e\hat\zbm
        }.
\end{equation}
From the primal-adjoint pair ($\hat\ybm_\Phibold^\star,\hat\lambdabold_\Phibold^\star$), the gradient of the reduced QoI ($\hat{f}_\Phibold$),
$\nabla \hat{f}_\Phibold : \Dcal \rightarrow {\Rbb^{N_\mubold}}$, is computed as
\begin{equation}
 \nabla \hat{f}_\Phibold : \mubold \mapsto \hat\gbm_\Phibold^\lambda(\hat\lambdabold_\Phibold^\star(\mubold), \hat\ybm_\Phibold^\star(\mubold), \mubold)^T =
 \pder{\hat{j}_\Phibold}{\mubold}(\hat\ybm_\Phibold^\star(\mubold),\mubold)^T - \pder{\hat\rbm_\Phibold}{\mubold}(\hat\ybm_\Phibold^\star(\mubold),\mubold)^T\hat\lambdabold_\Phibold^\star(\mubold),
\end{equation}
where the operator that reconstructs the reduced QoI gradient from the reduced adjoint solution,
$\hat\gbm_\Phibold^\lambda : \Rbb^n\times\Rbb^n\times\Dcal \rightarrow \Rbb^{1\times N_\mubold}$, is
\begin{equation}
 \hat\gbm_\Phibold^\lambda : (\hat\zbm, \hat\ybm, \mubold) \mapsto
 \pder{\hat{j}_\Phibold}{\mubold}(\hat\ybm,\mubold) - \hat\zbm^T\pder{\hat\rbm_\Phibold}{\mubold}(\hat\ybm,\mubold) =
 \gbm^\lambdabold(\Phibold\hat\zbm, \Phibold\hat\ybm,\mubold).
\end{equation}
The reduced gradient operator also inherits an unassembled structure from the HDM
\begin{equation} \label{eqn:rom:qoi_grad_elem}
\hat\gbm_\Phibold^\lambda(\hat\zbm, \hat\ybm, \mubold) = \sum^{N_\mathtt{e}}_{e=1} \bracket{
        \pder{j_e}{\mubold}(\Phibold_e\hat\ybm, \mubold) - \hat\zbm^T\Phibold_e^T \pder{\rbm_e}{\mubold}(\Phibold_e\hat\ybm, \Phibold_e'\ybm, \mubold)}.
\end{equation}

The reduced sensitivity residual, $\hat\rbm^\partial_\Phibold : \Rbb^{n\times N_\mubold}\times\Rbb^n\times\Dcal \rightarrow \Rbb^{n\times N_\mubold}$, is defined as
\begin{equation} \label{eqn:rom:res_sens}
 \hat\rbm^\partial_\Phibold : (\hat\wbm,\hat\ybm,\mubold) \mapsto \pder{\hat\rbm_\Phibold}{\hat\ybm}(\hat\ybm,\mubold)\hat\wbm + \pder{\hat{\rbm}_\Phibold}{\mubold}(\hat\ybm,\mubold) = \Phibold^T \hat\rbm^\partial(\Phibold\hat\wbm, \Phibold\hat\ybm,\mubold),
\end{equation}
where the equality follows directly from the definitions in (\ref{pde:eqn:res_sens}) and (\ref{eqn:rom:res}).
The reduced sensitivity problem is: given $\mubold\in\Dcal$ and the corresponding reduced primal solution
$\hat\ybm_\Phibold^\star$ satisfying (\ref{eqn:rom:res_eqn}), find the reduced sensitivity solution $\partial_\mubold\hat\ybm_\Phibold^\star\in\Rbb^{n\times N_\mubold}$
satisfying
\begin{equation}
 \rbm^\partial_\Phibold(\partial_\mubold\hat\ybm_\Phibold^\star, \hat\ybm_\Phibold^\star, \mubold) = \zerobold.
\end{equation}
Assuming the reduced residual Jacobian is well-defined and invertible for every $(\hat\ybm_\Phibold^\star(\mubold),\mubold)$
pair with $\mubold\in\Dcal$, there exists a unique reduced sensitivity solution
\begin{equation}
 \partial_\mubold\hat\ybm_\Phibold^\star(\mubold) = -\pder{\hat\rbm_\Phibold}{\hat\ybm}(\hat\ybm_\Phibold^\star(\mubold),\mubold)^{-1}\pder{\hat\rbm_\Phibold}{\mubold}(\hat\ybm_\Phibold^\star(\mubold),\mubold),
\end{equation}
making the implicit map $\mubold\mapsto\partial_\mubold\hat\ybm_\Phibold^\star(\mubold)$ well-defined for all $\mubold\in\Dcal$.
The reduced sensitivity residual inherits an unassembled structure from the HDM using
the relationships in (\ref{eqn:hdm:res_sens_elem}) and (\ref{eqn:rom:res_sens})
\begin{equation} \label{eqn:rom:sens_adj_elem}
  \hat\rbm^\partial_\Phibold(\hat\wbm,\hat\ybm,\mubold) =
   \sum^{N_\mathtt{e}}_{e=1} \bracket{
       \Phibold_e^T\pder{\rbm_e}{\ubm_e}(\Phibold_e\hat\ybm, \Phibold_e'\hat\ybm, \mubold) \Phibold_e\hat\wbm +
       \Phibold_e^T\pder{\rbm_e}{\ubm_e'}(\Phibold_e\hat\ybm, \Phibold_e'\hat\ybm, \mubold) \Phibold_e'\hat\wbm +
       \Phibold_e^T\pder{\rbm_e}{\mubold}(\Phibold_e\hat\ybm, \Phibold_e'\hat\ybm, \mubold)
     }.
\end{equation}

Despite the potentially significant reduction in degrees of freedom between the HDM system (\ref{eqn:hdm:res_eqn}) and the
reduced system (\ref{eqn:rom:res_eqn}), computational efficiency will not necessarily be achieved due to the cost of
constructing the nonlinear terms. This can be clearly seen from the unassembled form of the primal residual
(\ref{eqn:rom:res_elem}), QoI (\ref{eqn:rom:qoi_elem}), adjoint residual (\ref{eqn:rom:res_adj}), sensitivity residual (\ref{eqn:rom:res_sens}), and QoI gradient (\ref{eqn:rom:qoi_grad_elem}), namely, each
of these operations requires elemental operations \textit{for all elements}. For systems comprised of many elements,
e.g., high-fidelity discretizations of PDEs, this can be a serious bottleneck. To accelerate formation of the nonlinear
terms, we turn to EQP hyperreduction.

\begin{remark} \label{rem:ptc_vs_newt}
One source of computational efficiency in the model reduction setting is the reduced-order model in
(\ref{eqn:rom:res_eqn}) can usually be solved directly using Newton's method initialized from the snapshots used to
train the basis $\Phibold$. In the HDM setting, Newton's method can rarely be used directly for
fluid applications; usually, some form of continuation is required for robust convergence, which
can require a significant number of iterations. This difference will be leveraged and discussed
further in Section~\ref{sec:numexp}.
\end{remark}

\subsection{An optimization-aware empirical quadrature procedure}
\label{sec:hyperreduction:lp_eqp}
To accelerate the assembly of the nonlinear terms in the various reduced quantities introduced in
Section~\ref{sec:hyperreduction:rom}, we use the empirical quadrature procedure \cite{yano_lp_2019, yano_discontinuous_2019}. In the adjoint-based
optimization setting, we use EQP to accelerate any operation that involves assembly over all elements (Section~\ref{sec:hyperreduction:eqp_form}),
i.e., evaluation of the primal and adjoint residuals, quantity of interest, and gradient reconstruction.
To ensure all hyperreduced quantities are accurate with respect to their reduced counterparts, we include additional
constraints on the original EQP linear program introduced in \cite{yano_discontinuous_2019} (Section~\ref{sec:hyperreduction:eqp_training}).

\subsubsection{Formulation}
\label{sec:hyperreduction:eqp_form}
The EQP construction replaces the unassembled form of the reduced primal residual with a weighted (hyperreduced) version,
$\tilde\rbm_\Phibold : \Rbb^n \times \Dcal \times \Rcal$, where
\begin{equation}\label{eqn:eqp:res_elem}
 \tilde\rbm_\Phibold : (\tilde\ybm, \mubold; \rhobold) \mapsto \sum_{e=1}^{N_\mathtt{e}} \rho_e\Phibold_e^T \rbm_e(\Phibold_e\tilde\ybm, \Phibold_e'\tilde\ybm,\mubold)
\end{equation}
and $\rhobold\in\Rcal$ is the vector of weights with $\Rcal\subset\Rbb^{N_\mathtt{e}}$ the set of admissible weights
such that $\onebold\in\Rcal$ ($\onebold$ is the vector with each entry equal to one). For each element $\Omega_e\in\Ecal_h$ with $\rho_e=0$, the operations on the element can be completely skipped so computational efficiency is achieved when the vector of weights is highly sparse; construction of $\rhobold$ is deferred to Section~\ref{sec:hyperreduction:eqp_training} and \ref{sec:hyperreduction:eqp_err_est}. The primal EQP problem reads: given $\mubold\in\Dcal$ and $\rhobold\in\Rcal$, find $\tilde\ybm_\Phibold^\star\in\Rbb^n$ such that
\begin{equation}\label{eqn:eqp:res_eqn}
 \tilde\rbm_\Phibold(\tilde\ybm_\Phibold^\star,\mubold;\rhobold) = \zerobold.
\end{equation}
For each $\rhobold\in\Rcal$, we assume there is a unique primal solution
$\tilde\ybm_\Phibold^\star = \tilde\ybm_\Phibold^\star(\mubold;\rhobold)$ satisfying (\ref{eqn:eqp:res_eqn}) for every $\mubold\in\Dcal$.
\begin{remark}\label{rem:eqp:recon_rom}
The standard and weighted reduced primal residual are related as
\begin{equation}
\hat\rbm_\Phibold(\hat\ybm,\mubold) = \tilde\rbm_\Phibold(\hat\ybm,\mubold;\onebold)
\end{equation}
for any $\hat\ybm\in\Rbb^n$ and $\mubold\in\Dcal$. Therefore, the corresponding reduced coordinates are equal
\begin{equation}
 \hat\ybm_\Phibold^\star(\mubold) = \tilde\ybm_\Phibold^\star(\mubold; \onebold).
\end{equation}
\end{remark}
\begin{remark}
Uniqueness of the primal solution will not hold for all $\rhobold\in\Rcal$ because
$\tilde\rbm_\Phibold(\,\cdot\,; \,\cdot\,, \zerobold)=\zerobold$, which means any $\tilde\ybm\in\Rbb^n$ satisfies (\ref{eqn:eqp:res_eqn}).
Such cases are not encountered in practice because $\rhobold$ is not chosen randomly, rather it is
the solution of a linear program (Section~\ref{sec:hyperreduction:eqp_training}) that promotes sparsity while retaining accuracy with respect to $\hat\rbm_\Phibold$.
\end{remark}

In the optimization setting, computational efficiency is also required for the evaluation of the QoI, the adjoint residual
(similar to \cite{yano2020goal,du_adaptive_2021}), and the gradient reconstruction. We use the same approach of introducing weights
($\rhobold$) into the unassembled form; efficiency is achieved when $\rhobold$ is sparse.
The hyperreduced QoI, $\tilde{j}_\Phibold : \Rbb^n\times\Dcal\times\Rcal \rightarrow \Rbb$, and
its restriction to the EQP solution manifold, $\tilde{f}_\Phibold : \Dcal\times\Rcal \rightarrow \Rbb$,
are defined as
\begin{equation}\label{eqn:eqp:qoi_elem}
 \tilde{j}_\Phibold : (\tilde\ybm, \mubold; \rhobold) \mapsto \sum_{e=1}^{N_\mathtt{e}} \rho_e j_e(\Phibold_e\tilde\ybm,\mubold), \qquad
 \tilde{f}_\Phibold : (\mubold; \rhobold) \mapsto \tilde{j}_\Phibold(\tilde\ybm_\Phibold^\star(\mubold;\rhobold),\mubold;\rhobold).
\end{equation}
The hyperreduced adjoint residual, $\trlam_\Phibold : \Rbb^n\times\Rbb^n\times\Dcal\times\Rcal \rightarrow \Rbb^n$, is defined as
\begin{equation} \label{eqn:eqp:res_adj}
 \trlam_\Phibold : (\tilde\zbm,\tilde\ybm,\mubold;\rhobold) \mapsto \pder{\tilde\rbm_\Phibold}{\tilde\ybm}(\tilde\ybm,\mubold;\rhobold)^T\tilde\zbm - \pder{\tilde{j}_\Phibold}{\tilde\ybm}(\tilde\ybm,\mubold;\rhobold)^T,
\end{equation}
or, in unassembled form,
\begin{equation} \label{eqn:eqp:res_adj_elem}
  \trlam_\Phibold(\tilde\zbm,\tilde\ybm,\mubold;\rhobold) =
   \sum^{N_\mathtt{e}}_{e=1} \rho_e\bracket{
        -\Phibold_e^T\pder{j_e}{\ubm_e}(\Phibold_e\tilde\ybm,\mubold)^T
           + \Phibold_e^T\pder{\rbm_e}{\ubm_e}(\Phibold_e\tilde\ybm, \Phibold_e'\tilde\ybm, \mubold)^T \Phibold_e\tilde\zbm
            + (\Phibold_e')^T\pder{\rbm_e}{\ubm_e'}(\Phibold_e\tilde\ybm, \Phibold_e'\tilde\ybm,\mubold)^T \Phibold_e\tilde\zbm
        }.
\end{equation}
The hyperreduced adjoint problem is: given $\mubold\in\Dcal$ and the corresponding hyperreduced primal solution
$\tilde\ybm_\Phibold^\star$ satisfying (\ref{eqn:eqp:res_eqn}), find the hyperreduced adjoint solution $\tilde\lambdabold_\Phibold^\star\in\Rbb^n$ satisfying
\begin{equation}
 \trlam_\Phibold(\tilde\lambdabold_\Phibold^\star, \tilde\ybm_\Phibold^\star, \mubold; \rhobold) = \zerobold.
\end{equation}
For each $\rhobold\in\Rcal$, we assume the hyperreduced Jacobian is well-defined and invertible for every $(\tilde\ybm_\Phibold^\star(\mubold;\rhobold),\mubold)$ pair with $\mubold\in\Dcal$ so there exists a unique hyperreduced adjoint solution
\begin{equation}
 \tilde\lambdabold_\Phibold^\star(\mubold;\rhobold) = \pder{\tilde\rbm_\Phibold}{\tilde\ybm}(\tilde\ybm_\Phibold^\star(\mubold;\rhobold),\mubold;\rhobold)^{-T}\pder{\tilde{j}_\Phibold}{\tilde\ybm}(\tilde\ybm_\Phibold^\star(\mubold;\rhobold),\mubold;\rhobold)^T,
\end{equation}
making the implicit map $(\mubold;\rhobold)\mapsto\tilde\lambdabold_\Phibold^\star(\mubold;\rhobold)$ well-defined.
From the primal-adjoint pair ($\tilde\ybm_\Phibold^\star,\tilde\lambdabold_\Phibold^\star$), the gradient of the reduced QoI ($\tilde{f}_\Phibold$),
$\nabla \tilde{f}_\Phibold : \Dcal\times\Rcal \rightarrow \Rbb^{N_\mubold}$, is computed as
\begin{equation}
 \nabla \tilde{f}_\Phibold : (\mubold;\rhobold) \mapsto \tilde\gbm_\Phibold^\lambda(\tilde\ybm_\Phibold^\star(\mubold;\rhobold),\tilde\lambdabold_\Phibold^\star(\mubold;\rhobold), \mubold; \rhobold)^T =
 \pder{\tilde{j}_\Phibold}{\mubold}(\tilde\ybm_\Phibold^\star(\mubold;\rhobold),\mubold;\rhobold)^T - \pder{\tilde\rbm_\Phibold}{\mubold}(\tilde\ybm_\Phibold^\star(\mubold;\rhobold),\mubold;\rhobold)^T\tilde\lambdabold_\Phibold^\star(\mubold;\rhobold),
\end{equation}
where the operator that reconstructs the hyperreduced QoI gradient from the hyperreduced adjoint solution,
$\tilde\gbm_\Phibold^\lambda : \Rbb^n\times\Rbb^n\times\Dcal\times\Rcal \rightarrow \Rbb^{1\times N_\mubold}$, is
\begin{equation}
 \tilde\gbm_\Phibold^\lambda : (\tilde\zbm, \tilde\ybm,\mubold;\rhobold) \mapsto
 \pder{\tilde{j}_\Phibold}{\mubold}(\tilde\ybm,\mubold;\rhobold) - \tilde\zbm^T\pder{\tilde\rbm_\Phibold}{\mubold}(\tilde\ybm,\mubold; \rhobold),
\end{equation}
or, in unassembled form,
\begin{equation}
\tilde\gbm_\Phibold^\lambda(\tilde\zbm, \tilde\ybm, \mubold; \rhobold) = \sum^{N_\mathtt{e}}_{e=1} \rho_e\bracket{
        \pder{j_e}{\mubold}(\Phibold_e\tilde\ybm, \mubold) - \tilde\zbm^T\Phibold_e^T \pder{\rbm_e}{\mubold}(\Phibold_e\tilde\ybm, \Phibold_e'\tilde\ybm, \mubold)}.
\end{equation}

The hyperreduced sensitivity residual, $\tilde\rbm^\partial_\Phibold : \Rbb^{n\times N_\mubold}\times\Rbb^n\times\Dcal\times\Rcal \rightarrow \Rbb^{n\times N_\mubold}$, is defined as
\begin{equation} \label{eqn:eqp:res_sens}
 \tilde\rbm^\partial_\Phibold : (\tilde\wbm,\tilde\ybm,\mubold;\rhobold) \mapsto \pder{\tilde\rbm_\Phibold}{\tilde\ybm}(\tilde\ybm,\mubold;\rhobold)\tilde\wbm + \pder{\tilde\rbm_\Phibold}{\mubold}(\tilde\ybm,\mubold;\rhobold),
\end{equation}
or, in unassembled form,
\begin{equation} \label{eqn:eqp:res_sens_elem}
  \tilde\rbm^\partial_\Phibold(\tilde\wbm,\tilde\ybm,\mubold;\rhobold) =
   \sum^{N_\mathtt{e}}_{e=1} \rho_e \bracket{
       \Phibold_e^T\pder{\rbm_e}{\ubm_e}(\Phibold_e\tilde\ybm, \Phibold_e'\tilde\ybm, \mubold) \Phibold_e\tilde\wbm +
       \Phibold_e^T\pder{\rbm_e}{\ubm_e'}(\Phibold_e\tilde\ybm, \Phibold_e'\tilde\ybm, \mubold) \Phibold_e'\tilde\wbm +
       \Phibold_e^T\pder{\rbm_e}{\mubold}(\Phibold_e\tilde\ybm, \Phibold_e'\tilde\ybm, \mubold)
     }.
\end{equation}
The hyperreduced sensitivity problem is: given $\mubold\in\Dcal$ and the corresponding hyperreduced primal solution $\tilde\ybm_\Phibold^\star$ satisfying (\ref{eqn:eqp:res_eqn}), find the hyperreduced sensitivity solution $\partial_\mubold\tilde\ybm_\Phibold^\star\in\Rbb^{n\times N_\mubold}$ satisfying
\begin{equation}
 \tilde\rbm^\partial_\Phibold(\partial_\mubold\tilde\ybm_\Phibold^\star, \tilde\ybm_\Phibold^\star, \mubold; \rhobold) = \zerobold.
\end{equation}
For each $\rhobold\in\Rcal$, we assume the hyperreduced Jacobian is well-defined and invertible for every $(\tilde\ybm_\Phibold^\star(\mubold;\rhobold),\mubold)$ pair with $\mubold\in\Dcal$ so there exists a unique hyperreduced sensitivity solution
\begin{equation}
 \partial_\mubold\tilde\ybm_\Phibold^\star(\mubold;\rhobold) = -\pder{\tilde\rbm_\Phibold}{\tilde\ybm}(\tilde\ybm_\Phibold^\star(\mubold;\rhobold),\mubold;\rhobold)^{-1}\pder{\tilde\rbm_\Phibold}{\mubold}(\tilde\ybm_\Phibold^\star(\mubold;\rhobold),\mubold;\rhobold),
\end{equation}
making the implicit map $(\mubold;\rhobold)\mapsto\partial_\mubold\ybm_\Phibold^\star(\mubold;\rhobold)$ well-defined.

\begin{remark}\label{rem:eqp:recon_other_rom}
Similar to the primal residual, the hyperreduced form of the other quantities agree with the reduced quantities when
$\rhobold=\onebold$, i.e.,
\begin{equation}
\begin{aligned}
\hat{j}_\Phibold(\hat\ybm,\mubold) = \tilde{j}_\Phibold(\hat\ybm,\mubold;\onebold), \\
\hat\rbm_\Phibold^\lambda(\hat\zbm,\hat\ybm,\mubold) = \tilde\rbm_\Phibold^\lambda(\hat\zbm,\hat\ybm,\mubold;\onebold),  \\
\hat\gbm_\Phibold^\lambda(\hat\zbm,\hat\ybm,\mubold) = \tilde\gbm_\Phibold^\lambda(\hat\zbm,\hat\ybm,\mubold;\onebold), \\
\hat\rbm_\Phibold^\partial(\hat\wbm,\hat\ybm,\mubold) = \tilde\rbm_\Phibold^\partial(\hat\wbm,\hat\ybm,\mubold;\onebold),  \\
\end{aligned}
\end{equation}
for any $\hat\ybm,\hat\zbm\in\Rbb^n$, $\hat\wbm\in\Rbb^{n\times N_\mubold}$, and $\mubold\in\Dcal$.
\end{remark}

\subsubsection{Training}
\label{sec:hyperreduction:eqp_training}
The success of EQP is inherently linked to the construction of a sparse weight vector that ensures the hyperreduced
quantities introduced in Section~\ref{sec:hyperreduction:eqp_form} accurately approximate the corresponding reduced quantity in Section~\ref{sec:hyperreduction:rom}.
In \cite{yano_lp_2019,yano_discontinuous_2019}, the EQP weights are chosen to be the solution of an $\ell_1$ minimization problem (to promote sparsity) that
includes several constraints, most important of which is the \textit{manifold accuracy} constraint that requires the
reduced ($\hat\rbm$) and hyperreduced ($\tilde\rbm$) primal residuals are sufficiently close on some training set.
Manifold constraints on the quantity of interest and adjoint residual, among others, are included when EQP is
used to accelerate dual-weighted residual error estimation \cite{yano2020goal,du_adaptive_2021}. In the adjoint-based optimization
setting, we require the hyperreduced primal residual, adjoint residual, quantity of interest, and gradient reconstruction operator
accurately approximate the corresponding reduced quantity.

To this end, we let $\Phibold$ be a given reduced basis and $\Xibold\subset \Dcal$ be a collection
of EQP training parameters, and define the EQP weights, $\rhobold^\star$,
as the solution of the following linear program
\begin{equation}\label{eqn:eqp:linprog}
 \rhobold^\star = \argoptunc{\rhobold\in\Ccal^\mathtt{nn}\cap\Ccal_{\Phibold,\Xibold,\deltabold}}{\sum_{e=1}^{N_\mathtt{e}} \rho_e},
\end{equation}
where $\Ccal^\mathtt{nn}$ is the set of nonnegative weights
\begin{equation}
 \Ccal^\mathtt{nn} \coloneqq \left\{ \rhobold\in\Rbb^{N_\mathtt{e}} \suchthat \rho_e \geq 0, ~ e = 1,\dots,N_\mathtt{e} \right\},
\end{equation}
$ \deltabold = (\delta_\mathtt{dv},\delta_\mathtt{rp},\delta_\mathtt{ra},\delta_\mathtt{ga},\delta_\mathtt{q},\delta_\mathtt{rs}) \in\Rbb^6$ is a collection of tolerances, and $\Ccal_{\Phibold,\Xibold,\deltabold}\subset\Rbb^{N_\mathtt{e}}$ is the intersection of some subset
of the following accuracy constraints\footnote{Superscript legend: $\mathtt{nn}$ = \underline{n}on\underline{n}egativity, $\mathtt{dv}$ = \underline{d}omain \underline{v}olume, $\mathtt{rp}$ = \underline{p}rimal \underline{r}esidual, $\mathtt{ra}$ = \underline{a}djoint \underline{r}esidual, $\mathtt{ga}$ = \underline{a}djoint \underline{g}radient reconstruction, $\mathtt{q}$ = \underline{q}uantity of interest, $\mathtt{rs}$ = \underline{s}ensitivity \underline{r}esidual.}
\begin{equation}\label{eqn:eqp:rescon}
\begin{aligned}
 \Ccal_{\delta}^\mathtt{dv} &\coloneqq \left\{ \rhobold\in\Rbb^{N_\mathtt{e}}\suchthat \left| |\Omega| - \sum_{e=1}^{N_\mathtt{e}} \rho_e|\Omega_e|\right| \leq \delta\right\}, \\
 \Ccal_{\Phibold,\Xibold,\delta}^\mathtt{rp} &\coloneqq \left\{ \rhobold\in\Rbb^{N_\mathtt{e}}\suchthat  \norm{\hat\rbm_\Phibold(\hat\ybm_\Phibold^\star(\mubold),\mubold)-\tilde\rbm_\Phibold(\hat\ybm_\Phibold^\star(\mubold),\mubold;\rhobold)}_\infty \leq \delta, \forall \mubold\in\Xibold\right\}, \\
 \Ccal_{\Phibold,\Xibold,\delta}^\mathtt{ra} &\coloneqq \left\{ \rhobold\in\Rbb^{N_\mathtt{e}} \suchthat  \norm{\hat\rbm_\Phibold^\lambda(\hat\lambdabold_\Phibold^\star(\mubold),\hat\ybm_\Phibold^\star(\mubold),\mubold)-\tilde\rbm_\Phibold^\lambda(\hat\lambdabold_\Phibold^\star(\mubold),\hat\ybm_\Phibold^\star(\mubold),\mubold;\rhobold)}_\infty \leq \delta, \forall \mubold\in\Xibold\right\}, \\
 \Ccal_{\Phibold,\Xibold,\delta}^\mathtt{ga} &\coloneqq \left\{ \rhobold\in\Rbb^{N_\mathtt{e}}\suchthat  \norm{\hat\gbm_\Phibold^\lambda(\hat\lambdabold_\Phibold^\star(\mubold),\hat\ybm_\Phibold^\star(\mubold),\mubold)-\tilde\gbm_\Phibold^\lambda(\hat\lambdabold_\Phibold^\star(\mubold),\hat\ybm_\Phibold^\star(\mubold),\mubold;\rhobold)}_\infty \leq \delta, \forall \mubold\in\Xibold\right\}, \\
 \Ccal_{\Phibold,\Xibold,\delta}^\mathtt{q} &\coloneqq \left\{ \rhobold\in\Rbb^{N_\mathtt{e}}\suchthat  \norm{\hat{j}_\Phibold(\hat\ybm_\Phibold^\star(\mubold),\mubold)-\tilde{j}_\Phibold(\hat\ybm_\Phibold^\star(\mubold),\mubold;\rhobold)}_\infty \leq \delta, \forall \mubold\in\Xibold\right\}, \\
 \Ccal_{\Phibold,\Xibold,\delta}^\mathtt{rs} &\coloneqq \left\{ \rhobold\in\Rbb^{N_\mathtt{e}} \suchthat  \norm{\hat\rbm_\Phibold^\partial(\partial_\mubold\hat\ybm_\Phibold^\star(\mubold),\hat\ybm_\Phibold^\star(\mubold),\mubold)-\tilde\rbm_\Phibold^\partial(\partial_\mubold\hat\ybm_\Phibold^\star(\mubold),\hat\ybm_\Phibold^\star(\mubold),\mubold;\rhobold)}_\infty \leq \delta, \forall \mubold\in\Xibold\right\},
\end{aligned}
\end{equation}
where $|\Scal|$ is the volume of the subset $\Scal\subset\Xcal$ of a metric space $\Xcal$.
Each constraint in (\ref{eqn:eqp:rescon}) ensures a selected reduced quantities introduced in
Section~\ref{sec:hyperreduction:rom} is sufficiently approximated by the corresponding hyperreduced quantity in Section~\ref{sec:hyperreduction:eqp_form}.
Finally, we define $\rhobold^\star : (\Phibold,\Xibold,\deltabold) \mapsto \rhobold^\star(\Phibold,\Xibold,\deltabold)$ as
the implicit map from a given reduced basis $\Phibold$, EQP training set $\Xibold$, and tolerances $\deltabold$
to the solution of the linear program in (\ref{eqn:eqp:linprog}). For now, we leave the tolerances and EQP training set unspecified; in Section~\ref{sec:trammo:eqp_tr}, these will be chosen to guarantee global convergence of the trust-region method to a local minimum of the unreduced optimization problem in (\ref{eqn:hdm:opt:full_space}).

\begin{remark} \label{rem:eqp:rescon_simple}
The residual-based constraints can be simplified to
\begin{equation}\label{eqn:eqp:rescon_simple}
\begin{aligned}
 \Ccal_{\Phibold,\Xibold,\delta}^\mathtt{rp} &\coloneqq \left\{ \rhobold\in\Rbb^{N_\mathtt{e}}\suchthat  \norm{\tilde\rbm_\Phibold(\hat\ybm_\Phibold^\star(\mubold),\mubold;\rhobold)}_\infty \leq \delta, \forall \mubold\in\Xibold\right\}, \\
 \Ccal_{\Phibold,\Xibold,\delta}^\mathtt{ra} &\coloneqq \left\{ \rhobold\in\Rbb^{N_\mathtt{e}} \suchthat  \norm{\tilde\rbm_\Phibold^\lambda(\hat\lambdabold_\Phibold^\star(\mubold),\hat\ybm_\Phibold^\star(\mubold),\mubold;\rhobold)}_\infty \leq \delta, \forall \mubold\in\Xibold\right\}, \\
 \Ccal_{\Phibold,\Xibold,\delta}^\mathtt{rs} &\coloneqq \left\{ \rhobold\in\Rbb^{N_\mathtt{e}} \suchthat  \norm{\tilde\rbm_\Phibold^\partial(\partial_\mubold\hat\ybm_\Phibold^\star(\mubold),\hat\ybm_\Phibold^\star(\mubold),\mubold;\rhobold)}_\infty \leq \delta, \forall \mubold\in\Xibold\right\},
\end{aligned}
\end{equation}
by definition of $\hat\ybm_\Phibold^\star$, $\hat\lambdabold_\Phibold^\star$, and $\partial_\mubold\hat\ybm_\Phibold^\star$, where
clearly $\zerobold \in  \Ccal_{\Phibold,\Xibold,\delta}^\mathtt{rp} \cap  \Ccal_{\Phibold,\Xibold,\delta}^\mathtt{ra} \cap  \Ccal_{\Phibold,\Xibold,\delta}^\mathtt{rs}$.
\end{remark}

\begin{remark} \label{rem:eqp:props}
The optimization problem in (\ref{eqn:eqp:linprog}) is guaranteed to have a feasible solution (regardless of which subset of the constraints is used)
because all hyperreduced quantities are equivalent to the corresponding reduced quantity when $\rhobold=\onebold$
(Remarks~\ref{rem:eqp:recon_rom}, \ref{rem:eqp:recon_other_rom}). Furthermore, the optimization problem in (\ref{eqn:eqp:linprog}) is a linear program because each hyperreduced term
is linear in $\rhobold$ and the infinity-norm bounds can be recast as a collection of inequality constraints (upper and lower bounds)
on its argument. Lastly, despite Remark~\ref{rem:eqp:rescon_simple}, the solution of (\ref{eqn:eqp:linprog}) is nonzero, i.e.,
$\rhobold^\star(\Phibold,\Xibold,\deltabold)\neq\zerobold$, provided at least one non-residual constraint
($\Ccal_\delta^\mathtt{dv}$, $\Ccal_{\Phibold,\Xibold,\delta}^\mathtt{ga}$,
$\Ccal_{\Phibold,\Xibold,\delta}^\mathtt{q}$) is included.
\end{remark}

\begin{remark}
The nonnegativity constraint $\Ccal^\mathtt{nn}$ and domain volume constraint $\Ccal_\delta^\mathtt{dv}$ are
included to maintain the interpretation of the weight vector as a sparse quadrature rule in the PDE setting \cite{yano_discontinuous_2019},
but not directly required for global convergence of the trust-region method in Section~\ref{sec:trammo:eqp_tr}. The domain constraint
$\Ccal_\delta^\mathtt{dv}$ is also included to ensure a nonzero weight vector (Remark~\ref{rem:eqp:props}).
\end{remark}

\begin{remark}
The constraints on the quantity of interest $\Ccal_{\Phibold,\Xibold,\delta}^\mathtt{q}$ and sensitivity residual
$\Ccal_{\Phibold,\Xibold,\delta}^\mathtt{rs}$ are not required for global convergence of the trust-region method
in Section~\ref{sec:trammo:eqp_tr}, but will be shown empirically (Section~\ref{sec:numexp}) to improve the convergence rate of the
method.
\end{remark}

\begin{remark}
The constraint $\Ccal_{\Phibold,\Xibold,\delta}^\mathtt{rp}$ is the manifold accuracy constraint in \cite{yano_discontinuous_2019} without the
Jacobian inverse multiplier, which is not needed to obtain a globally convergent trust-region method. The
quantity of interest constraint $\Ccal_{\Phibold,\delta}^\mathtt{q}$ as well as a constraint similar to the
adjoint residual $\Ccal_{\Phibold,\delta}^\mathtt{ra}$ were previously proposed to accelerate dual-weighted
residual error estimation \cite{yano2020goal,du_adaptive_2021}.
\end{remark}

\begin{remark}
Due to the additional constraints in the proposed optimization-aware variant of EQP relative to the original
EQP method, there will be more nonzero weights for a given tolerance. We show in Section~\ref{sec:numexp} that
sparse weight vectors (that manifest as sparse reduced meshes in the PDE setting) will still be obtained,
particularly at early optimization iterations when the tolerances are loose.
\end{remark}

\begin{remark}\label{rem:eqp:rm_lin_depend}
In practice, some of the constraints in (\ref{eqn:eqp:rescon}) can lead to a linear program with a constraint matrix with
(nearly) linearly dependent rows, which can make the linear program difficult and costly to solve. To avoid this, we
remove linearly dependent rows from the constraint matrix using a QR factorization prior to solving the linear program
in (\ref{eqn:eqp:linprog}).
\end{remark}

\begin{remark}
If the DG formulation in (\ref{eqn:hdm:disc_dg}) is used, there is an issue with (\ref{eqn:eqp:rescon_simple}), as stated, in the case where the ROM solution is exact
at $\mubold\in\Xibold$, i.e., $\ubm^\star(\mubold) = \Phibold\hat\ybm_\Phibold^\star(\mubold)$. In this case, due to the structure
of DG in (\ref{eqn:hdm:dg_res_elem}), \textit{each element residual} is zero, i.e., $\rbm_e(\Phibold_e\hat\ybm_\Phibold^\star(\mubold),\mubold) = \zerobold$,
which means $\tilde\rbm_\Phibold(\hat\ybm_\Phibold^\star(\mubold),\mubold;\rhobold)=\zerobold$ \textit{for any}
$\rhobold\in\Rbb^{N_\mathtt{e}}$, which in turn implies $\Ccal_{\Phibold,\Xibold,\deltabold}=\Rbb^{N_\mathtt{e}}$, i.e., the
manifold accuracy constraint is not imposing a meaningful accuracy condition on the weight vector. Two ways to circumvent
this are: (1) use a DG formulation where the test functions of face terms are averaged across neighboring elements \cite{yano_discontinuous_2019}
or (2) split the $\Ccal_{\Phibold,\Xibold,\deltabold}^\mathtt{rp}$ constraint into two constraints, one for the face terms and
one for the volume terms in (\ref{eqn:eqp:rescon_simple}).
\end{remark}

\subsection{Error estimation}
\label{sec:hyperreduction:eqp_err_est}
We now introduce residual-based error estimates for the hyperreduced quantity of interest and its gradient. We begin with
a residual-based error bound that applies under regularity assumptions (Assumptions~\ref{assum:hdm}-\ref{assum:eqp} in~\ref{sec:appendix_A}) independent
of the training procedure for the reduced basis $\Phibold$ and weight vector $\rhobold$.
\begin{theorem}\label{the:qoi_grad_errbnd}
Under Assumptions~\ref{assum:hdm}-\ref{assum:eqp}, there exist constants $c_1,c_2>0$ such that for any $\mubold\in\Dcal$
and $\rhobold\in\Rcal$
\begin{equation}\label{eqn:eqp:qoi_errbnd}
 |f(\mubold) - \tilde{f}_\Phibold(\mubold;\rhobold)| \leq c_1\norm{\rbm(\Phibold\hat\ybm_\Phibold^\star(\mubold),\mubold)} + c_2\norm{\tilde\rbm_\Phibold(\hat\ybm_\Phibold^\star(\mubold),\mubold;\rhobold)} + \abs{\hat{j}(\hat\ybm_\Phibold^\star(\mubold),\mubold)-\tilde{j}_\Phibold(\hat\ybm_\Phibold^\star(\mubold),\mubold;\rhobold)}. 
\end{equation}
Furthermore, there exist constants $c_1',c_2',c_3',c_4'>0$ such that
\begin{equation}\label{eqn:eqp:qoi_grad_errbnd}
\begin{aligned}
 \norm{\nabla f(\mubold) - \nabla\tilde{f}_\Phibold(\mubold;\rhobold)} \leq
 &c_1'\norm{\rbm(\Phibold\hat\ybm_\Phibold^\star(\mubold),\mubold)} +
 c_2'\norm{\rbm^\lambda(\Phibold\hat\lambdabold_\Phibold^\star(\mubold),\Phibold\hat\ybm_\Phibold^\star(\mubold),\mubold)} + \\
 & c_3'\norm{\tilde\rbm_\Phibold(\hat\ybm_\Phibold^\star(\mubold;\rhobold),\mubold;\rhobold)} + 
 c_4'\norm{\tilde\rbm^\lambda(\hat\lambdabold_\Phibold^\star(\mubold),\hat\ybm_\Phibold^\star(\mubold),\mubold;\rhobold)} + \\
 & \norm{\hat\gbm_\Phibold^\lambda(\hat\lambdabold_\Phibold^\star(\mubold),\hat\ybm_\Phibold^\star(\mubold),\mubold) - \tilde\gbm_\Phibold^\lambda(\hat\lambdabold_\Phibold^\star(\mubold),\hat\ybm_\Phibold^\star(\mubold),\mubold;\rhobold)}.
\end{aligned}
\end{equation}
\begin{proof}
See~\ref{sec:appendix_B}.
\end{proof}
\end{theorem}

\begin{remark}
 Residual-based error estimates such as the one in Theorem~\ref{the:qoi_grad_errbnd} are rarely useful in practice because the constants
 multiplying each term are usually difficult to compute or estimate in practice, and even if these constants can be computed,
 the bounds tend to have low effectivity. However, in the optimization setting, we only require \textit{asymptotic bounds}
 where the constants in the bounds must exist but do not need to be computed, which makes residual-based estimates useful
 in this setting (Section~\ref{sec:trammo:eqp_tr}).
\end{remark}

Now we specialize the results in Theorem~\ref{the:qoi_grad_errbnd} with requirements on the training procedure for $\Phibold$ and $\rhobold$.
In particular, if the HDM primal and adjoint solution at a given $\mubold\in\Dcal$ are included in the column space of the
reduced basis, then the first term in (\ref{eqn:eqp:qoi_errbnd}) and the first two terms in (\ref{eqn:eqp:qoi_grad_errbnd}) are zero. All remaining terms are the difference
between the reduced and hyperreduced quantity evaluated at the reduced state $\hat\ybm_\Phibold^\star$, which are
exactly controlled by the EQP constraints in Section~\ref{sec:hyperreduction:eqp_training}, which leads to the following corollaries.

\begin{corollary}\label{cor:qoi_errbnd}
Suppose Assumptions~\ref{assum:hdm}-\ref{assum:eqp} hold with $\Rcal\supset\Ccal_{\Phibold,\Xibold,\delta_\mathtt{rp}}^\mathtt{rp} \cap \Ccal_{\Phibold,\Xibold,\delta_\mathtt{q}}^\mathtt{q}$ and consider any $\mubold\in\Dcal$. Then, if the reduced basis
$\Phibold\in\Rbb^{N_\ubm\times n}$ satisfies
\begin{equation}\label{eqn:prim_dual_basis}
 \ubm^\star(\mubold) \in \mathrm{Ran}~\Phibold, \qquad
 \lambdabold^\star(\mubold) \in \mathrm{Ran}~\Phibold,
\end{equation}
and the weight vector $\rhobold$ is the solution of (\ref{eqn:eqp:linprog}) with constraint set
$\Ccal_{\Phibold,\Xibold,\deltabold} \subseteq \Ccal_{\Phibold,\Xibold,\delta_\mathtt{rp}}^\mathtt{rp} \cap \Ccal_{\Phibold,\Xibold,\delta_\mathtt{q}}^\mathtt{q}$ and EQP training set $\Xibold\subset\Dcal$ with $\mubold\in\Xibold$, there
exist constants $c_2 > 0$ (independent of $\mubold$) such that
\begin{equation} \label{eqn:eqp:qoi_errbnd2}
 |f(\mubold) - \tilde{f}_\Phibold(\mubold;\rhobold)| \leq c_2 \delta_\mathtt{rp} + \delta_\mathtt{q}.
\end{equation}
\begin{proof}
See~\ref{sec:appendix_B}.
\end{proof}
\end{corollary}
 
\begin{corollary}\label{cor:qoi_grad_errbnd}
Suppose Assumptions~\ref{assum:hdm}-\ref{assum:eqp} hold with $\Rcal\supset\Ccal_{\Phibold,\Xibold,\delta_\mathtt{rp}}^\mathtt{rp} \cap \Ccal_{\Phibold,\Xibold,\delta_\mathtt{ra}}^\mathtt{ra} \cap \Ccal_{\Phibold,\Xibold,\delta_\mathtt{ga}}^\mathtt{ga}$ and consider any $\mubold\in\Dcal$.
Then, if the reduced basis
$\Phibold\in\Rbb^{N_\ubm\times n}$ satisfies (\ref{eqn:prim_dual_basis}) and the  weight vector is the solution of (\ref{eqn:eqp:linprog}) with constraint set
$\Ccal_{\Phibold,\Xibold,\deltabold} \subseteq \Ccal_{\Phibold,\Xibold,\delta_\mathtt{rp}}^\mathtt{rp} \cap \Ccal_{\Phibold,\Xibold,\delta_\mathtt{ra}}^\mathtt{ra} \cap \Ccal_{\Phibold,\Xibold,\delta_\mathtt{ga}}^\mathtt{ga}$ and EQP training set $\Xibold\subset\Dcal$ with $\mubold\in\Xibold$, there exist constants
$c_3',c_4'>0$ (independent of $\mubold$) such that
\begin{equation} \label{eqn:eqp:qoi_grad_errbnd2}
 \norm{\nabla f(\mubold) - \nabla\tilde{f}_\Phibold(\mubold;\rhobold)} \leq
 c_3'\delta_\mathtt{rp} + c_4'\delta_\mathtt{ra} + \delta_\mathtt{ga}.
\end{equation}
\begin{proof}
See~\ref{sec:appendix_B}.
\end{proof}
\end{corollary}

\begin{corollary} \label{cor:both_errbnd}
Suppose Assumptions~\ref{assum:hdm}-\ref{assum:eqp} hold with $\Rcal\supset\Ccal_{\Phibold,\Xibold,\delta_\mathtt{rp}}^\mathtt{rp} \cap \Ccal_{\Phibold,\Xibold,\delta_\mathtt{ra}}^\mathtt{ra} \cap \Ccal_{\Phibold,\Xibold,\delta_\mathtt{ga}}^\mathtt{ga} \cap \Ccal_{\Phibold,\Xibold,\delta_\mathtt{q}}^\mathtt{q}$
and consider any $\mubold\in\Dcal$. Then, if the reduced basis
$\Phibold\in\Rbb^{N_\ubm\times n}$ satisfies (\ref{eqn:prim_dual_basis}) and the weight vector is the solution of (\ref{eqn:eqp:linprog}) with constraint set
$\Ccal_{\Phibold,\Xibold,\deltabold} \subseteq \Ccal_{\Phibold,\Xibold,\delta_\mathtt{rp}}^\mathtt{rp} \cap \Ccal_{\Phibold,\Xibold,\delta_\mathtt{ra}}^\mathtt{ra} \cap \Ccal_{\Phibold,\Xibold,\delta_\mathtt{ga}}^\mathtt{ga} \cap \Ccal_{\Phibold,\Xibold,\delta_\mathtt{q}}^\mathtt{q}$ and EQP training set $\Xibold\subset\Dcal$ with $\mubold\in\Xibold$, then both (\ref{eqn:eqp:qoi_errbnd2}) and (\ref{eqn:eqp:qoi_grad_errbnd2}) hold.
\begin{proof}
The proof of this corollary directly follows Corollaries~\ref{cor:qoi_errbnd} and ~\ref{cor:qoi_grad_errbnd}
\end{proof}
\end{corollary}

\subsection{Application: shape optimization}
\label{sec:hyperreduction:shapeopt}
To close this section, we return to the shape optimization setting. Because we have cast the shape optimization problem
as the generic form of the discretized PDE (\ref{eqn:hdm:res_eqn}), all reduced and hyperreduced terms follow accordingly. However, as written,
the mesh motion is a potential bottleneck in the reduced workflow because it requires high-dimensional operations. To see
this,  we write the unassembled residual and quantity of interest
\begin{equation}
 \tilde\rbm_\Phibold(\tilde\ybm,\mubold;\rhobold) = \sum_{e=1}^{N_\mathtt{e}}\rho_e \Phibold_e^T\Rbm_e(\Phibold_e\tilde\ybm,\Phibold_e'\tilde\ybm,\phibold_e(\mubold)), \qquad
 \tilde{j}_\Phibold(\tilde\ybm,\mubold;\rhobold) = \sum_{e=1}^{N_\mathtt{e}} \rho_e J_e(\Phibold_e\tilde\ybm,\phibold_e(\mubold))
\end{equation}
from which we can see both terms depend on $\phibold_e(\mubold)$, i.e., for each value of $\mubold$ encountered
online, $\phibold_e(\mubold) = \Qbm_e^T\phibold(\mubold)$ must be computed for each $e\in\{1,\dots,N_\mathtt{e}\}$
where $\rho_e > 0$. Expanding $\phibold(\mubold)$ using (\ref{eqn:hdm:shapeopt:param_mapping}), we have
\begin{equation}
  \phibold_e(\mubold) =  \Abm_e\phibold_\mathtt{o}(\mubold).
\end{equation}
where $\Abm_e \in \Rbb^{N_\xbm^e\times N_\xbm^\mathtt{o}}$ is defined as
\begin{equation}\label{eqn:rom:shapeopt:eval_A}
 \Abm_e = \Qbm_{\mathtt{o},e}^T - \Qbm_{\mathtt{c},e}^T\Kbm_\mathtt{cc}^{-1}\Kbm_\mathtt{co}
\end{equation}
and the rows of $\Qbm_e$ are partitioned into optimized and unconstrained mesh coordinate
degrees of freedom (as described in Section~\ref{sec:pdeopt:shape_opt}) as $\Qbm_e = (\Qbm_{\mathtt{o},e},\Qbm_{\mathtt{c},e})$
with $\Qbm_{\mathtt{o},e}\in\Rbb^{N_\xbm^\mathtt{o}\times N_\xbm^e}$ and
$\Qbm_{\mathtt{c},e}\in\Rbb^{N_\xbm^\mathtt{c}\times N_\xbm^e}$.
The cost and feasibility of evaluating $\Abm_e$ using (\ref{eqn:rom:shapeopt:eval_A}) strongly depends on the size of $N_\xbm^\mathtt{c}$.
If $\Kbm_\mathtt{cc}$ is small enough to be store and factorized, then $\Kbm_\mathtt{cc}^{-1}\Kbm_\mathtt{co}$
can be computed at the cost of a factorization of $\Kbm_\mathtt{cc}$ and $N_\xbm^\mathtt{o}$ forward/backward
substitutions. On the other hand, if $\Kbm_\mathtt{cc}$ is too large for direct solvers to be feasible (often the case
for CFD), then $N_\xbm^\mathtt{o}$ iterative system solves of the form $\Kbm_\mathtt{cc} \vbm = \Kbm_\mathtt{co}$
are required, which can be very expensive depending on the size of $N_\xbm^\mathtt{o}$ and $N_\xbm^\mathtt{c}$.
Often in aerodynamic shape optimization, $N_\xbm^\mathtt{o}$ is proportional to the number of nodes
on the surface being optimized and $N_\xbm^\mathtt{c}$ proportional to the number of nodes in the
fluid mesh not on the surface being optimized. While both of these can be quite large (particularly for turbulent
flows), usually $N_\xbm^\mathtt{c} \gg N_\xbm^\mathtt{o}$.

To avoid this potentially significant cost in the hyperreduced model for shape optimization, we propose an
alternate approach to mesh deformation that uses projection-based model reduction. First, we recast the
equation in (\ref{eqn:hdm:shapeopt:mesh_coord}) as
\begin{equation}
 \Kbm_\mathtt{cc}\xbm_\mathtt{c} = -\Kbm_\mathtt{co}\xbm_\mathtt{o}
\end{equation}
and assume the constrained degrees of freedom can be well-approximated in a low-dimensional subspace
\begin{equation}
 \xbm_\mathtt{c} \approx \hat\xbm_\mathtt{c} \coloneqq \Psibold \hat\varpibold,
\end{equation}
where $\hat\xbm_\mathtt{c}\in\Rbb^{N_\xbm^\mathtt{c}}$ is the subspace approximate to $\xbm_\mathtt{c}$,
$\Psibold\in\Rbb^{N_\xbm^\mathtt{c}\times r}$ is the reduced basis for the constrained mesh coordinate
DoFs with $r \ll N_\xbm^\mathtt{c}$, and $\hat\varpibold\in\Rbb^r$ are the corresponding reduced coordinates.
To close the system, we apply a Galerkin projection to yield a linear system of equations for the reduced coordinates
\begin{equation}\label{eqn:rom:shapeopt:reduc_linelast}
 \hat\Kbm_\mathtt{cc}\varpibold = -\hat\Kbm_\mathtt{co}\xbm_\mathtt{o},
\end{equation}
where $\hat\Kbm_{cc}\in\Rbb^{r\times r}$ and $\hat\Kbm_\mathtt{co}\in\Rbb^{r\times N_\xbm^\mathtt{o}}$ are
the reduced elasticity stiffness matrix terms
\begin{equation}
 \hat\Kbm_\mathtt{cc} = \Psibold^T\Kbm_\mathtt{cc}\Psibold, \qquad
 \hat\Kbm_\mathtt{co} = \Psibold^T\Kbm_\mathtt{co}.
\end{equation}
By solving this system and substituting into (\ref{eqn:rom:shapeopt:reduc_linelast}), we have the expression for the constrained mesh coordinate DoFs
\begin{equation}
 \hat\xbm_\mathtt{c} = -\Psibold \hat\Kbm_\mathtt{cc}^{-1}\hat\Kbm_\mathtt{co}\xbm_\mathtt{o}.
\end{equation}
Finally, we use the subspace approximation $\hat\xbm_\mathtt{c}$ in place of the
true constrained mesh coordinate DoFs $\xbm_\mathtt{c}$ in (\ref{eqn:hdm:shapeopt:mesh_coord}) to define an reduced version of the
mesh motion mapping $\hat\phibold_e : \Dcal \rightarrow \Rbb^{N_\xbm^e}$ with
\begin{equation}
  \hat\phibold_e : \mubold \mapsto \hat\Abm_e\phibold_\mathtt{o}(\mubold),
\end{equation}
where $\hat\Abm_e \in \Rbb^{N_\xbm^e\times N_\xbm^\mathtt{o}}$ is defined as
\begin{equation}
 \hat\Abm_e = \Qbm_{\mathtt{o},e}^T - \Psibold_e\hat\Kbm_{cc}^{-1}\hat\Kbm_{co}
\end{equation}
and $\Psibold_e = \Qbm_{\mathtt{c},e}^T\Psibold$. The reduced stiffness matrix $\hat\Kbm_\mathtt{cc}$
is small ($r \ll N_\xbm^\mathtt{c}$) and can be factorized efficiently, keeping the overall cost
computing $\hat\Abm_e$ low.

The reduced basis $\Psibold$ is constructed once-and-for-all using a simple training strategy that
is effective for a small to moderate number of parameters. The basis is defined by compressing perturbations
of each shape parameter about a nominal configuration $\mubold_0\in\Dcal$ (in this work, we
use the original shape used to initialize the optimization iterations), i.e.,
\begin{equation}\label{eqn:rom:shapeopt:basis_train}
 \Psibold = \mathtt{POD}_{N_\xbm^\mathtt{c}, 2 N_\mubold}^r \left(\begin{bmatrix} \mubold_0 + \epsilon \onebold_1 & \mubold_0 - \epsilon\onebold_1 & \cdots & \mubold_0 + \epsilon \onebold_{N_\mubold} & \mubold_0 - \epsilon\onebold_{N_\mubold} \end{bmatrix}\right),
\end{equation}
where $\mathtt{POD}_{m,n}^k : \Rbb^{m\times n} \rightarrow \Rbb^{m\times k}$ applies the SVD to the argument
(snapshot matrix of size $m\times n$) and extracts the $k$ left singular vectors, and $\epsilon\in\Rbb$ is the
magnitude of the perturbation. The reduced mesh motion is demonstrated in Figures~\ref{fig:rom:shapopt:org_naca}-\ref{rom:fig:reduc_mshmot} using a NACA0012 airfoil parametrized using Bezier curves.
Figures~\ref{fig:rom:shapopt:org_naca} and~\ref{fig:rom:shapopt:msh_perturb} show the original shape and mesh of the airfoil,
and the perturbations used to train the mesh motion basis in (\ref{eqn:rom:shapeopt:basis_train}) with $\epsilon=0.5$ and
$r=35$.
Figure~\ref{rom:fig:reduc_mshmot} shows the training procedure applied to a NACA0012
airfoil, as well as the full and reduced mesh motion when the airfoil is deformed to an RAE2822.
\begin{figure}[H]
  \centering
  \includegraphics[width=0.7\textwidth]{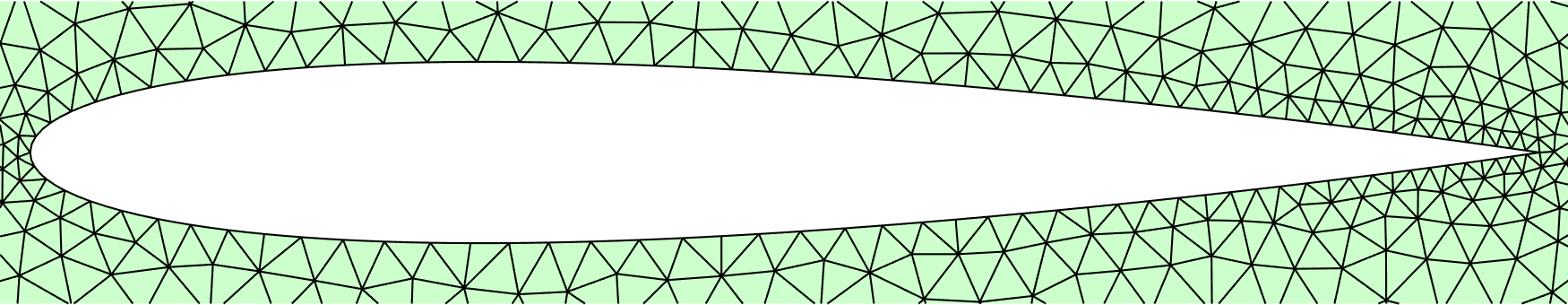}
  \caption{Original shape and mesh of the NACA0012 airfoil}
  \label{fig:rom:shapopt:org_naca}
\end{figure}

\begin{figure}[H]
  \centering
  \begin{minipage}[t]{0.325\textwidth}
      \centering
      \includegraphics[width=\textwidth]{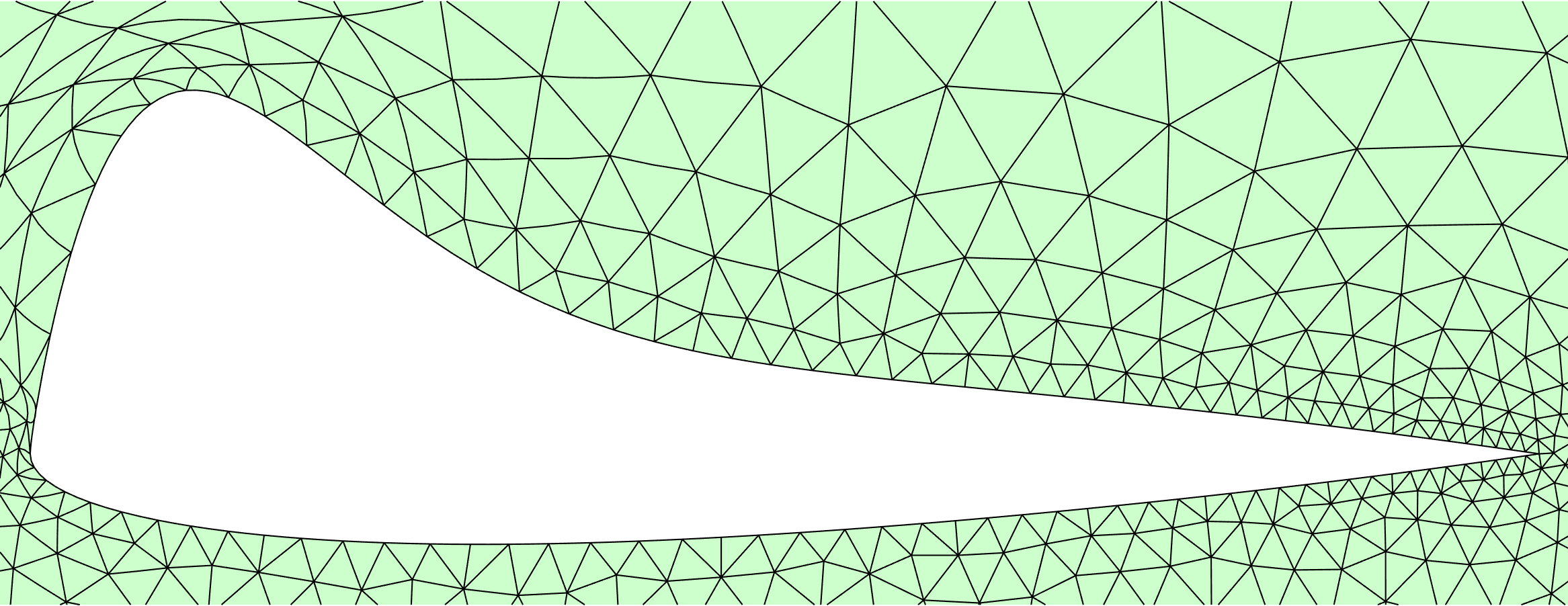}
  \end{minipage}
  \hfill
  \begin{minipage}[t]{0.325\textwidth}
      \centering
      \includegraphics[width=\textwidth]{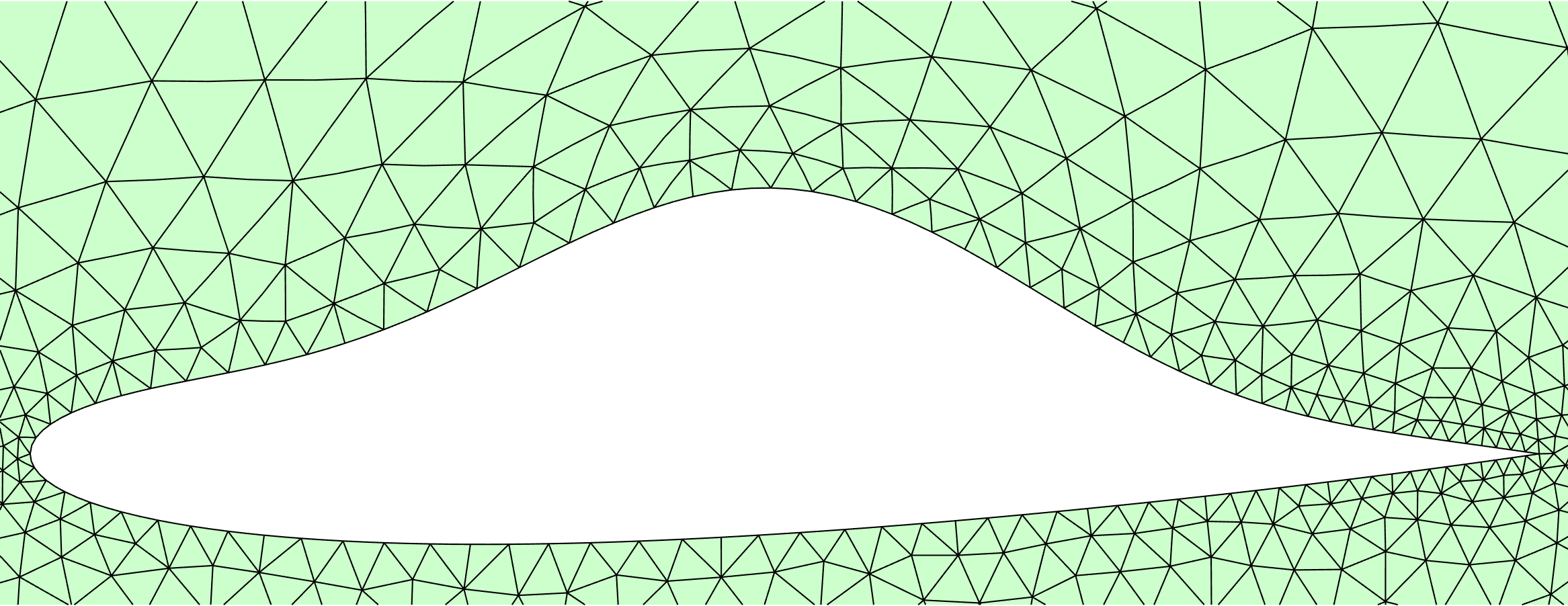}
  \end{minipage}
  \hfill
  \begin{minipage}[t]{0.325\textwidth}
    \centering
    \includegraphics[width=\textwidth]{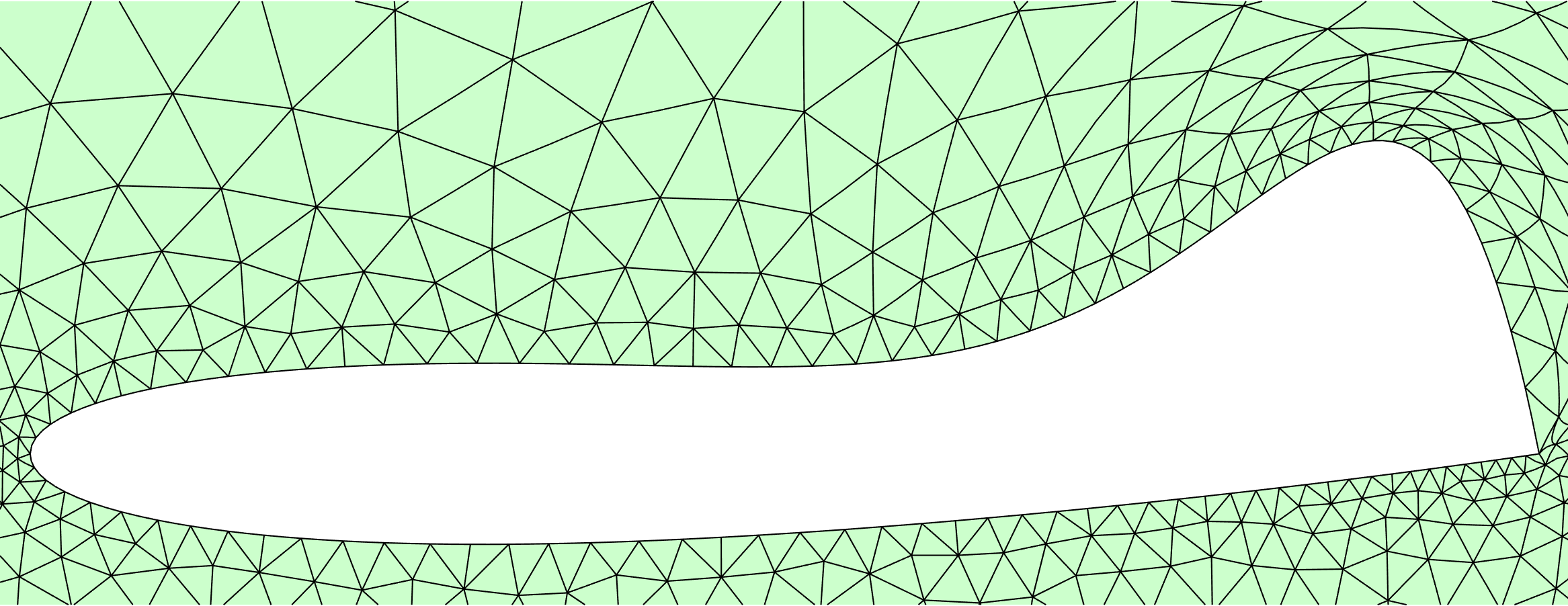}
  \end{minipage}

  \vspace{3.00 mm}

  \begin{minipage}[t]{0.325\textwidth}
    \centering
    \includegraphics[width=\textwidth]{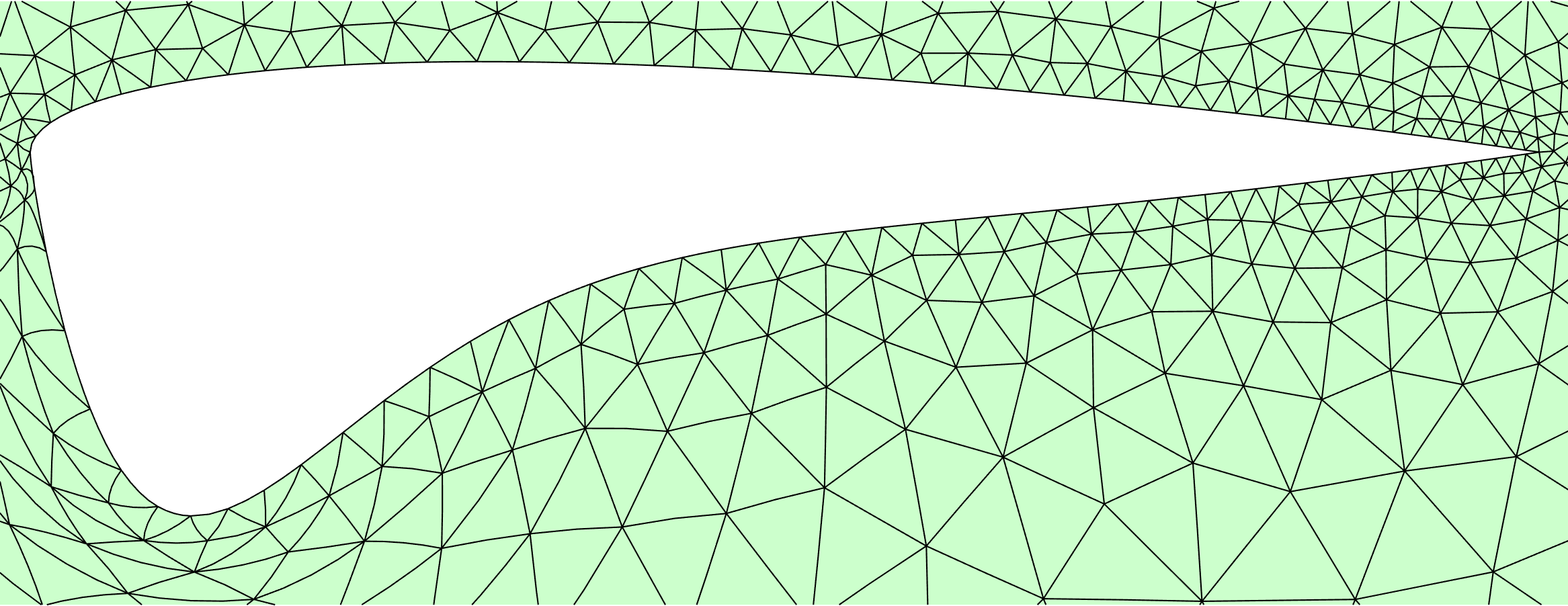}
  \end{minipage}
  \hfill
  \begin{minipage}[t]{0.325\textwidth}
    \centering
    \includegraphics[width=\textwidth]{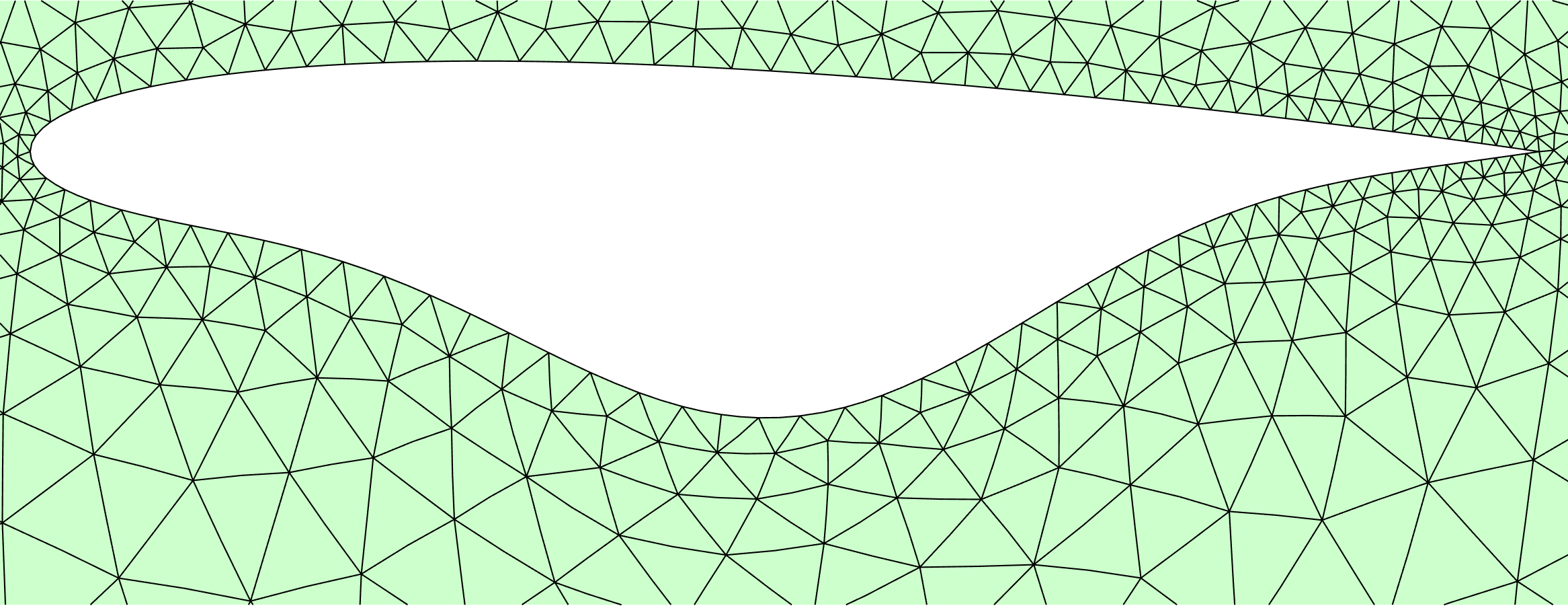}
  \end{minipage}
  \hfill
  \begin{minipage}[t]{0.325\textwidth}
    \centering
    \includegraphics[width=\textwidth]{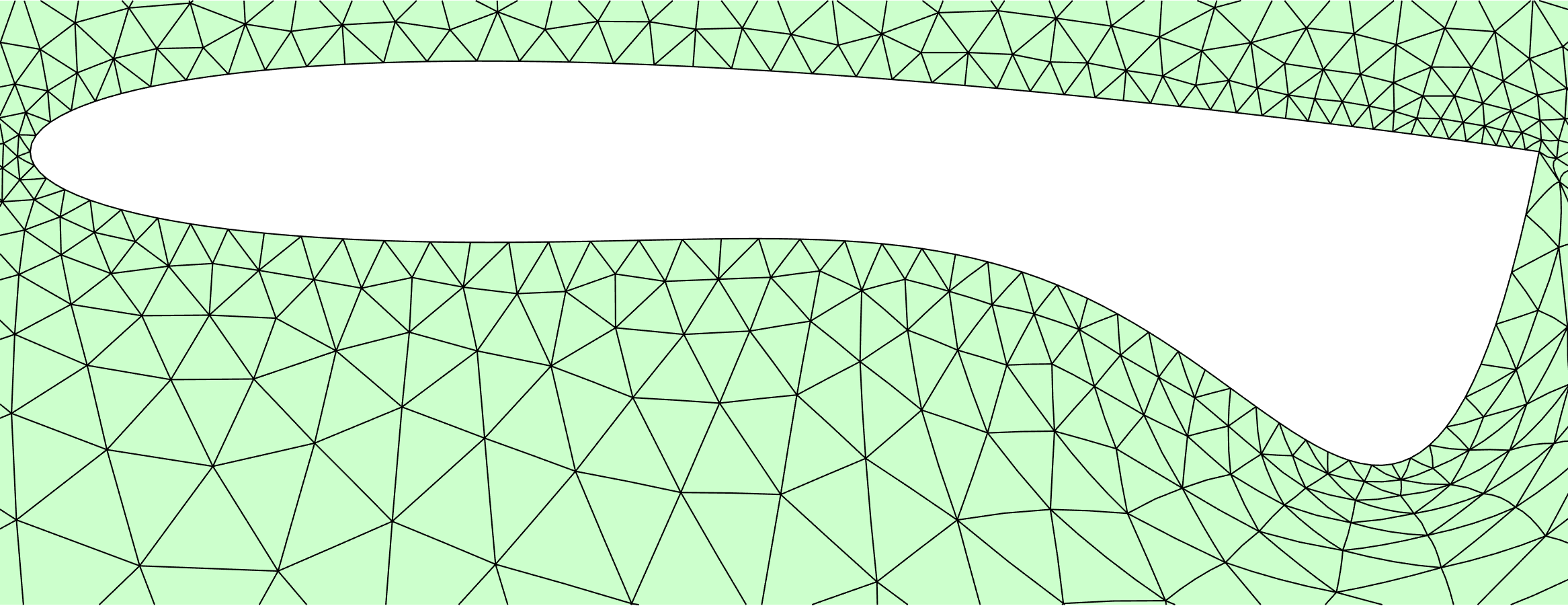}
  \end{minipage}
  \caption{Subset of the perturbations used to train the reduced mesh motion model.}
  \label{fig:rom:shapopt:msh_perturb}
\end{figure}

\begin{figure}[H]
    \centering
    \begin{minipage}[t]{0.49\textwidth}
        \centering
        \includegraphics[width=\textwidth]{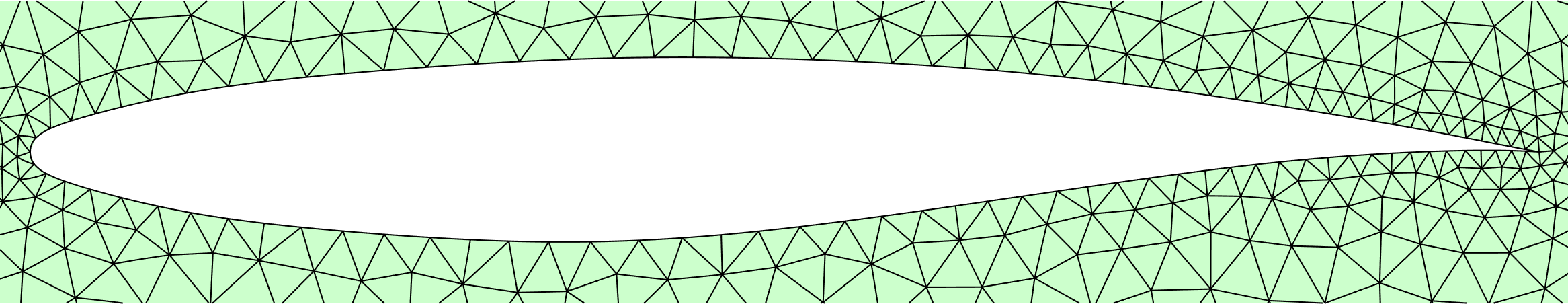}
    \end{minipage}
    \hfill
    \begin{minipage}[t]{0.49\textwidth}
        \centering
        \includegraphics[width=\textwidth]{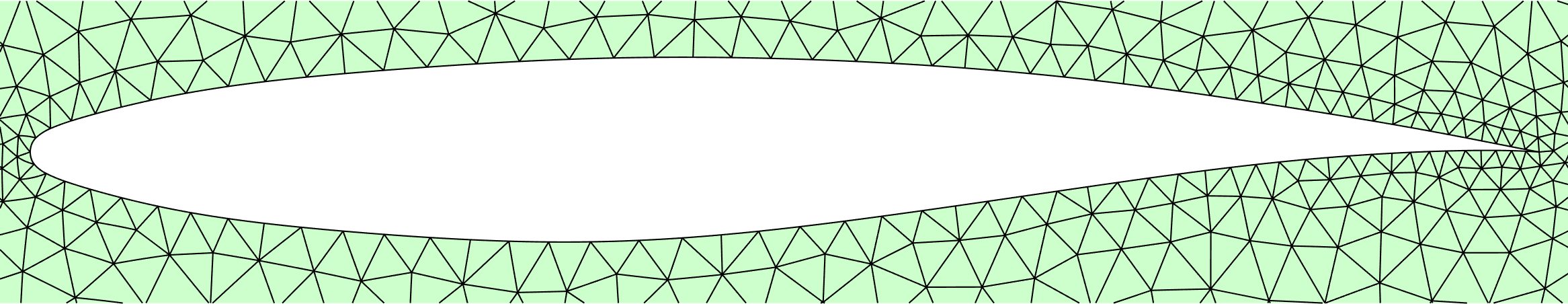}
    \end{minipage}
    \caption{Comparison of original mesh motion obtained using linear elasticity (\textit{left}) and the reduced mesh motion (\textit{right}).}
    \label{rom:fig:reduc_mshmot}
\end{figure}

\begin{remark}
 In general, $\epsilon$ in (\ref{eqn:rom:shapeopt:basis_train}) will depend on the scale of the problem being considered, but should be chosen so
 the resulting deformation introduces non-trivial mesh motion. In this work, we choose $\epsilon=1$ for both examples in
 Section~\ref{sec:numexp}.
\end{remark}

\begin{remark}
Linear, physics-based mesh motion problem is ideally suited for projection-based model reduction for several reasons.
The governing equations are linear and thus significant speedups are easily obtained without hyperreduction. Most importantly,
a reduced-order model is built of $\xbm_\mathtt{c}$ (constrained mesh coordinates DoFs, e.g., nodes away
from surface being optimized), which has \textit{no impact} on the shape of the optimized surface
(completely controlled by $\xbm_\mathtt{o}$). This means the reduced-order model doesn't not need to
perfectly predict $\xbm_\mathtt{c}$ (from full linear elasticity) as long as it leads to meshes that are
well-conditioned.
\end{remark}

\begin{remark}
In the optimization setting, $\hat\Abm_e$ is computed at the first iteration for each element in the entire mesh
because $\rhobold$ will adapt throughout the optimization iterations. Once $\rhobold$ is available at a given
iteration, only the $\hat\Abm_e$ with $\rho_e > 0$ will be used. Furthermore, the basis $\Psibold$ and therefore
the $\{\hat\Abm_e\}_{e=1}^{N_\mathtt{e}}$ terms will not be updated throughout the optimization iteration, which
makes amortizing the initial $2N_\mubold$ unreduced mesh motions required to train $\Psibold$ in (\ref{eqn:rom:shapeopt:basis_train}) trivial.
\end{remark}

\section{Globally convergent trust-region method using hyperreduced models}
\label{sec:trammo}
In this section, we use the hyperreduced models introduced in Section~\ref{sec:hyperreduction} to accelerate the optimization problem
of interest (\ref{eqn:hdm:opt:reduc_space}). To ensure the hyperreduced optimization problem converges to a local minimum of the original
(unreduced) problem, we embed it in trust-region algorithm that allows for models with inexact gradients at trust-region
centers and asymptotic error bounds \cite{kouri_trust-region_2013}. To this end, we introduce the general trust-region method in
which the hyperreduced model will be embedded (Section~\ref{sec:trammo:tr_org}), as well as the proposed approach to train
and leverage hyperreduced models in the trust-region framework to ensure global convergence (Section~\ref{sec:trammo:eqp_tr}).

\subsection{A trust-region method with inexact model gradients and asymptotic error bounds}\label{sec:trammo:tr_org}
Recall the optimization problem in (\ref{eqn:hdm:opt:reduc_space}), where the objective function $f$ satisfies Assumption~\ref{assum:tr_inexact} below.
A trust-region method constructs a sequence of trust-region centers
$\{\mubold_k\}_{k=1}^\infty$ whose limit will be a local solution of (\ref{eqn:hdm:opt:reduc_space}). Trust-region methods
construct a smooth approximation model $m_k : \Dcal \rightarrow \Rbb$ at each trust-region center
such that $m_k(\mubold) \approx f(\mubold)$ for all $\mubold\in\{\mubold\in\Dcal | \norm{\mubold-\mubold_k}\leq \Delta_k\}$,
where $\Delta_k>0$ is the trust-region radius. A candidate step $\check\mubold_k$ is produced by
approximately solving the trust-region subproblem
\begin{equation}\label{eqn:tr:subprob}
 \optconOne{\mubold \in \Dcal}{m_k(\mubold)}{\norm{\mubold-\mubold_k}_2\leq \Delta_{k}.}
\end{equation}
The trust-region subproblem does not need to be solved exactly; rather, it must only satisfy the
fraction of Cauchy decrease condition \cite{conn_trust-region_2000}. The step candidate is evaluated using the actual-to-predicted reduction ratio,
$\varrho_k$ in (\ref{eqn:tr:redu_ratio}), to determine whether to accept the step and how to adjust the trust-region radius.

While there are many trust-region methods available in the literature, we base our hyperreduced optimization
framework on the method in \cite{kouri_trust-region_2013} because it only requires an asymptotic error condition of the form
\begin{equation}\label{eqn:tr:asym_err}
\norm{\nabla m_k(\mubold_k) - \nabla f(\mubold_k)} \leq \xi \min\{\norm{\nabla m_k(\mubold_k)}, \Delta_k\},
\end{equation}
where $\xi > 0$ is any arbitrary constant (independent of $k$). Due to the arbitrariness of $\xi$, the above
bound is not particularly meaningful for a fixed $k$ (the bound may lack effectivity), but it requires the
model gradient to become increasingly accurate as $\norm{\nabla m_k(\mubold_k)} \rightarrow 0$ or
$\Delta_k \rightarrow 0$ (e.g., an asymptotic bound). If the model $m_k$ is equipped with the error bound
\begin{equation}\label{eqn:tr:grad_cond_varphi}
 \norm{\nabla m_k(\mubold_k) - \nabla f(\mubold_k)} \leq \xi \varphi_k(\mubold_k),
\end{equation}
where $\varphi_k : \Dcal \rightarrow \Rbb$ denotes a \textit{computable} error indicator for the model gradient,
then the gradient condition can be stated solely in terms of the error bound as
\begin{equation}\label{eqn:tr:varphi}
 \varphi_k(\mubold_k) \leq \kappa_\varphi \min\{\norm{\nabla m_k(\mubold_k)},\Delta_k\},
\end{equation}
where $\kappa_\varphi > 0$ is a chosen constant. This provides a computable criteria to use to select
the model $m_k$ at a given trust-region step.
The complete trust-region algorithm, summarized in Algorithm~\ref{alg:tr_inexact}, is globally convergent
(see \cite{kouri_trust-region_2013} for proof) with $\Dcal=\Rbb^{N_\mubold}$ under the following assumptions.
\begin{assume}\label{assum:tr_inexact}
    Assumptions on the trust-region method with inexact gradient condition
    \begin{enumerate}[label=\textbf{(AT\arabic*)}]
        \item \label{trammo:assump:1} $f$ is twice continuously differentiable and bounded below
        \item \label{trammo:assump:2} $m_k$ is twice continuously differentiable for $k=1,2,\ldots.$
        \item \label{trammo:assump:3}There exits $\kappa_1 >0, \kappa_2>1$ 
        such that $\norm{\nabla^2 f(\mubold)} \leq \kappa_1$ and $\norm{\nabla^2 m_k(\mubold)} \leq \kappa_2 - 1$
        \item \label{trammo:assump:4} There exists $\xi >0 $ such that
        \begin{equation*}
            \norm{\nabla m_k (\mubold_k)-\nabla f(\mubold_k)} \leq \xi \min\curlyb{\norm{\nabla m_k (\mubold_k)}, \Delta_k}.
        \end{equation*}
    \end{enumerate}
    \label{trammo:assump}
\end{assume}
\begin{algorithm}[H]
    \begin{algorithmic}[1]
        \REQUIRE Current iterate $\mubold_k$ and radius $\Delta_k$, and parameters $0<\gamma_1\leq\gamma_2<1$, $\Delta_\mathrm{max}>0$, $0< \eta_1<\eta_2<1$ \\
        \ENSURE Next iterate $\mubold_{k+1}$
        \STATE 
            {\bf Model update:} Construct the approximation model, $m_{k}(\mubold)$, that satisfies \ref{trammo:assump:2} and \ref{trammo:assump:4}.
        \STATE 
            {\bf Step computation:} Solve the trust-region subproblem to get the candidate center $\check{\mubold}_{k}$
            \begin{equation*}
                \optconOne{\mubold \in \Dcal}{m_k(\mubold)}{\norm{\mubold-\mubold_k}_2\leq \Delta_{k}}
            \end{equation*}
            such that $\check{\mubold}_k$ satisfies the fraction of \textit{Cauchy decrease condition} \cite{conn_trust-region_2000}
        \STATE
            {\bf Actual-to-predicted reduction ratio:} Evaluate the actual-to-predicted reduction ratio
            \begin{equation}\label{eqn:tr:redu_ratio}
                \varrho_{k}=\frac{f(\mubold_{k})-f(\check{\mubold}_{k})}{m_k(\mubold_{k})-m_k(\check{\mubold}_{k})}
            \end{equation}

        \STATE {\bf Step acceptance}: \\
            \begin{tabbing}
                \qquad \qquad \qquad 
                \=\textbf{if} \qquad \=$\varrho_{k} \geq \eta_1$  \qquad \= \textbf{then}  \qquad \=$\mubold_{k+1}=\check\mubold_k$ \qquad \= \textbf{else} \qquad \=$\mubold_{k+1}=\mubold_k$ \qquad \= \textbf{end if}
            \end{tabbing}
            
        \STATE
            {\bf Trust-region radius update:} \\
            \begin{tabbing}
                \qquad \qquad \qquad 
                \=\textbf{if} \qquad \=$\varrho_{k} < \eta_1$ \qquad \qquad \= \textbf{then}  \qquad \=$\Delta_{k+1} \in [\gamma_1 \Delta_k, \gamma_2 \Delta_k]$ \qquad \= \textbf{end if} \\

                \>\textbf{if}  \>$\varrho_{k} \in [\eta_1, \eta_2)$ \>\textbf{then} \>$\Delta_{k+1} \in [\gamma_2 \Delta_k, \Delta_k]$ \>\textbf{end if}\\

                \>\textbf{if}  \>$\varrho_{k} \geq \eta_2$ \>\textbf{then} \>$\Delta_{k+1} \in [\Delta_k, \Delta_{\text{max}}]$ \>\textbf{end if} 
            \end{tabbing}
\end{algorithmic} 
\caption{Trust-region method with inexact gradient condition}
\label{alg:tr_inexact}
\end{algorithm}

\begin{remark}
Many trust-region algorithms require the model $m_k$ and its
gradient match the true objective $f$ at trust-region centers either exactly or to within a prescribed tolerance.
Because a hyperreduced model can only, in general, match the corresponding HDM solution if $\rhobold=\onebold$
(i.e., no hyperreduction), methods that require model exactness at trust-region centers (e.g., \cite{alexandrov_trust-region_1998,zahr_adaptive_2016,yano_globally_2021}) are not useful for
this method. Furthermore, methods that require accuracy of the model and its gradient to within specified tolerances (e.g., \cite{carter1991global})
are also not feasible as this would rely on tight error bounds between a hyperreduced model and the corresponding HDM solution, which do not exist for general, nonlinear systems.
\end{remark}

\begin{remark}
Only existence of the constant $\xi$ is needed; its actual value is not necessary because it is not used in
Algorithm~\ref{alg:tr_inexact}. This is crucial to establish global convergence in the model reduction setting for a general
class of nonlinear system because the constant $\xi$ absorbs Lipschitz constants and bounds that are
rarely computable.
\end{remark}


\subsection{Accelerated optimization using trust regions and on-the-fly model hyperreduction}\label{sec:trammo:eqp_tr}
With the individual components---the hyperreduced models and the globally convergent trust-region method based on
asymptotic error bounds---in place, we combine them to define the proposed trust-region method accelerated with
model hyperreduction, to be referred to as EQP/TR in the remainder. We begin by defining the trust-region model ($m_k$)
at the $k$th trust-region center $\mubold_k$ as the quadratic approximation to the hyperreduced quantity of interest
\begin{equation}\label{eqn:tr:mk}
 m_k(\mubold) = \tilde{f}_{\Phibold_k}(\mubold_k; \rhobold_k) + \nabla\tilde{f}_{\Phibold_k}(\mubold_k; \rhobold_k)^T(\mubold-\mubold_k) + \frac{1}{2} (\mubold-\mubold_k)^T\tilde\Hbm_k(\mubold-\mubold_k),
\end{equation}
where $\Phibold_k\in\Rbb^{N_\ubm \times n_k}$ is the reduced basis, $\rhobold_k\in\Ccal_{\Phibold_k,\Xibold_k,\deltabold_k}$ is the weight
vector, and $\tilde\Hbm_k\in\Rbb^{N_\mubold\times N_\mubold}$ is the Hessian of $\tilde{f}_{\Phibold_k}$
\begin{equation}
 \tilde\Hbm_k \coloneqq \nabla^2\tilde{f}_{\Phibold_k}(\mubold_k; \rhobold_k),
\end{equation}
all at the $k$th trust-region center. Hessian-vector products are computed using a finite difference approximation involving
the gradient
\begin{equation}
 \tilde\Hbm_k \vbm \approx \frac{1}{\varepsilon}\left[\nabla\tilde{f}_{\Phibold_k}(\mubold_k+\varepsilon \vbm; \rhobold_k) - \nabla\tilde{f}_{\Phibold_k}(\mubold_k;\rhobold_k) \right],
\end{equation}
where $\varepsilon\in\Rbb_{>0}$ is the finite difference step size ($\varepsilon = 10^{-6}$ in this work).
Assuming the reduced basis dimension $n_k$ is small and the weight vector $\rhobold_k$ is sparse, the model $m_k$ and
its gradient $\nabla m_k$ will be much less expensive to query than the HDM $f$.

The reduced basis is chosen as
\begin{equation}\label{eqn:tr:construct_basis}
 \Phibold_k = \mathtt{GramSchmidt}\left(\begin{bmatrix} \ubm^\star(\mubold_k) & \lambdabold^\star(\mubold_k) & \Phibold_k^\mathtt{p} & \Phibold_k^\mathtt{a} \end{bmatrix}\right),
\end{equation}
to guarantee $\ubm^\star(\mubold_k),\lambdabold^\star(\mubold_k)\in\mathrm{Ran}~\Phibold_k$
(required for Corollaries~\ref{cor:qoi_errbnd}-\ref{cor:both_errbnd} to hold), where
\begin{equation}\label{eqn:tr:pod_snapshots}
 \Phibold_k^\mathtt{p} \coloneqq \mathtt{POD}_{N_\ubm,k}^{p_k}(\Ubm_{k-1})\in\Rbb^{N_\ubm\times p_k}, \qquad
 \Phibold_k^\mathtt{a} \coloneqq \mathtt{POD}_{N_\ubm,k}^{q_k}(\Vbm_{k-1})\in\Rbb^{N_\ubm\times q_k}
\end{equation}
are optimal compressions of state and adjoint snapshots from \textit{all previous iterations}
\begin{equation}\label{eqn:tr:snapshots}
\begin{aligned}
 \Ubm_k &= \begin{bmatrix} \ubm^\star(\mubold_0) & \cdots & \ubm^\star(\mubold_k)\end{bmatrix} \in \Rbb^{N_\ubm\times(k+1)}, \\
 \Vbm_k &= \begin{bmatrix} \lambdabold^\star(\mubold_0) & \cdots & \lambdabold^\star(\mubold_k)\end{bmatrix} \in \Rbb^{N_\ubm\times(k+1)}.
\end{aligned}
\end{equation}
With this choice, the size of the hyperreduced model will evolve as $n_k = 2 + p_k + q_k$, where $0\leq p_k\leq k$ and $0\leq q_k\leq k$
are user-defined parameters.

Next, we choose the constraint set $\Ccal_{\Phibold_k,\Xibold_k,\deltabold_k}$ such that
\begin{equation}\label{eqn:tr:conv_constr_set}
\Ccal_{\Phibold_k,\Xibold_k,\deltabold_k} \subseteq
\Ccal_{\Phibold_k,\Xibold_k,\delta_{\mathtt{rp},k}}^\mathtt{rp} \cap \Ccal_{\Phibold_k,\Xibold_k,\delta_{\mathtt{ra},k}}^\mathtt{ra} \cap \Ccal_{\Phibold_k,\Xibold_k,\delta_{\mathtt{ga},k}}^\mathtt{ga}
\end{equation}
and the EQP training set $\Xibold_k\subset\Dcal$ such that $\mubold_k \in \Xibold_k$ to
ensure Corollary~\ref{cor:qoi_grad_errbnd} holds, which leads to the result
\begin{equation}
 \norm{\nabla f(\mubold_k) - \nabla m_k(\mubold_k)} =
 \norm{\nabla f(\mubold_k) - \nabla \tilde{f}_{\Phibold_k}(\mubold_k;\rhobold_k)} \leq
 c_3'\delta_{\mathtt{rp},k} + c_4'\delta_{\mathtt{ra},k} + \delta_{\mathtt{ga},k}.
\end{equation}
Therefore, we take $\varphi_k$ in (\ref{eqn:tr:varphi}) to be
\begin{equation}
 \varphi_k \coloneqq \kappa_1\delta_{\mathtt{rp},k} + \kappa_2\delta_{\mathtt{ra},k} + \kappa_3\delta_{\mathtt{ga},k},
\end{equation}
where $\kappa_1,\kappa_2,\kappa_3>0$ are user-defined parameters.
Then, condition (\ref{eqn:tr:asym_err}) leads to the following bound on the tolerances
\begin{equation}
 \kappa_1\delta_{\mathtt{rp},k} + \kappa_2\delta_{\mathtt{ra},k} + \kappa_3\delta_{\mathtt{ga},k} \leq \hat\kappa~\mathrm{min}\left\{\norm{\nabla m_k(\mubold_k)},\Delta_k\right\}.
\end{equation}
For simplicity, we impose the slightly stronger condition that equally splits the bound among the three tolerances as
\begin{equation}\label{eqn:tr:strong_cond}
\begin{aligned}
 \delta_{\mathtt{rp},k} &\leq \frac{\hat\kappa}{3\kappa_1}\mathrm{min}\left\{\norm{\nabla m_k(\mubold_k)},\Delta_k\right\}, \\
 \delta_{\mathtt{ra},k} &\leq \frac{\hat\kappa}{3\kappa_2}\mathrm{min}\left\{\norm{\nabla m_k(\mubold_k)},\Delta_k\right\}, \\
 \delta_{\mathtt{ga},k} &\leq \frac{\hat\kappa}{3\kappa_3}\mathrm{min}\left\{\norm{\nabla m_k(\mubold_k)},\Delta_k\right\}.
\end{aligned}
\end{equation}
Weighting of the tolerances can be achieved through the choice of $\kappa_1,\kappa_2,\kappa_3$, although we take $\kappa_1=\kappa_2=\kappa_3$ in this work). Finally, with these choices, we chose the weight vector $\rhobold_k$
to the solution of (\ref{eqn:eqp:linprog}), i.e.,
\begin{equation}\label{eqn:tr:compute_eqp_weights}
\rhobold_k = \rhobold^\star(\Phibold_k, \Xibold_k,\deltabold_k),
\end{equation}
where the tolerances
\begin{equation}\label{eqn:tr:all_constr}
 \deltabold_k = (\delta_\mathtt{dv},\delta_{\mathtt{rp},k},\delta_{\mathtt{ra},k},\delta_{\mathtt{ga},k},\delta_\mathtt{q},\delta_\mathtt{rs})
\end{equation}
are chosen according to (\ref{eqn:tr:strong_cond}) with arbitrary $\delta_\mathtt{dv},\delta_\mathtt{q},\delta_\mathtt{rs} > 0$.

The complete EQP/TR algorithm is summarized in Algorithm~\ref{alg:tr_hyp} (Algorithm~\ref{alg:tr_inexact} specialized to the choice for the
choice of $m_k$ in the hyperreduction setting). The specific construction of
the reduced basis $\Phibold_k$ and weight vector $\rhobold_k$ outlined above are sufficient to
ensure the complete algorithm is globally convergent (Theorem~\ref{the:qoi_grad_errbnd}).

\begin{algorithm}[H]
    \begin{algorithmic}[1]
        \REQUIRE Current iterate $\mubold_k$ and radius $\Delta_k$, trust-region parameters $0<\gamma_1\leq\gamma_2<1, \Delta_\mathrm{max}>0, 0<\eta_1<\eta_2<1$, snapshot matrices $\Ubm_k$ and $\Vbm_k$
       \ENSURE Next iterate $\mubold_{k+1}$, updated snapshot matrices $\Ubm_{k+1}$ and $\Vbm_{k+1}$
        \STATE {\bf Model update:} Build approximation model $m_k(\mubold)$ in (\ref{eqn:tr:mk})
        \begin{itemize}[topsep=0pt, itemsep=0pt]
          \setlength\itemsep{0pt}
          \item Solve primal and adjoint HDM: $\ubm^\star(\mubold_k)$, $\lambdabold^\star(\mubold_k)$
          \item Construct reduced basis $\Phibold_k$ according to (\ref{eqn:tr:construct_basis})
          \item Compute EQP weights $\rhobold_k$ according to (\ref{eqn:tr:compute_eqp_weights}) with tolerances given by (\ref{eqn:tr:strong_cond}) and (\ref{eqn:tr:all_constr})
          \item Update snapshot matrices
          \begin{equation}
           \Ubm_{k+1} = \begin{bmatrix} \ubm^\star(\mubold_k) & \Ubm_k \end{bmatrix}, \qquad
           \Vbm_{k+1} = \begin{bmatrix} \lambdabold^\star(\mubold_k) & \Vbm_k \end{bmatrix}
          \end{equation}
        \end{itemize}
        \STATE {\bf Step computation:} Solve the trust-region subproblem (\ref{eqn:tr:subprob}) to get the candidate center $\check\mubold_k$
        \STATE {\bf Actual-to-predicted reduction ratio:} Compute the actual-to-predicted reduction ratio $\varrho_k$ in (\ref{eqn:tr:redu_ratio})
        \STATE {\bf Step acceptance}: \\
            \begin{tabbing}
                \qquad \qquad \qquad 
                \=\textbf{if} \qquad \=$\varrho_{k} \geq \eta_1$  \qquad \= \textbf{then}  \qquad \=$\mubold_{k+1}=\check\mubold_k$ \qquad \= \textbf{else} \qquad \=$\mubold_{k+1}=\mubold_k$ \qquad \= \textbf{end if}
            \end{tabbing}
            
        \STATE
            {\bf Trust-region radius update:} \\
            \begin{tabbing}
                \qquad \qquad \qquad 
                \=\textbf{if} \qquad \=$\varrho_{k} < \eta_1$ \qquad \qquad \= \textbf{then}  \qquad \=$\Delta_{k+1} \in [\gamma_1 \Delta_k, \gamma_2 \Delta_k]$ \qquad \= \textbf{end if} \\

                \>\textbf{if}  \>$\varrho_{k} \in [\eta_1, \eta_2)$ \>\textbf{then} \>$\Delta_{k+1} \in [\gamma_2 \Delta_k, \Delta_k]$ \>\textbf{end if}\\

                \>\textbf{if}  \>$\varrho_{k} \geq \eta_2$ \>\textbf{then} \>$\Delta_{k+1} \in [\Delta_k, \Delta_{\text{max}}]$ \>\textbf{end if} 
            \end{tabbing}
\end{algorithmic}
\caption{Trust-region method with hyperreduced approximation models}
\label{alg:tr_hyp}
\end{algorithm}

\begin{theorem}
Suppose Assumptions~\ref{trammo:assump:1}-\ref{trammo:assump:3} and~\ref{assum:hdm}-\ref{assum:eqp} hold
with $\Rcal\supset\cup_{k=1}^\infty\Ccal_{\Phibold_k,\Xibold_k,\deltabold_k}$. Then the iterates $\{\mubold_k\}$ generated by Algorithm~\ref{alg:tr_hyp} satisfy
\begin{equation}\label{eqn:tr:inf_mk}
 \liminf_{k\rightarrow\infty}\,\norm{\nabla m_k(\mubold_k)} = \liminf_{k\rightarrow\infty}\,\norm{\nabla f(\mubold_k)} = 0
\end{equation}
independent of the choice of $p_k$ and $q_k$.
\begin{proof}
From Assumption~\ref{trammo:assump:1}-\ref{trammo:assump:3} and Theorem 4.3 of \cite{kouri_trust-region_2013}, the result (\ref{eqn:tr:inf_mk}) holds if the model $m_k$ satisfies (\ref{eqn:tr:asym_err}).
With the choice of $\Phibold_k$ in (\ref{eqn:tr:construct_basis}) or (\ref{eqn:tr:construct_basis_sens}), we have $\ubm^\star(\mubold_k),\lambdabold^\star(\mubold_k)\in\mathrm{Ran}~\Phibold_k$
independent of the choice of $p_k$ or $q_k$.
Furthermore, from the choice of $\Xibold_k$ such that $\mubold_k\in\Xibold_k$ (e.g., $\Xibold_k=\{\mubold_k\}$ as in (\ref{eqn:tr:eqp_train_set})) and
the constraint set (\ref{eqn:tr:conv_constr_set}), the assumptions of Corollary~\ref{cor:qoi_grad_errbnd} are satisfied, which implies the existence of constants
$c_3',c_4'>0$ such that
\begin{equation}
 \norm{\nabla f(\mubold_k) - \nabla m_k(\mubold_k)} = \norm{\nabla f(\mubold_k) - \nabla \tilde{f}_\Phibold(\mubold_k;\rhobold_k)} \leq
   c_3'\delta_{\mathtt{rp},k} + c_4'\delta_{\mathtt{ra},k} + \delta_{\mathtt{ga},k}.
\end{equation}
From the condition in (\ref{eqn:tr:strong_cond}), this reduces to
\begin{equation}
\norm{\nabla f(\mubold_k) - \nabla m_k(\mubold_k)} \leq \hat\kappa \left(\frac{c_3'}{3\kappa_1} + \frac{c_4'}{3\kappa_2} + \frac{1}{3\kappa_3}\right) \mathrm{min}\left\{\norm{\nabla m_k(\mubold_k)},\Delta_k\right\},
\end{equation} 
which is identical to (\ref{eqn:tr:asym_err}) with $\xi = \hat\kappa \left(\frac{c_3'}{3\kappa_1} + \frac{c_4'}{3\kappa_2} + \frac{1}{3\kappa_3}\right)$.
Therefore, the result in (\ref{eqn:tr:inf_mk}) holds.
\end{proof}
\end{theorem}

\begin{remark}
The trust-region subproblems do not need to be solved exactly, rather they must guaranteed a fraction of the Cauchy decrease
condition \cite{conn_trust-region_2000}. In this work, we leverage this opportunity for efficiency by using the Steihaug-Toint Truncated Conjugate Gradient
method, which guarantees the candidate step will satisfy the fraction of Cauchy decrease condition.
\end{remark}

\begin{remark}
We choose the quadratic model in (\ref{eqn:tr:mk}) rather than the more obvious choice of directly using the hyperreduced model,
i.e., $m_k(\mubold) = \tilde{f}_{\Phibold_k}(\mubold; \rhobold_k)$, as done in other work \cite{zahr_adaptive_2016,qian_certified_2017,yano_globally_2021}
because the trust-region subproblems are less expensive to solve. For example, a truncated conjugate
gradient method can be used rather than a more sophisticated solver for a general nonlinear subproblem. Even though the trust-region
subproblem only requires hyperreduced model queries, it can become a bottleneck (especially at later iterations when the
reduced basis is large) and the solver can be prone to failure. In some cases these issues are supplanted because this approach
can lead to models with higher predictive capability, which leads to more successful iterations and fewer HDM queries \cite{zahr_adaptive_2016,qian_certified_2017}. 
\end{remark}

\begin{remark}
With the choice of the reduced basis $\Phibold_k$ in (\ref{eqn:tr:construct_basis}), the basis size is bounded as $n_k \leq 2(k+1)$ leading to a very small
hyperreduced model in the early iterations, which can limit its predictive capability. Therefore, following the work in \cite{zahr_efficient_2019},
we will also consider a related approach that includes the HDM sensitivities at the starting point
\begin{equation}\label{eqn:tr:construct_basis_sens}
 \Phibold_k = \mathtt{GramSchmidt}\left(\begin{bmatrix} \ubm^\star(\mubold_k) & \lambdabold^\star(\mubold_k) & \partial_\mubold\ubm^\star(\mubold_0) & \Phibold_k^\mathtt{p} & \Phibold_k^\mathtt{a} \end{bmatrix}\right),
\end{equation}
which leads to a basis of size $n_k = 2+N_\mubold+p_k+q_k$ (bounded by $n_k \leq 2(k+1)+N_\mubold$).
\end{remark}

\begin{remark}
Fast singular value decomposition updates \cite{brand_fast_2006} can be used to mitigate the cost of constructing $\Phibold_k^\mathtt{p}$ and
$\Phibold_k^\mathtt{q}$ as the snapshot matrices grow with $k$ in (\ref{eqn:tr:pod_snapshots}) and (\ref{eqn:tr:snapshots}).
\end{remark}

\begin{remark}\label{rem:eqp_train_set}
The only requirement on the EQP training set $\Xibold_k$ is that it contains the trust-region center, i.e., $\mubold_k \in \Xibold_k$.
In this work, we take
\begin{equation}\label{eqn:tr:eqp_train_set}
\Xibold_k = \{\mubold_k\}
\end{equation}
because this is the simplest and least expensive option. Other natural choices
are to include previous trust-region centers, i.e., $\Xibold_k = \{\mubold_{k-t}, \dots, \mubold_k\}$, or local sampling about the
current trust-region center. Numerical investigation into these options showed they significantly increase computational cost without
appreciably improving the predictive capability of the hyperreduced model, so we focus solely on the approach in (\ref{eqn:tr:eqp_train_set}).
\end{remark}

\begin{remark}
As stated, the tolerances $\delta_{\mathtt{rp},k},\delta_{\mathtt{ra},k},\delta_{\mathtt{ga},k}$ in (\ref{eqn:tr:all_constr}) must be computed iteratively
because the bound for the tolerance depends on $m_k$, which in turn depends on the tolerances. In practice, we lag the
right-hand side to $m_{k-1}(\mubold_{k-1})$. Experimentation with both approaches yielded no discernible difference in
convergence, but a significantly higher computational cost for the iterative approach, so we only consider the lagged approach.
\end{remark}

\begin{remark}
In some cases, an affine subspace approximation $\ubm^\star \approx \bar\ubm + \Phibold\hat\ybm^\star$ is preferred to the linear
approximation in (\ref{eqn:rom:ansatz}). All results and discussions directly carry though in the affine setting with only a small change to the
basis construction. Specifically, we choose the affine offset at the $k$th iteration to be the trust-region center, i.e.,
$\bar\ubm_k = \ubm^\star(\mubold_k)$, which no longer requires $\ubm^\star(\mubold_k)$ to be explicitly included
in the basis construction in (\ref{eqn:tr:construct_basis}) and (\ref{eqn:tr:construct_basis_sens}). Lastly, we modify the primal snapshot matrix to be
\begin{equation}
 \Ubm_{k-1} = \begin{bmatrix} \ubm^\star(\mubold_0)-\ubm^\star(\mubold_k) & \cdots & \ubm^\star(\mubold_{k-1})-\ubm^\star(\mubold_k) \end{bmatrix}
\end{equation}
because the reduced basis $\Phibold_k$ now only represents \textit{deviations} of the primal state from the offset
(rather than the primal state itself).
\end{remark}

\section{Numerical experiments}
\label{sec:numexp}
In this section, we apply the proposed hyperreduced optimization method to solve two shape optimization problems to study
the performance of the overall method and its sensitivity with respect to key parameters including the gradient condition
tolerance ($\hat\kappa/\kappa_i$), the size of the reduced basis ($n_k$), and
the empirical quadrature constraints used ($\Ccal_{\Phibold,\Xibold,\deltabold}$). Specifically, we study the
proposed EQP/TR method under the following collection of EQP constraints,
\begin{equation}\label{eqn:numexp:constr_set}
\begin{aligned}
 \Ccal_{\Phibold,\Xibold,\deltabold}^{(1)} &= \Ccal_{\delta_\mathtt{dv}}^\mathtt{dv} \cap \Ccal_{\Phibold,\Xibold,\delta_\mathtt{rp}}^\mathtt{rp} \cap \Ccal_{\Phibold,\Xibold,\delta_\mathtt{ra}}^\mathtt{ra} \cap \Ccal_{\Phibold,\Xibold,\delta_\mathtt{ga}}^\mathtt{ga} \cap \Ccal_{\Phibold,\Xibold,\delta_\mathtt{q}}^\mathtt{q} \\
 \Ccal_{\Phibold,\Xibold,\deltabold}^{(2)} &= \Ccal_{\delta_\mathtt{dv}}^\mathtt{dv} \cap \Ccal_{\Phibold,\Xibold,\delta_\mathtt{rp}}^\mathtt{rp} \cap \Ccal_{\Phibold,\Xibold,\delta_\mathtt{ra}}^\mathtt{ra} \cap \Ccal_{\Phibold,\Xibold,\delta_\mathtt{ga}}^\mathtt{ga}  \cap \Ccal_{\Phibold,\Xibold,\delta_\mathtt{rs}}^\mathtt{rs} \\
 \Ccal_{\Phibold,\Xibold,\deltabold}^{(3)} &= \Ccal_{\delta_\mathtt{dv}}^\mathtt{dv} \cap \Ccal_{\Phibold,\Xibold,\delta_\mathtt{rp}}^\mathtt{rp} \cap \Ccal_{\Phibold,\Xibold,\delta_\mathtt{ra}}^\mathtt{ra} \cap \Ccal_{\Phibold,\Xibold,\delta_\mathtt{ga}}^\mathtt{ga} \cap \Ccal_{\Phibold,\Xibold,\delta_\mathtt{rs}}^\mathtt{rs} \cap \Ccal_{\Phibold,\Xibold,\delta_\mathtt{q}}^\mathtt{q},
\end{aligned}
\end{equation}
where $\Ccal_{\Phibold,\Xibold,\deltabold}^{(1)}$ is sufficient to establish global convergence of the method, but we will
show that adding the sensitivity residual constraint and/or the QoI constraint can accelerate convergence. To provide a
baseline for comparison, we will compare the EQP/TR methods to directly solving the optimization problem in (\ref{eqn:hdm:opt:reduc_space})
using HDM evaluations (no model reduction) with a BFGS linesearch method and the trust-region approach in \cite{zahr_adaptive_2016}
based on reduced-order models only (no hyperreduction). Furthermore, we also demonstrate the effectiveness of including
the sensitivity solution at the starting point $\partial_\mubold \ubm^\star(\mubold_0)$ by comparing several of the aforementioned
approaches with and without it. Table~\ref{tab:numexp:var_model} summarizes the various methods we study in this section, and sets notation
for later reference.
\begin{table}[H]
\centering
\caption{Various optimization methods studied in this work.}
\label{tab:numexp:var_model}
\begin{tabular}{l|c|c|c}
label & model & $\partial_\mubold \ubm^\star(\mubold_0)$ included in training & EQP constraint set \\\hline
$\mathrm{HDM}$ & HDM & - & - \\
$\mathrm{ROM}$ & ROM & no & - \\
$\mathrm{ROM}_{\partial_0}$ & ROM & yes & - \\
$\mathrm{EQP}^{(i)}$ & EQP & no & $\Ccal_{\Phibold,\Xibold,\deltabold}^{(i)}$ \\
$\mathrm{EQP}_{\partial_0}^{(i)}$ & EQP & yes & $\Ccal_{\Phibold,\Xibold,\deltabold}^{(i)}$ \\
\end{tabular}
\end{table}

To assess the performance of each method in Table~\ref{tab:numexp:var_model}, we consider the computational cost required
to achieve a given value of the objective function, i.e., given $\epsilon>0$, we study the computational cost required
for each algorithm to satisfy $S_k < \epsilon$, where $S_k$ is the normalized distance from the optimal objective value
\begin{equation}\label{eqn:numexp:sk}
 S_k \coloneqq |f(\mubold_k)|, \qquad
 S_k \coloneqq \frac{|f(\mubold_k) - f^*|}{|f^*|}
\end{equation}
and $f^*$ is the optimal value of the objective function. The first definition of $S_k$ is used when
$f^* = 0$ and second is used otherwise. In addition, we will consider the convergence
history of $S_k$ to provide a complete picture of the performance of each method. Because we have
included hyperreduction, CPU time will directly be used to quantify the computational cost as opposed
to other work that has relied on cost estimates \cite{zahr_adaptive_2016,zahr_efficient_2019}.

There are relatively few user-defined parameters because the procedure to define the basis $\Phibold_k$
and half of the EQP tolerances ($\delta_\mathtt{rp}, \delta_\mathtt{ra}, \delta_\mathtt{ga}$) are prescribed
in Section~\ref{sec:trammo:eqp_tr} and come directly from the optimization convergence criteria. Among the remaining
user-defined parameters are those related to the trust-region algorithm itself; however, these are
well-studied at this point so we fix to reasonable values: $\eta_1 = 0.1$, $\eta_2 = 0.75$,
$\gamma_1=0.5$, and $\gamma_2 = 1$. Another important trust-region parameter is the initial trust-region
radius $\Delta_0$, which has also been extensively studied in other work \cite{yano_globally_2021} so we
set it to a reasonable value $\Delta_0=0.1$. The remaining parameters that must be set are
the EQP tolerances that do not appear in the convergence criteria ($\delta_\mathtt{dv}, \delta_\mathtt{q}, \delta_\mathtt{rs}$).
In this work, we are only interested in whether including these additional constraints improves the convergence rate
of the EQP/TR method so we either set $\delta_\mathtt{dv}=\delta_\mathtt{q}=\delta_\mathtt{rs}=\infty$ (constraints not used) or
$\delta_\mathtt{dv}=10^{-4}$, $\delta_\mathtt{q}=10^{-6}$ and $\delta_\mathtt{rs}=10^{-3}$.


\subsection{Shape optimization of aorto-coronaric bypass}
\label{sec:numexp:bypass}
The first problem we consider is shape optimization of an aorto-coronaric bypass adapted from \cite{rozza_optimization_2005}.
Let $\Omega \in \Rbb^2$ be the domain (initial configuration) shown Figure~\ref{fig:bypass:schematic} with incompressible,
viscous flow governed by the steady, incompressible Navier-Stokes equations
 \begin{equation} \label{eqn:numexp:bypass:ins}
     (v\cdot\nabla)v - \nu\nabla^2 v + \frac{1}{\rho_0}\nabla P = 0, \quad \nabla \cdot v = 0 \quad\text{in}~~\Omega,
 \end{equation}
 where $\rho_0\in\Rbb_{>0}$ is the density of the fluid, $\nu\in\Rbb_{>0}$
 is the kinematic viscosity of the fluid, and $\func{v}{\Omega}{\Rbb^d}$ and
 $\func{P}{\Omega}{\Rbb_{>0}}$ are the velocity and pressure, respectively,
 of the fluid implicitly defined as  the solution of (\ref{eqn:numexp:bypass:ins}).
 Boundary conditions for the boundaries identified in Figure~\ref{fig:bypass:schematic} are
 \begin{equation}\label{eqn:numexp:bypass:ins_bc}
   v = v_\mathrm{in} \quad\text{on}~~\Gamma_\mathrm{in}, \qquad
   \sigma \cdot n = 0\quad\text{on}~~\Gamma_\mathrm{out}, \qquad
   v = 0 \quad\text{on}~~\Gamma_\mathrm{w},
 \end{equation}
where $v_\text{in} = (10,0) $ is the inflow velocity and $\func{n}{\partial\Omega}{\Rbb^d}$ is the outward unit
normal to the boundary of the domain.
In this work we take $\rho_0 = 1$ and $\nu = \frac{L \norm{v_\mathrm{in}}}{\mathrm{Re}}$, where $L=0.2$
is the length scale for the problem and $\mathrm{Re}=500$ is the corresponding Reynolds number.
\begin{figure}[H]
    \centering
    \tikzset{every picture/.style={scale=1.1}}
    \input{_py/bypass/bypass0_geom.tikz}
    \caption{Schematic of the aorto-coronaric bypass with boundaries: $\Gamma_\mathrm{in}$ (\ref{line:bypass:inlet}),
    $\Gamma_\mathrm{out}$ (\ref{line:bypass:outlet}),
    $\Gamma_\mathrm{c}$ (\ref{line:bypass:control}),
    and $\Gamma_\mathrm{w} \coloneqq \partial\Omega\setminus(\Gamma_\mathrm{in}\cup\Gamma_\mathrm{out})$ (\ref{line:bypass:domain}).}
    \label{fig:bypass:schematic}
\end{figure}

The objective of this problem is to minimize the vorticity ($\nabla \times v$) of the flow in the region $\Omega_\mathrm{wd}$
by adjusting the shape of the upper wall ($\Gamma_\mathrm{c}$). The shape of $\Gamma_\mathrm{c}$ is parametrized using
a Bezier curve with 8 control points and the deformation is extended to the rest of the domain using linear elasticity
(Section~\ref{sec:hyperreduction:shapeopt}). In this case, the parameter vector $\mubold$ denotes the collection of Bezier control points.
This leads to the following PDE-constrained shape optimization problem 
\begin{equation}
  \underset{v,P,\mubold}{\text{minimize}} ~~\int_{\Omega_\mathrm{wd}(\mubold)} \nabla \times v ~ dV + \frac{\alpha}{2}\norm{\mubold}_2^2, \quad
  \text{subject to:} ~~ \Lcal(v,P;\mubold)=0, ~~ -0.4\leq \mubold \leq 0.4,
\end{equation}
where $\Lcal$ is the differential operator that includes the incompressible Navier-Stokes equations (\ref{eqn:numexp:bypass:ins}) and boundary
conditions (\ref{eqn:numexp:bypass:ins_bc}), the bound constraints were included to ensure the shapes remain physically relevant, and we set the
regularization parameter to $\alpha=500$. The starting point and optimal solution for this problem are shown in Figure~\ref{fig:numexp:bypass:soln}.
\begin{figure}[H]
    \centering
    \begin{minipage}[t]{0.49\textwidth}
        \centering
        \includegraphics[width=\textwidth]{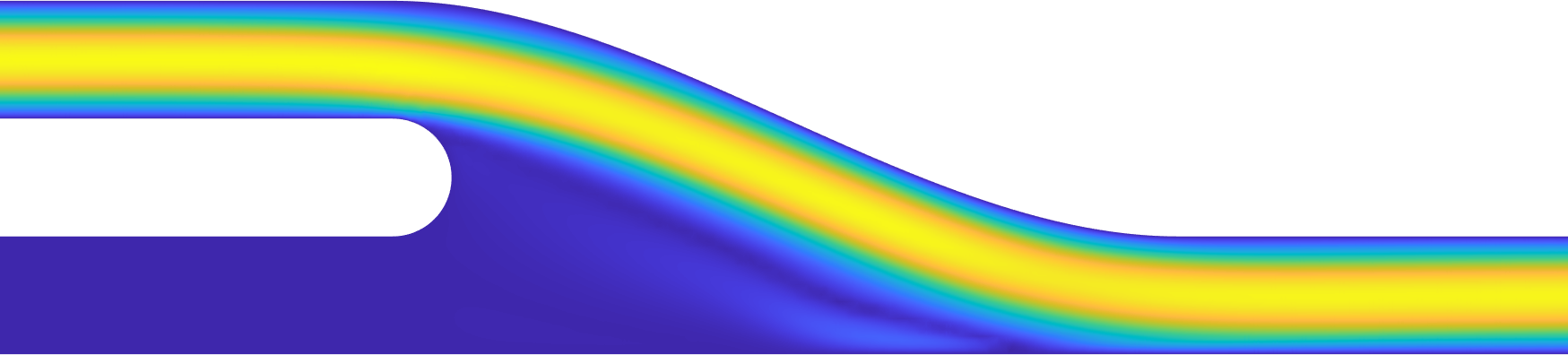}
    \end{minipage}
    \hfill
    \begin{minipage}[t]{0.49\textwidth}
        \centering
        \includegraphics[width=\textwidth]{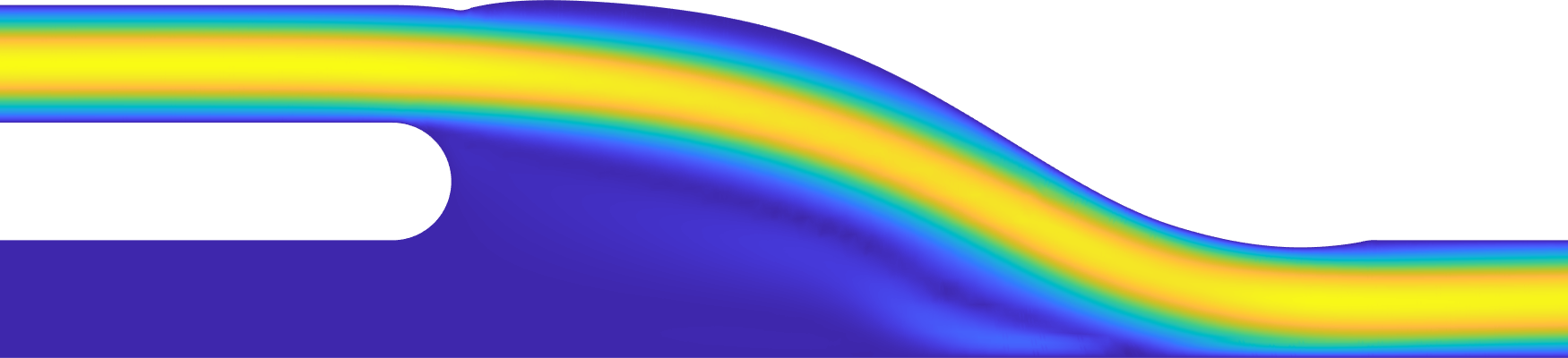}
    \end{minipage}
    \caption{The domain shape and velocity magnitude for the bypass problem at the starting configuration (\textit{left}) and optimal solution (\textit{right}).}
    \label{fig:numexp:bypass:soln}
\end{figure}

The governing equations and quantity of interest are discretized using the finite element method with $\Pcal^2-\Pcal^1$
Taylor-Hood elements to yield a discrete optimization problem of the form (\ref{eqn:hdm:opt:reduc_space}). A quadratic mesh consisting of $1749$
triangular elements was generated using DistMesh \cite{persson_simple_2004} and used for all numerical experiments.

\subsubsection{Influence of the gradient condition tolerance}
To begin, we study the impact of $\hat\kappa/\kappa_i$ in (\ref{eqn:tr:strong_cond}) on convergence of the EQP/TR method, which are
user-defined parameters that define the relative importance of each term in the gradient bound, as well as the initial
magnitude of the various EQP tolerances. For simplicity, we take $\kappa \coloneqq \hat\kappa/\kappa_i$ for $i=1,2,3$.
To keep the size of this study manageable, we study $\kappa$ in isolation by only
considering the $\mathrm{EQP}_{\partial_0}^{(3)}$ method and not truncating the reduced basis.
The convergence rate of the EQP/TR method depends
moderately on the choice of $\kappa$; however, for all values considered, EQP/TR converged faster than the
$\mathrm{HDM}$ method (Figure~\ref{fig:numexp:bypass:kappa}). Furthermore, $\kappa=10^{-4}$ leads to the fastest convergence
for this configuration and will be used in the remainder.
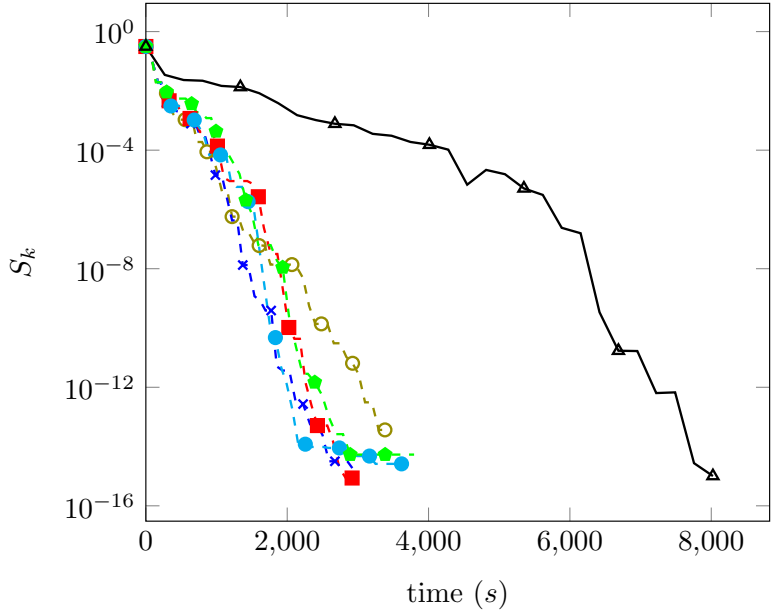
\begin{figure}[H]
    \centering
    \tikzset{every picture/.style={scale=1.1}}
    \begin{tikzpicture}
\begin{axis}[
xmin=0,
ymode=log,
xlabel={time $(s)$},
ylabel={$S_k$}]
\addplot [mark=o, mark size=2, mark options={solid}, dashed, thick, olive, mark repeat=5]
coordinates {
( 0.00000000e+00,  3.14463285e-01)
( 6.55390244e+01,  2.18930844e-01)
( 9.29927762e+01,  9.36700143e-02)
( 1.41737382e+02,  1.97916823e-02)
( 2.09025799e+02,  1.97916823e-02)
( 2.91261000e+02,  8.28307244e-03)
( 3.55787452e+02,  2.16713139e-03)
( 4.22496108e+02,  1.23271886e-03)
( 4.89744462e+02,  1.06780814e-03)
( 5.22336271e+02,  1.06780814e-03)
( 5.55889197e+02,  1.06780814e-03)
( 5.88544608e+02,  1.06780814e-03)
( 6.55958577e+02,  1.06780814e-03)
( 7.24129987e+02,  1.88556968e-04)
( 7.95627486e+02,  1.88556968e-04)
( 8.64119290e+02,  8.80644557e-05)
( 9.33128645e+02,  8.80644557e-05)
( 1.00340811e+03,  2.07774020e-05)
( 1.07753328e+03,  3.97947154e-06)
( 1.15606845e+03,  3.97947154e-06)
( 1.22014376e+03,  5.71965789e-07)
( 1.30725619e+03,  5.71965789e-07)
( 1.38128376e+03,  1.75976686e-07)
( 1.44642925e+03,  1.75976686e-07)
( 1.52197208e+03,  8.46319911e-08)
( 1.60258956e+03,  6.04234215e-08)
( 1.68237610e+03,  6.04234215e-08)
( 1.76373007e+03,  1.37713077e-08)
( 1.91181925e+03,  1.37713077e-08)
( 1.98840301e+03,  1.37713077e-08)
( 2.06695397e+03,  1.37713077e-08)
( 2.14696146e+03,  1.01343240e-08)
( 2.22500264e+03,  4.76630027e-09)
( 2.31093532e+03,  5.44139724e-10)
( 2.39813729e+03,  1.37555550e-10)
( 2.48332737e+03,  1.37555550e-10)
( 2.56920704e+03,  1.37555550e-10)
( 2.65902791e+03,  3.05232956e-11)
( 2.74907443e+03,  3.05232956e-11)
( 2.83768376e+03,  1.04758972e-11)
( 2.92484752e+03,  6.45036575e-12)
( 3.01117227e+03,  2.74427958e-12)
( 3.10155104e+03,  3.10190215e-13)
( 3.19839680e+03,  3.10190215e-13)
( 3.29454706e+03,  3.64419273e-14)
( 3.38297377e+03,  3.64419273e-14)};\label{line:bypass:kappa3}

\addplot [mark=x, mark size=2, mark options={solid}, dashed, thick, blue, mark repeat=5]
coordinates {
( 0.00000000e+00,  3.14463285e-01)
( 6.67281485e+01,  2.18964731e-01)
( 9.43358384e+01,  9.31033568e-02)
( 1.75265180e+02,  1.86158060e-02)
( 2.09267711e+02,  1.86158060e-02)
( 2.92133832e+02,  4.06542158e-03)
( 3.62016327e+02,  4.06542158e-03)
( 4.31756689e+02,  2.08268849e-03)
( 4.99804704e+02,  1.86565882e-03)
( 5.69886635e+02,  1.86565882e-03)
( 6.36901107e+02,  7.85640653e-04)
( 7.05009544e+02,  7.85640653e-04)
( 7.73070303e+02,  2.70862147e-04)
( 8.42922794e+02,  2.70862147e-04)
( 9.12223738e+02,  6.69991499e-05)
( 9.83906714e+02,  1.44038287e-05)
( 1.07455285e+03,  1.44038287e-05)
( 1.14707324e+03,  2.41503987e-06)
( 1.22050393e+03,  4.23228041e-07)
( 1.29438046e+03,  4.23228041e-07)
( 1.37137507e+03,  1.32967247e-08)
( 1.44778459e+03,  1.32967247e-08)
( 1.53778006e+03,  1.17766687e-09)
( 1.61589665e+03,  8.43891206e-10)
( 1.69340834e+03,  3.88131544e-10)
( 1.77285373e+03,  3.88131544e-10)
( 1.85322974e+03,  5.03144436e-12)
( 1.93453262e+03,  3.66212448e-12)
( 2.03042063e+03,  3.66212448e-12)
( 2.12645922e+03,  2.71289903e-13)
( 2.22452971e+03,  2.71289903e-13)
( 2.32341095e+03,  5.66874425e-14)
( 2.41775189e+03,  5.66874425e-14)
( 2.49978687e+03,  2.44392290e-14)
( 2.58892281e+03,  3.18143810e-15)
( 2.67500566e+03,  3.18143810e-15)
( 2.76680015e+03,  3.03682728e-15)
( 2.85794867e+03,  3.03682728e-15)
( 2.94328443e+03,  1.59071905e-15)};\label{line:bypass:kappa4}

\addplot [mark=square*, mark size=2, mark options={solid}, dashed, thick, red, mark repeat=5]
coordinates {
( 0.00000000e+00,  3.14463285e-01)
( 6.67900521e+01,  2.18967865e-01)
( 9.48486113e+01,  9.30327237e-02)
( 1.76639498e+02,  1.86635893e-02)
( 2.45112833e+02,  1.86635893e-02)
( 3.28846516e+02,  4.62939274e-03)
( 3.65113031e+02,  3.18002496e-03)
( 4.50953096e+02,  3.18002496e-03)
( 4.84041571e+02,  2.18295568e-03)
( 5.57053446e+02,  2.18295568e-03)
( 6.25721284e+02,  1.17451482e-03)
( 6.97934729e+02,  9.68777544e-04)
( 7.67239000e+02,  4.10832973e-04)
( 8.62242386e+02,  4.10832973e-04)
( 9.37753519e+02,  4.10832973e-04)
( 1.01282738e+03,  1.37223308e-04)
( 1.10570710e+03,  1.16558413e-05)
( 1.18473503e+03,  9.07009867e-06)
( 1.43533591e+03,  9.07009867e-06)
( 1.51401193e+03,  7.26904601e-06)
( 1.59168984e+03,  2.68076143e-06)
( 1.66976846e+03,  2.40027416e-07)
( 1.76848471e+03,  3.14182133e-08)
( 1.86091478e+03,  3.14182133e-08)
( 1.94110806e+03,  1.82512495e-09)
( 2.02232711e+03,  1.04035484e-10)
( 2.10187618e+03,  4.32286579e-11)
( 2.17935924e+03,  4.32286579e-11)
( 2.26017610e+03,  1.40764175e-12)
( 2.34416784e+03,  2.41644685e-13)
( 2.42789577e+03,  5.10476204e-14)
( 2.56741177e+03,  5.10476204e-14)
( 2.64915414e+03,  1.47503039e-14)
( 2.73903923e+03,  1.87994070e-15)
( 2.82794426e+03,  8.67664937e-16)
( 2.91835902e+03,  8.67664937e-16)
( 3.00162784e+03,  8.67664937e-16)};\label{line:bypass:kappa5}

\addplot [mark=*, mark size=2, mark options={solid}, dashed, thick, cyan, mark repeat=5]
coordinates {
( 0.00000000e+00,  3.14463285e-01)
( 6.62572764e+01,  2.18968176e-01)
( 9.42103101e+01,  9.30255195e-02)
( 1.75806083e+02,  1.86697205e-02)
( 2.59883307e+02,  1.86697205e-02)
( 3.54083494e+02,  3.09886265e-03)
( 4.41120756e+02,  3.09886265e-03)
( 5.08402811e+02,  3.09886265e-03)
( 5.82680227e+02,  1.04042907e-03)
( 6.18696634e+02,  1.04042907e-03)
( 6.84725429e+02,  1.04042907e-03)
( 7.54544802e+02,  5.94302252e-04)
( 8.22247774e+02,  3.12404821e-04)
( 8.96915785e+02,  6.89046436e-05)
( 9.86401137e+02,  6.89046436e-05)
( 1.05658273e+03,  6.89046436e-05)
( 1.13360699e+03,  6.89046436e-05)
( 1.20329871e+03,  9.04782323e-06)
( 1.28282499e+03,  5.76361457e-06)
( 1.36104601e+03,  5.76361457e-06)
( 1.44698186e+03,  1.83198862e-06)
( 1.52201172e+03,  1.28754173e-06)
( 1.59436218e+03,  7.37558861e-08)
( 1.67379316e+03,  3.07622505e-09)
( 1.75280882e+03,  1.36320284e-10)
( 1.83303721e+03,  4.74463771e-11)
( 1.91453929e+03,  4.70621462e-12)
( 1.99337683e+03,  1.26982763e-12)
( 2.07505006e+03,  1.68182387e-13)
( 2.15577623e+03,  1.20026983e-14)
( 2.25647881e+03,  1.20026983e-14)
( 2.34514123e+03,  9.68892512e-15)
( 2.42609857e+03,  9.68892512e-15)
( 2.50614827e+03,  9.68892512e-15)
( 2.58686679e+03,  8.96587101e-15)
( 2.73807607e+03,  8.96587101e-15)
( 2.81803544e+03,  7.66437361e-15)
( 2.90272006e+03,  5.92904373e-15)
( 2.98697129e+03,  4.77215715e-15)
( 3.07489707e+03,  4.77215715e-15)
( 3.16432095e+03,  4.77215715e-15)
( 3.24989251e+03,  2.60299481e-15)
( 3.34001245e+03,  2.60299481e-15)
( 3.43010041e+03,  2.60299481e-15)
( 3.52076155e+03,  2.60299481e-15)
( 3.61786309e+03,  2.60299481e-15)};\label{line:bypass:kappa6}

\addplot [mark=pentagon*, mark size=2, mark options={solid}, dashed, thick, green, mark repeat=5]
coordinates {
( 0.00000000e+00,  3.14463285e-01)
( 6.60967876e+01,  2.18968207e-01)
( 9.41935216e+01,  9.30247974e-02)
( 1.42958706e+02,  1.95941984e-02)
( 2.09052156e+02,  1.95941984e-02)
( 2.93874862e+02,  9.09189210e-03)
( 3.26830833e+02,  9.09189210e-03)
( 4.10644193e+02,  5.49011082e-03)
( 4.76124128e+02,  5.49011082e-03)
( 5.62457770e+02,  5.49011082e-03)
( 6.47871342e+02,  3.67511117e-03)
( 7.17025811e+02,  1.90267932e-03)
( 7.87506486e+02,  1.90267932e-03)
( 8.56000920e+02,  1.20757395e-03)
( 9.23930212e+02,  1.20757395e-03)
( 9.92440299e+02,  4.31926981e-04)
( 1.06661952e+03,  8.34398457e-05)
( 1.15283001e+03,  8.34398457e-05)
( 1.23936495e+03,  3.61940808e-05)
( 1.32533856e+03,  1.57023343e-05)
( 1.41417143e+03,  2.06250463e-06)
( 1.50360018e+03,  6.78923342e-07)
( 1.59206331e+03,  6.24795712e-08)
( 1.76583639e+03,  6.24795712e-08)
( 1.85694563e+03,  1.82540658e-08)
( 1.93562513e+03,  1.10540485e-08)
( 2.02753630e+03,  2.26754976e-10)
( 2.12121115e+03,  2.47241124e-11)
( 2.21722440e+03,  2.84406105e-12)
( 2.29286482e+03,  2.84406105e-12)
( 2.39018276e+03,  1.46996901e-12)
( 2.49242349e+03,  5.46195078e-13)
( 2.59121277e+03,  9.61661971e-14)
( 2.69404945e+03,  2.61745589e-14)
( 2.78988953e+03,  2.61745589e-14)
( 2.89147047e+03,  5.35060044e-15)
( 3.00058598e+03,  5.35060044e-15)
( 3.09472026e+03,  5.35060044e-15)
( 3.18933704e+03,  5.35060044e-15)
( 3.28700938e+03,  5.35060044e-15)
( 3.38576110e+03,  5.35060044e-15)
( 3.50083951e+03,  5.35060044e-15)
( 3.59706910e+03,  5.35060044e-15)
( 3.69591212e+03,  5.35060044e-15)
( 3.79544350e+03,  5.35060044e-15)};\label{line:bypass:kappa7}

\addplot [mark=triangle, mark size=2, mark options={solid}, black, solid, thick, mark repeat=5]
coordinates {
( 0.00000000e+00,  3.14463285e-01)
( 2.67456985e+02,  3.46730476e-02)
( 5.34913971e+02,  2.32431038e-02)
( 8.02370956e+02,  2.20161595e-02)
( 1.06982794e+03,  1.47224217e-02)
( 1.33728493e+03,  1.34717338e-02)
( 1.60474191e+03,  8.34006444e-03)
( 1.87219890e+03,  3.85489939e-03)
( 2.13965588e+03,  1.51784124e-03)
( 2.40711287e+03,  1.03190432e-03)
( 2.67456985e+03,  7.77961970e-04)
( 2.94202684e+03,  6.90349618e-04)
( 3.20948382e+03,  3.55252119e-04)
( 3.47694081e+03,  3.06576788e-04)
( 3.74439780e+03,  1.89100760e-04)
( 4.01185478e+03,  1.52277321e-04)
( 4.27931177e+03,  1.04037107e-04)
( 4.54676875e+03,  6.84675105e-06)
( 4.81422574e+03,  2.15496340e-05)
( 5.08168272e+03,  1.54376135e-05)
( 5.34913971e+03,  5.12007544e-06)
( 5.61659669e+03,  3.09213090e-06)
( 5.88405368e+03,  2.39977947e-07)
( 6.15151066e+03,  1.57355226e-07)
( 6.41896765e+03,  3.40636722e-10)
( 6.68642463e+03,  1.70854795e-11)
( 6.95388162e+03,  1.66763755e-11)
( 7.22133861e+03,  6.37733728e-13)
( 7.48879559e+03,  6.73163380e-13)
( 7.75625258e+03,  2.74760563e-15)
( 8.02370956e+03,  1.01227576e-15)
( 8.29116655e+03,  0.00000000e+00)};\label{line:bypass:kappa:hdm}

\end{axis}
\end{tikzpicture}
    \caption{Convergence history of $\mathrm{EQP}_{\partial_0}^{(3)}$ method applied to bypass problem with
    no basis truncation for different values of $\kappa$. Legend: 
    $\kappa = 10^{-3}$ (\ref{line:bypass:kappa3}),
    $\kappa = 10^{-4}$ (\ref{line:bypass:kappa4}),
    $\kappa = 10^{-5}$ (\ref{line:bypass:kappa5}),
    $\kappa = 10^{-6}$ (\ref{line:bypass:kappa6}),
    $\kappa = 10^{-7}$ (\ref{line:bypass:kappa7}),
    HDM (\ref{line:bypass:kappa:hdm})
    }
    \label{fig:numexp:bypass:kappa}
\end{figure}

\subsubsection{Influence of the reduced basis size}
Next, we study the influence of the reduced basis size, i.e., $n_k = 2+N_\mubold+p_k+q_k$, on the performance of the
EQP/TR methods.
Even though the method is globally convergent in the limit where $n_k = 2+N_\mubold$ ($p_k=q_k=0$), the optimal choice of
basis size involves a trade-off where larger bases have more predictive capability that potentially allow for better subproblem solutions
and faster convergence at the price of more expensive subproblems. To study this tradeoff, we consider the performance of
the $\mathrm{EQP}_{\partial_0}^{(3)}$ method for five basis truncation strategies (Figure~\ref{fig:numexp:bypass:basis_sz}):
(i) $n_k \leq 20$ ($p_k = q_k \leq 5$), (ii) $n_k \leq 30$ ($p_k = q_k \leq 10$), 
(iii) $n_k \leq 40$ ($p_k = q_k \leq 15$), (iv) $n_k \leq 50$ ($p_k = q_k \leq 20$), and
(v) $n_k \leq \infty $ ($p_k = q_k = k$, i.e., no truncation).
Despite the decreased cost of the subproblem when using a truncated basis, convergence is much
faster when using the more accurate trust-region models that come from larger reduced bases. Only
in the case where the reduced basis is fairly large ($n_k \leq 50$) is the truncation approach competitive
with using all available information. Therefore, in the remainder of this work, we do not truncate the
reduced basis (i.e., we take $p_k=q_k=k$).
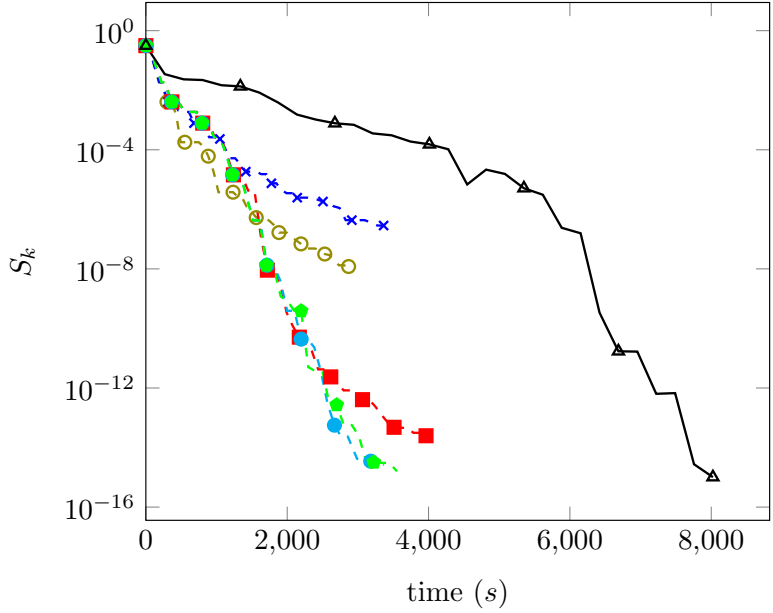
\begin{figure}[H]
    \centering
    \tikzset{every picture/.style={scale=1.1}}
    \begin{tikzpicture}
\begin{axis}[
xmin=0,
ymode=log,
xlabel={time $(s)$},
ylabel={$S_k$}]
\addplot [mark=o, mark size=2, mark options={solid}, dashed, olive, thick, mark repeat=5]
coordinates {
( 0.00000000e+00,  3.14463285e-01)
( 7.22161647e+01,  2.18964731e-01)
( 1.02216971e+02,  9.31033568e-02)
( 1.84780787e+02,  1.86158060e-02)
( 2.17085555e+02,  1.86158060e-02)
( 2.98425917e+02,  4.06542158e-03)
( 3.64887346e+02,  4.06542158e-03)
( 4.29089160e+02,  1.49131383e-03)
( 4.92680141e+02,  1.79526449e-04)
( 5.21122955e+02,  1.79526449e-04)
( 5.51035403e+02,  1.79526449e-04)
( 6.14480613e+02,  1.79526449e-04)
( 6.80785448e+02,  1.79526449e-04)
( 7.48220189e+02,  1.79526449e-04)
( 8.15099227e+02,  1.28687917e-04)
( 8.82939898e+02,  6.10249776e-05)
( 9.57679087e+02,  1.86537173e-05)
( 1.02751268e+03,  3.76664812e-06)
( 1.09762290e+03,  3.76664812e-06)
( 1.16433156e+03,  3.76664812e-06)
( 1.23363837e+03,  3.76664812e-06)
( 1.30054476e+03,  2.66446624e-06)
( 1.35533375e+03,  2.66446624e-06)
( 1.42661637e+03,  1.22219378e-06)
( 1.49421743e+03,  8.40711255e-07)
( 1.56188272e+03,  5.27634161e-07)
( 1.62867392e+03,  5.27634161e-07)
( 1.68220870e+03,  5.27634161e-07)
( 1.75017341e+03,  4.06249763e-07)
( 1.81570252e+03,  2.90862553e-07)
( 1.88329253e+03,  1.65690398e-07)
( 1.93886059e+03,  1.65690398e-07)
( 1.99251717e+03,  1.65690398e-07)
( 2.05931518e+03,  1.17991216e-07)
( 2.12649097e+03,  9.08722321e-08)
( 2.19520897e+03,  6.90840624e-08)
( 2.26329267e+03,  4.83367843e-08)
( 2.33002186e+03,  4.83367843e-08)
( 2.39643102e+03,  4.83367843e-08)
( 2.46265290e+03,  3.96313364e-08)
( 2.53097527e+03,  3.15069166e-08)
( 2.59731087e+03,  2.91361716e-08)
( 2.66491731e+03,  1.92093212e-08)
( 2.73198058e+03,  1.31912024e-08)
( 2.80140512e+03,  1.31912024e-08)
( 2.86889485e+03,  1.20011317e-08)};\label{line:bypass:trunc1}

\addplot [mark=x, mark size=2, mark options={solid}, dashed, blue, thick, mark repeat=5]
coordinates {
( 0.00000000e+00,  3.14463285e-01)
( 6.96813661e+01,  2.18964731e-01)
( 1.01073223e+02,  9.31033568e-02)
( 1.88099993e+02,  1.86158060e-02)
( 2.23586296e+02,  1.86158060e-02)
( 3.12331378e+02,  4.06542158e-03)
( 3.86637389e+02,  4.06542158e-03)
( 4.60345226e+02,  2.08268849e-03)
( 5.35884725e+02,  1.86565882e-03)
( 6.08567332e+02,  1.86565882e-03)
( 6.80874939e+02,  7.85640653e-04)
( 7.52269202e+02,  7.85640653e-04)
( 8.23706785e+02,  5.03567404e-04)
( 8.97328835e+02,  2.63587202e-04)
( 9.72949717e+02,  2.63587202e-04)
( 1.04727641e+03,  2.30941179e-04)
( 1.11759269e+03,  1.14848340e-04)
( 1.18984970e+03,  5.19367614e-05)
( 1.26321258e+03,  5.19367614e-05)
( 1.33793035e+03,  3.29348831e-05)
( 1.41388131e+03,  1.84553861e-05)
( 1.48638934e+03,  1.84553861e-05)
( 1.55747483e+03,  1.84553861e-05)
( 1.63265640e+03,  1.52019669e-05)
( 1.70431799e+03,  1.52019669e-05)
( 1.77640243e+03,  7.49240616e-06)
( 1.84692461e+03,  7.49240616e-06)
( 1.91899514e+03,  4.42516066e-06)
( 1.99183501e+03,  3.48407068e-06)
( 2.06580528e+03,  3.48407068e-06)
( 2.14158216e+03,  2.46491765e-06)
( 2.21629323e+03,  2.46491765e-06)
( 2.28785758e+03,  2.46491765e-06)
( 2.35983027e+03,  2.25531856e-06)
( 2.43387640e+03,  2.09732894e-06)
( 2.50621530e+03,  1.82780190e-06)
( 2.57950914e+03,  1.67290086e-06)
( 2.65355372e+03,  1.27359865e-06)
( 2.74147772e+03,  1.10831657e-06)
( 2.82742478e+03,  4.33061803e-07)
( 2.91458807e+03,  4.33061803e-07)
( 3.00667476e+03,  4.33061803e-07)
( 3.09565440e+03,  4.33061803e-07)
( 3.18227104e+03,  3.63594010e-07)
( 3.27670070e+03,  2.85647551e-07)
( 3.36405979e+03,  2.85647551e-07)};\label{line:bypass:trunc2}

\addplot [mark=square*, mark size=2, mark options={solid}, dashed, red, thick, mark repeat=5]
coordinates {
( 0.00000000e+00,  3.14463285e-01)
( 8.22301153e+01,  2.18964731e-01)
( 1.16289931e+02,  9.31033568e-02)
( 2.19187028e+02,  1.86158060e-02)
( 2.62917556e+02,  1.86158060e-02)
( 3.68232632e+02,  4.06542158e-03)
( 4.57108773e+02,  4.06542158e-03)
( 5.44370106e+02,  2.08268849e-03)
( 6.31728985e+02,  1.86565882e-03)
( 7.18979867e+02,  1.86565882e-03)
( 8.04401865e+02,  7.85640653e-04)
( 8.87478500e+02,  7.85640653e-04)
( 9.73923509e+02,  2.70862147e-04)
( 1.06137216e+03,  2.70862147e-04)
( 1.14861579e+03,  6.69991499e-05)
( 1.23906621e+03,  1.44038287e-05)
( 1.35333323e+03,  1.44038287e-05)
( 1.44717522e+03,  4.26459166e-06)
( 1.53691034e+03,  4.26459166e-06)
( 1.62688177e+03,  1.26503784e-07)
( 1.71935161e+03,  9.05809410e-09)
( 1.80863484e+03,  9.05809410e-09)
( 1.89968237e+03,  3.43889092e-09)
( 1.99120271e+03,  3.42886866e-10)
( 2.08124302e+03,  1.16610841e-10)
( 2.17135705e+03,  5.09796533e-11)
( 2.26161028e+03,  3.70228290e-11)
( 2.34811036e+03,  1.75801931e-11)
( 2.43809591e+03,  4.21743003e-12)
( 2.53025369e+03,  4.21743003e-12)
( 2.61663141e+03,  2.35860252e-12)
( 2.70594911e+03,  1.58565767e-12)
( 2.79778277e+03,  8.31367620e-13)
( 2.88942306e+03,  8.31367620e-13)
( 2.97602064e+03,  5.62680711e-13)
( 3.06545317e+03,  4.05777969e-13)
( 3.15795208e+03,  4.05777969e-13)
( 3.24331588e+03,  2.44681512e-13)
( 3.33174630e+03,  1.26823692e-13)
( 3.42105254e+03,  7.05700815e-14)
( 3.51202057e+03,  4.78661823e-14)
( 3.60850742e+03,  4.59862416e-14)
( 3.69940543e+03,  4.59862416e-14)
( 3.78506905e+03,  3.12359377e-14)
( 3.87531676e+03,  3.12359377e-14)
( 3.96423102e+03,  2.48730615e-14)};\label{line:bypass:trunc3}

\addplot [mark=*, mark size=2, mark options={solid}, dashed, cyan, thick, mark repeat=5]
coordinates {
( 0.00000000e+00,  3.14463285e-01)
( 8.01353357e+01,  2.18964731e-01)
( 1.15326145e+02,  9.31033568e-02)
( 2.17016783e+02,  1.86158060e-02)
( 2.58374768e+02,  1.86158060e-02)
( 3.63736040e+02,  4.06542158e-03)
( 4.50920104e+02,  4.06542158e-03)
( 5.36697509e+02,  2.08268849e-03)
( 6.22584828e+02,  1.86565882e-03)
( 7.10750989e+02,  1.86565882e-03)
( 7.94956621e+02,  7.85640653e-04)
( 8.78754856e+02,  7.85640653e-04)
( 9.65456524e+02,  2.70862147e-04)
( 1.05111531e+03,  2.70862147e-04)
( 1.13802465e+03,  6.69991499e-05)
( 1.22802678e+03,  1.44038287e-05)
( 1.33976714e+03,  1.44038287e-05)
( 1.43084792e+03,  2.41503987e-06)
( 1.52244601e+03,  4.23228041e-07)
( 1.61691431e+03,  4.23228041e-07)
( 1.71311016e+03,  1.32967247e-08)
( 1.80917160e+03,  1.32967247e-08)
( 1.90267337e+03,  3.97908074e-09)
( 1.99986925e+03,  3.91277841e-10)
( 2.10187104e+03,  3.91277841e-10)
( 2.19725899e+03,  4.37673332e-11)
( 2.29215487e+03,  4.37673332e-11)
( 2.38313781e+03,  2.29764906e-11)
( 2.47429612e+03,  3.81251973e-12)
( 2.57069850e+03,  2.99199792e-13)
( 2.66843053e+03,  5.63982209e-14)
( 2.76244164e+03,  2.31377316e-14)
( 2.85833580e+03,  2.31377316e-14)
( 2.98874605e+03,  4.19371386e-15)
( 3.08641380e+03,  3.47065975e-15)
( 3.18169740e+03,  3.47065975e-15)
( 3.27813901e+03,  3.47065975e-15)
( 3.36758077e+03,  2.60299481e-15)};\label{line:bypass:trunc4}

\addplot [mark=pentagon*, mark size=2, mark options={solid}, dashed, green, thick, mark repeat=5]
coordinates {
( 0.00000000e+00,  3.14463285e-01)
( 8.21824761e+01,  2.18964731e-01)
( 1.17103270e+02,  9.31033568e-02)
( 2.23388485e+02,  1.86158060e-02)
( 2.64540822e+02,  1.86158060e-02)
( 3.68510144e+02,  4.06542158e-03)
( 4.54676937e+02,  4.06542158e-03)
( 5.39938043e+02,  2.08268849e-03)
( 6.25302563e+02,  1.86565882e-03)
( 7.11433361e+02,  1.86565882e-03)
( 7.94849289e+02,  7.85640653e-04)
( 8.76915435e+02,  7.85640653e-04)
( 9.62757786e+02,  2.70862147e-04)
( 1.05069613e+03,  2.70862147e-04)
( 1.13523653e+03,  6.69991499e-05)
( 1.22440359e+03,  1.44038287e-05)
( 1.33623795e+03,  1.44038287e-05)
( 1.42674146e+03,  2.41503987e-06)
( 1.51740603e+03,  4.23228041e-07)
( 1.60929611e+03,  4.23228041e-07)
( 1.70489989e+03,  1.32967247e-08)
( 1.80059844e+03,  1.32967247e-08)
( 1.90883345e+03,  1.17766687e-09)
( 2.00502108e+03,  8.43891206e-10)
( 2.09960835e+03,  3.88131544e-10)
( 2.19707612e+03,  3.88131544e-10)
( 2.29724769e+03,  5.03144436e-12)
( 2.39719417e+03,  3.66212448e-12)
( 2.49839096e+03,  3.66212448e-12)
( 2.60039878e+03,  2.71289903e-13)
( 2.70061449e+03,  2.71289903e-13)
( 2.80094126e+03,  5.66874425e-14)
( 2.91474384e+03,  5.66874425e-14)
( 3.01314930e+03,  2.44392290e-14)
( 3.12111639e+03,  3.18143810e-15)
( 3.22590420e+03,  3.18143810e-15)
( 3.33777804e+03,  3.03682728e-15)
( 3.44814115e+03,  3.03682728e-15)
( 3.55157529e+03,  1.59071905e-15)};\label{line:bypass:trunc5}

\addplot [mark=triangle, mark size=2, mark options={solid}, solid, black, thick, mark repeat=5]
coordinates {
( 0.00000000e+00,  3.14463285e-01)
( 2.67456985e+02,  3.46730476e-02)
( 5.34913971e+02,  2.32431038e-02)
( 8.02370956e+02,  2.20161595e-02)
( 1.06982794e+03,  1.47224217e-02)
( 1.33728493e+03,  1.34717338e-02)
( 1.60474191e+03,  8.34006444e-03)
( 1.87219890e+03,  3.85489939e-03)
( 2.13965588e+03,  1.51784124e-03)
( 2.40711287e+03,  1.03190432e-03)
( 2.67456985e+03,  7.77961970e-04)
( 2.94202684e+03,  6.90349618e-04)
( 3.20948382e+03,  3.55252119e-04)
( 3.47694081e+03,  3.06576788e-04)
( 3.74439780e+03,  1.89100760e-04)
( 4.01185478e+03,  1.52277321e-04)
( 4.27931177e+03,  1.04037107e-04)
( 4.54676875e+03,  6.84675105e-06)
( 4.81422574e+03,  2.15496340e-05)
( 5.08168272e+03,  1.54376135e-05)
( 5.34913971e+03,  5.12007544e-06)
( 5.61659669e+03,  3.09213090e-06)
( 5.88405368e+03,  2.39977947e-07)
( 6.15151066e+03,  1.57355226e-07)
( 6.41896765e+03,  3.40636722e-10)
( 6.68642463e+03,  1.70854795e-11)
( 6.95388162e+03,  1.66763755e-11)
( 7.22133861e+03,  6.37733728e-13)
( 7.48879559e+03,  6.73163380e-13)
( 7.75625258e+03,  2.74760563e-15)
( 8.02370956e+03,  1.01227576e-15)
( 8.29116655e+03,  0.00000000e+00)};\label{line:bypass:hdm}

\end{axis}
\end{tikzpicture}
    \caption{Convergence history of $\mathrm{EQP}_{\partial_0}^{(3)}$ method applied to bypass problem with different basis truncation levels and $\kappa=10^{-4}$. Legend: 
    $n_k \leq 20$ (\ref{line:bypass:trunc1}),
    $n_k \leq 30$ (\ref{line:bypass:trunc2}),
    $n_k \leq 40$ (\ref{line:bypass:trunc3}),
    $n_k \leq 50$ (\ref{line:bypass:trunc4}),
    $n_k \leq \infty$ (\ref{line:bypass:trunc5}),
    HDM (\ref{line:bypass:hdm})
    }
    \label{fig:numexp:bypass:basis_sz}
\end{figure}

\subsubsection{Influence of the empirical quadrature procedure constraints}
Finally, we study the influence of the EQP constraints by directly comparing the methods in Table~\ref{tab:numexp:var_model} with the
gradient tolerance $\kappa = 10^{-4}$ and retaining the full reduced basis (no truncation) (Figure~\ref{fig:numexp:bypass:basis_sz}).
Recall that all methods satisfy the criteria in Theorem~\ref{the:qoi_grad_errbnd} and are therefore globally convergent, with
$\mathrm{EQP}^{(1)}$ being lightest method (e.g., fewest EQP constraints and snapshots) that
ensures global convergence. The other EQP variants include additional snapshots or EQP constraints with
the goal of improving the convergence rate.

First, we observe the $\mathrm{EQP}^{(1)}$ and $\mathrm{EQP}_{\partial_0}^{(1)}$ are
converging extremely slowly to the optimal value of the objective function. While they initially provide some
benefit relative to directly using the $\mathrm{HDM}$ for very weak tolerances, the HDM method
converges much faster. This confirms the global convergence result for this problem, but also shows global
convergence is not sufficient for the EQP method to be beneficial, and adding additional snapshots at the
first iteration is not sufficient to overcome the limitation.
 
Next, we add the sensitivity residual constraint and consider the $\mathrm{EQP}^{(2)}$
and $\mathrm{EQP}_{\partial_0}^{(2)}$ methods. The convergence of these methods
is significantly faster than the corresponding $\mathrm{EQP}^{(1)}$ and $\mathrm{EQP}_{\partial_0}^{(1)}$
method (Figure~\ref{fig:numexp:bypass:mthd}) suggesting the sensitivity residual EQP constraint is important to obtain a high-quality
sample mesh in the optimization setting.  Furthermore, the initial sensitivity snapshots do not have a significant
impact on the convergence rate for this problem. 

Next, we incorporate the quantity of interest constraint and consider $\mathrm{EQP}^{(3)}$
and $\mathrm{EQP}_{\partial_0}^{(3)}$ methods. The convergence of these methods
is significantly faster than the corresponding $\mathrm{EQP}^{(2)}$ and $\mathrm{EQP}_{\partial_0}^{(2)}$
method (Figure~\ref{fig:numexp:bypass:mthd}) suggesting the QoI EQP constraint further helps improve the convergence
rate of the hyperreduced trust-region framework. Both methods perform favorably compared to the
HDM-only and ROM-only methods, and the initial sensitivity snapshots provide a slightly
advantage, making $\mathrm{EQP}_{\partial_0}^{(3)}$ the best overall method.

Figure~\ref{fig:numexp:bypass:eqp_split} shows the cost breakdown of $\mathrm{EQP}_{\partial_0}^{(3)}$ as a function
of trust-region iteration.
The cost at each iteration is split into four main sources at $k$th trust-region center: (i) construction of the basis $\Phibold_k$ at the beginning using (\ref{eqn:tr:construct_basis_sens}) (not including snapshot computations, only the compression and orthogonalization), (ii) EQP training in (\ref{eqn:eqp:linprog}) (solution of the linear program and ROM solves required to form the constraint matrix), (iii) trust-region subproblem in (\ref{eqn:tr:subprob}), and (iv) HDM evaluations to compute snapshots and evaluate steps in (\ref{eqn:tr:redu_ratio}). It is clear that (iv) dominates the computational cost at each iteration as the average costs of (i)-(iii) are at least one order smaller than (iv) in the magnitude. The oscillations in the cost to solve the HDM at early iterations are a result of large shape variations that require a different number of nonlinear iterations to converge. The shape change stabilizes when it is close the optimal shape, so the cost of (iv) stabilizes. As the optimizer iterates, the cost of (i)-(ii) trend to increase since the size of $\Phibold_k$ and element usage increase. A larger size of reduced basis increases the cost of solving the linear program in (\ref{eqn:eqp:linprog}) because the
number of constraints scales with $n_k$. Lastly, we observe that for this problem, the cost of solving the trust-region subproblem
and the cost to compute the EQP weights are similar.

\begin{figure}[H]
\centering
\tikzset{every picture/.style={scale=1.1}}
\input{_py/bypass/bypass_diff_mthd.tikz}
\caption{Convergence history of the optimization methods in Table~\ref{tab:numexp:var_model} applied to bypass problem with $\kappa=10^{-4}$ and no basis truncation. Legend:     
HDM (\ref{line:bypass:mthd:hdm}),
ROM (\ref{line:bypass:mthd:rom_adj}),
$\mathrm{ROM}_{\partial_0}$ (\ref{line:bypass:mthd:rom_adj_dU0}),
$\mathrm{EQP}^{(1)}$ (\ref{line:bypass:mthd:eqp_adj}),
$\mathrm{EQP}^{(2)}$ (\ref{line:bypass:mthd:eqp_adj_constr_no_qoi}),
$\mathrm{EQP}^{(3)}$ (\ref{line:bypass:mthd:eqp_adj_constr}),
$\mathrm{EQP}_{\partial_0}^{(1)}$ (\ref{line:bypass:mthd:eqp_adj_dU0}),
$\mathrm{EQP}_{\partial_0}^{(2)}$ (\ref{line:bypass:mthd:eqp_adj_dU0_constr_no_qoi}),
$\mathrm{EQP}_{\partial_0}^{(3)}$ (\ref{line:bypass:mthd:eqp_adj_dU0_constr}).}
\label{fig:numexp:bypass:mthd}
\end{figure}

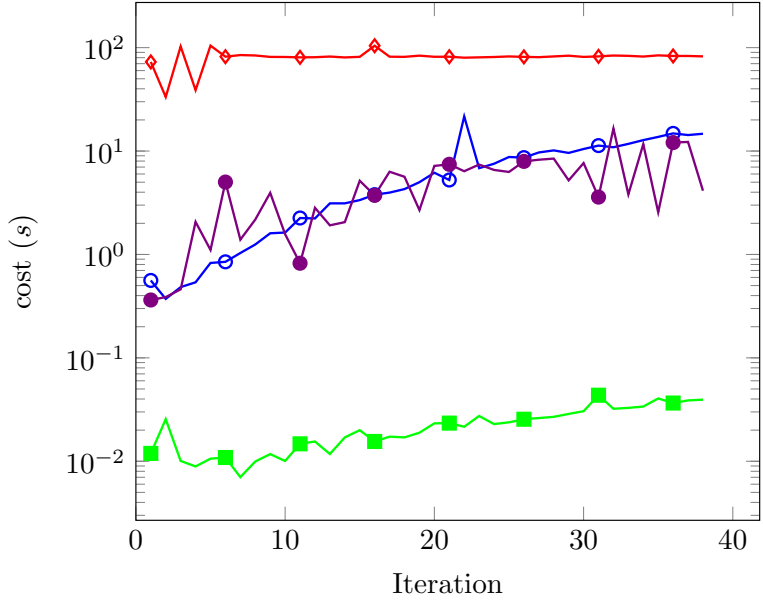
\begin{figure}[H]
\centering
\tikzset{every picture/.style={scale=1.1}}
\begin{tikzpicture}
\begin{axis}[
xmin=0,
ymode=log,
xlabel={Iteration},
ylabel={cost $(s)$}]
\addplot [mark=square*, mark size=2, mark options={solid}, thick, solid, green, mark repeat=5]
coordinates {
( 1.00000000e+00,  1.19093000e-02)
( 2.00000000e+00,  2.55347000e-02)
( 3.00000000e+00,  1.00715000e-02)
( 4.00000000e+00,  8.90860000e-03)
( 5.00000000e+00,  1.05879000e-02)
( 6.00000000e+00,  1.08738000e-02)
( 7.00000000e+00,  7.02060000e-03)
( 8.00000000e+00,  9.94510000e-03)
( 9.00000000e+00,  1.17280000e-02)
( 1.00000000e+01,  1.00605000e-02)
( 1.10000000e+01,  1.47404000e-02)
( 1.20000000e+01,  1.55243000e-02)
( 1.30000000e+01,  1.17548000e-02)
( 1.40000000e+01,  1.70178000e-02)
( 1.50000000e+01,  1.99786000e-02)
( 1.60000000e+01,  1.55364000e-02)
( 1.70000000e+01,  1.72883000e-02)
( 1.80000000e+01,  1.70143000e-02)
( 1.90000000e+01,  1.88912000e-02)
( 2.00000000e+01,  2.32531000e-02)
( 2.10000000e+01,  2.33820000e-02)
( 2.20000000e+01,  2.15582000e-02)
( 2.30000000e+01,  2.74067000e-02)
( 2.40000000e+01,  2.28990000e-02)
( 2.50000000e+01,  2.37917000e-02)
( 2.60000000e+01,  2.54971000e-02)
( 2.70000000e+01,  2.61680000e-02)
( 2.80000000e+01,  2.68799000e-02)
( 2.90000000e+01,  2.86889000e-02)
( 3.00000000e+01,  3.05016000e-02)
( 3.10000000e+01,  4.35232000e-02)
( 3.20000000e+01,  3.21904000e-02)
( 3.30000000e+01,  3.28229000e-02)
( 3.40000000e+01,  3.38104000e-02)
( 3.50000000e+01,  4.05058000e-02)
( 3.60000000e+01,  3.65695000e-02)
( 3.70000000e+01,  3.87528000e-02)
( 3.80000000e+01,  3.94266000e-02)};\label{line:bypass:eqp:basis}

\addplot [mark=o, mark size=2, mark options={solid}, thick, solid, blue, mark repeat=5]
coordinates {
( 1.00000000e+00,  5.60660100e-01)
( 2.00000000e+00,  3.71648900e-01)
( 3.00000000e+00,  4.83221200e-01)
( 4.00000000e+00,  5.37951100e-01)
( 5.00000000e+00,  8.27670500e-01)
( 6.00000000e+00,  8.48463900e-01)
( 7.00000000e+00,  1.03291650e+00)
( 8.00000000e+00,  1.24840120e+00)
( 9.00000000e+00,  1.60268210e+00)
( 1.00000000e+01,  1.62877820e+00)
( 1.10000000e+01,  2.25100690e+00)
( 1.20000000e+01,  2.22928030e+00)
( 1.30000000e+01,  3.12037970e+00)
( 1.40000000e+01,  3.12023660e+00)
( 1.50000000e+01,  3.35489180e+00)
( 1.60000000e+01,  3.78797440e+00)
( 1.70000000e+01,  3.96388890e+00)
( 1.80000000e+01,  4.28857710e+00)
( 1.90000000e+01,  4.99837040e+00)
( 2.00000000e+01,  6.17005430e+00)
( 2.10000000e+01,  5.24801140e+00)
( 2.20000000e+01,  2.16036879e+01)
( 2.30000000e+01,  6.83147110e+00)
( 2.40000000e+01,  7.54009150e+00)
( 2.50000000e+01,  8.76085710e+00)
( 2.60000000e+01,  8.63637170e+00)
( 2.70000000e+01,  9.69047120e+00)
( 2.80000000e+01,  1.01448032e+01)
( 2.90000000e+01,  9.59282920e+00)
( 3.00000000e+01,  1.04680708e+01)
( 3.10000000e+01,  1.12921667e+01)
( 3.20000000e+01,  1.08857552e+01)
( 3.30000000e+01,  1.17329272e+01)
( 3.40000000e+01,  1.27577747e+01)
( 3.50000000e+01,  1.37034223e+01)
( 3.60000000e+01,  1.48174784e+01)
( 3.70000000e+01,  1.42804658e+01)
( 3.80000000e+01,  1.46959773e+01)};\label{line:bypass:eqp:eqp}

\addplot [mark=diamond, mark size=2, mark options={solid}, thick, solid, red, mark repeat=5]
coordinates {
( 1.00000000e+00,  7.28031519e+01)
( 2.00000000e+00,  3.33257269e+01)
( 3.00000000e+00,  1.02744543e+02)
( 4.00000000e+00,  3.88306940e+01)
( 5.00000000e+00,  1.04577947e+02)
( 6.00000000e+00,  8.18030539e+01)
( 7.00000000e+00,  8.47714204e+01)
( 8.00000000e+00,  8.40042452e+01)
( 9.00000000e+00,  8.13733063e+01)
( 1.00000000e+01,  8.12269992e+01)
( 1.10000000e+01,  8.05476317e+01)
( 1.20000000e+01,  8.08165739e+01)
( 1.30000000e+01,  8.20419996e+01)
( 1.40000000e+01,  8.04160472e+01)
( 1.50000000e+01,  8.15000510e+01)
( 1.60000000e+01,  1.04370705e+02)
( 1.70000000e+01,  8.17480387e+01)
( 1.80000000e+01,  8.14586815e+01)
( 1.90000000e+01,  8.36929656e+01)
( 2.00000000e+01,  8.15974730e+01)
( 2.10000000e+01,  8.16895673e+01)
( 2.20000000e+01,  7.99741784e+01)
( 2.30000000e+01,  8.05679868e+01)
( 2.40000000e+01,  8.11805932e+01)
( 2.50000000e+01,  8.22313465e+01)
( 2.60000000e+01,  8.14252741e+01)
( 2.70000000e+01,  8.08853526e+01)
( 2.80000000e+01,  8.22405281e+01)
( 2.90000000e+01,  8.36016644e+01)
( 3.00000000e+01,  8.13744794e+01)
( 3.10000000e+01,  8.23031146e+01)
( 3.20000000e+01,  8.39102443e+01)
( 3.30000000e+01,  8.33092999e+01)
( 3.40000000e+01,  8.20027094e+01)
( 3.50000000e+01,  8.42348920e+01)
( 3.60000000e+01,  8.30990357e+01)
( 3.70000000e+01,  8.30948210e+01)
( 3.80000000e+01,  8.23275734e+01)};\label{line:bypass:eqp:assess}

\addplot [mark=*, mark size=2, mark options={solid}, thick, solid, violet, mark repeat=5]
coordinates {
( 1.00000000e+00,  3.62794800e-01)
( 2.00000000e+00,  3.86905800e-01)
( 3.00000000e+00,  4.59995100e-01)
( 4.00000000e+00,  2.06572560e+00)
( 5.00000000e+00,  1.10563810e+00)
( 6.00000000e+00,  5.01360670e+00)
( 7.00000000e+00,  1.39070140e+00)
( 8.00000000e+00,  2.17535000e+00)
( 9.00000000e+00,  3.94389140e+00)
( 1.00000000e+01,  1.56960180e+00)
( 1.10000000e+01,  8.22365500e-01)
( 1.20000000e+01,  2.82486550e+00)
( 1.30000000e+01,  1.91296380e+00)
( 1.40000000e+01,  2.05337450e+00)
( 1.50000000e+01,  5.16223710e+00)
( 1.60000000e+01,  3.72751340e+00)
( 1.70000000e+01,  6.30858600e+00)
( 1.80000000e+01,  5.63973170e+00)
( 1.90000000e+01,  2.68697770e+00)
( 2.00000000e+01,  7.18246440e+00)
( 2.10000000e+01,  7.44458910e+00)
( 2.20000000e+01,  6.38379740e+00)
( 2.30000000e+01,  7.40266650e+00)
( 2.40000000e+01,  6.56902970e+00)
( 2.50000000e+01,  6.27801730e+00)
( 2.60000000e+01,  7.94534580e+00)
( 2.70000000e+01,  8.24410290e+00)
( 2.80000000e+01,  8.45334450e+00)
( 2.90000000e+01,  5.20586600e+00)
( 3.00000000e+01,  7.65265070e+00)
( 3.10000000e+01,  3.58807190e+00)
( 3.20000000e+01,  1.64994781e+01)
( 3.30000000e+01,  3.80960080e+00)
( 3.40000000e+01,  1.16598541e+01)
( 3.50000000e+01,  2.55969720e+00)
( 3.60000000e+01,  1.20764404e+01)
( 3.70000000e+01,  1.22496040e+01)
( 3.80000000e+01,  4.13549100e+00)};\label{line:bypass:eqp:subprob}

\end{axis}
\end{tikzpicture}
\caption{Cost breakdown at each major iteration for $\mathrm{EQP}_{\partial_0}^{(3)}$ applied to the bypass problem. Legend:     
Build $\Phibold_k$ (i) (\ref{line:bypass:eqp:basis}),
compute the EQP weights (ii) (\ref{line:bypass:eqp:eqp}),
solve the TR subproblem (iii) (\ref{line:bypass:eqp:subprob}), and 
compute snapshots by solving the HDM (iv) (\ref{line:bypass:eqp:assess}).
}
\label{fig:numexp:bypass:eqp_split}
\end{figure}
Table \ref{tab:bypass:hist} presents the convergence history of $\mathrm{EQP}_{\partial_0}^{(3)}$ at selected iterations. The method converges to a first-order critical point as $\norm{\nabla f(\mubold_k)} \rightarrow 0$ in 36 major iterations. After only 25 major iterations, the gradient has been reduced by five orders of magnitude. At earlier iterations, the basis is small and the sample mesh
is highly sparse, which makes the hyperreduced model only match the HDM to a few digits. However, progress toward the optimal solution is still achieved ($\norm{\nabla f(\mubold_k)}$ and $S_k$ decrease on the next iteration) because the convergence criteria for the trust-region method are satisfied by construction of the  $\mathrm{EQP}_{\partial_0}^{(3)}$ method. 
The basis grows at each iteration due to our choice to not apply truncation, and the sample mesh is picking up more
elements because the EQP tolerance in (\ref{eqn:tr:strong_cond}) are tightening as $\norm{\nabla m_k(\mubold_k)}$ decreases.
Figure~\ref{fig:numexp:bypass:samp_msh} shows several samples meshes produced throughout the $\mathrm{EQP}_{\partial_0}^{(3)}$ algorithm. The sampled elements are mainly distributed in the optimization region $\Omega_\mathtt{wd}$ with some elements on the inlet as well.
\begin{table}[H]
    \centering
    \begin{tabular}{c c c c c c c c} 
        \hline
        Iter & $m_k(\mubold_k)$ & $\abs{f(\mubold_k)-m_k(\mubold_k)}$ & $\norm{\nabla f(\mubold_k)} $ & $S_k$ & $\varrho_k$ & $\norm{\rhobold_k}_0 (\%)$& $n_k$\\ [0.5ex] 
        \hline
        1 & 2.0664e+03 & 3.2768e-01 & 1.7795e+03 & 3.1446e-01 & 1.0305e+00 & 4.75 & 10 \\
2 & 1.9174e+03 & 7.6653e-01 & 1.2951e+03 & 2.1896e-01 & 5.9425e-01 & 5.37 & 12 \\
3 & 1.7194e+03 & 6.7779e-01 & 8.0538e+02 & 9.3103e-02 & 1.2742e+00 & 6.86 & 14 \\
4 & 1.6027e+03 & 1.1540e+00 & 3.5247e+02 & 1.8616e-02 & -4.6572e-01 & 7.60 & 16 \\
5 & 1.6021e+03 & 5.0410e-01 & 3.5247e+02 & 1.8616e-02 & 6.8718e-01 & 8.40 & 18 \\
10 & 1.5753e+03 & 6.3705e-02 & 1.2219e+02 & 1.8657e-03 & 1.0078e+00 & 13.32 & 28 \\
15 & 1.5726e+03 & 1.3424e-01 & 2.8562e+01 & 6.6999e-05 & 9.3863e-01 & 17.72 & 38 \\
20 & 1.5724e+03 & 6.4038e-02 & 1.3296e+00 & 4.2323e-07 & 7.9254e-01 & 21.38 & 48 \\
25 & 1.5723e+03 & 1.0269e-03 & 6.5277e-02 & 3.8813e-10 & -2.4107e-01 & 23.16 & 58 \\
30 & 1.5723e+03 & 9.7412e-04 & 1.0526e-03 & 2.7129e-13 & 6.2298e-02 & 27.04 & 68 \\
36 & 1.5723e+03 & 9.9904e-04 & 4.9346e-05 & 3.1814e-15 & 1.0000e-01 & 30.65 & 80 \\

    \end{tabular}
    \caption{Convergence history of the $\mathrm{EQP}_{\partial_0}^{(3)}$ method applied to the bypass problem
      with $\kappa=10^{-4}$ and no basis truncation.}
    \label{tab:bypass:hist}
\end{table}
\begin{figure}[H]
    \centering
    \begin{minipage}[t]{0.49\textwidth}
        \centering
        \includegraphics[width=\textwidth]{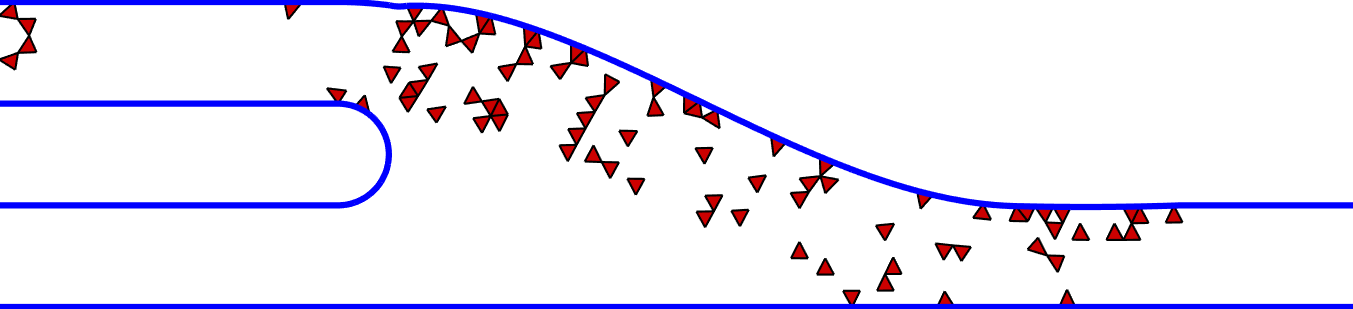}
    \end{minipage}
    \hfill
    \begin{minipage}[t]{0.49\textwidth}
        \centering
        \includegraphics[width=\textwidth]{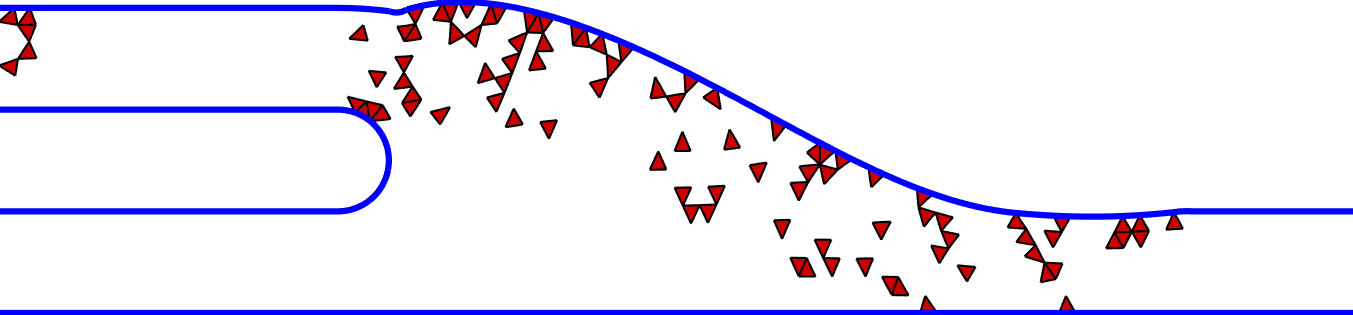}
    \end{minipage}

    \vspace{3.00mm}
    
    \begin{minipage}[t]{0.49\textwidth}
        \centering
        \includegraphics[width=\textwidth]{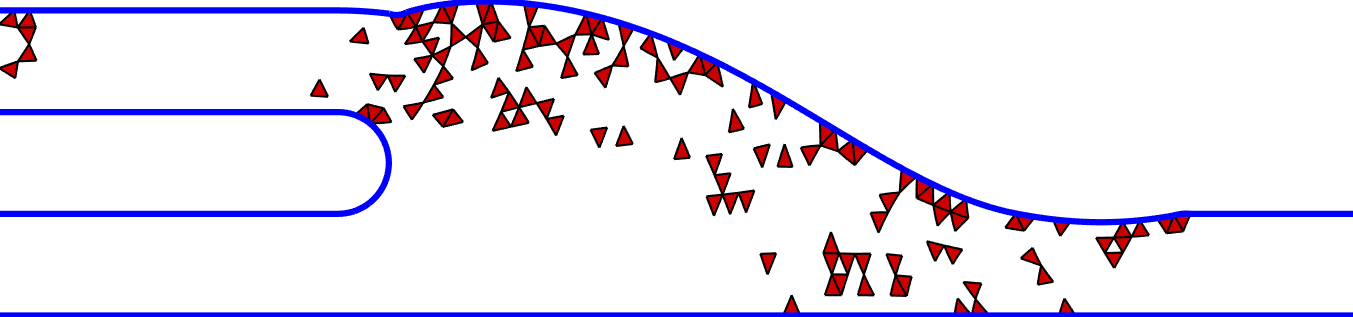}
    \end{minipage}
    \hfill
    \begin{minipage}[t]{0.49\textwidth}
        \centering
        \includegraphics[width=\textwidth]{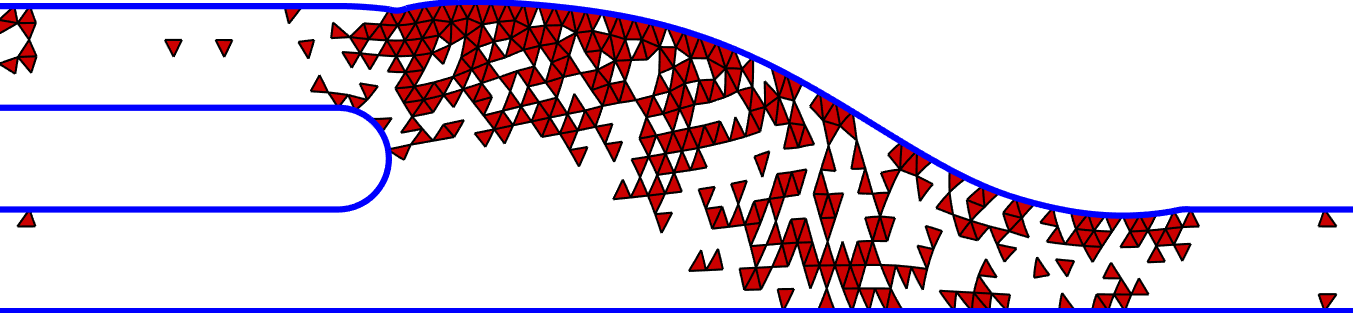}
    \end{minipage}
    \caption{Domain shape and sample mesh at the beginning of major iterations $k=1,2,3$ and 20 (\textit{left-to-right, top-to-bottom}) of
      the $\mathrm{EQP}_{\partial_0}^{(3)}$ method for the bypass problem. Because this problem using
      a continuous finite element discretization, only the nodes associated with the elements in red are required to
      construct the EQP residual ($\tilde\rbm$).}
    \label{fig:numexp:bypass:samp_msh}
\end{figure}

Finally, we compare three methods---$\mathrm{HDM}$, $\mathrm{ROM}$, $\mathrm{EQP}_{\partial_0}^{(3)}$---in terms of the computational cost required to achieve a given cutoff tolerance. We will consider four cutoff levels that represent a loose ($\epsilon=10^{-3}$), moderate ($\epsilon=10^{-5}$), tight ($\epsilon=10^{-8}$), and very tight ($\epsilon=10^{-10}$) accuracy requirement. First we observe the overall speedup of the $\mathrm{EQP}_{\partial_0}^{(3)}$ method is between $2.6-3.5$ for
this problem, which primarily comes from the reduction of HDM solves by a factor of $2-3$, whereas the speedup of the
corresponding ROM method is only between $1.1-1.7$. For the $\mathrm{EQP}_{\partial_0}^{(3)}$ method, the most
substantial speedup comes from the loose and moderate tolerances. The speedup reduction at tighter tolerances comes
from the fact that both the basis and sample mesh are growing at each iteration so EQP training and model evaluations
are becoming more expensive.

\begin{table}[H]
\centering
\begin{tabular}{c|c|c c c c c} 
    \hline
    Method & $\epsilon$ & \# HDM & \# ROM &\# EQP & cost(s) & speedup\\ [0.5ex] 
    \hline
    HDM & 1.0000e-03     & 27       & - & - & 3.1778e+03         & - \\
  & 1.0000e-05     & 51       & - & -  & 5.3083e+03        & - \\
  & 1.0000e-08     & 61       & - & -  & 6.3584e+03        & - \\
  & 1.0000e-10     & 62       & - & -  & 6.4683e+03        & - \\
\hline 
$\mathrm{ROM}$ & 1.0000e-03 &   23 & 149 & - & 2.8280e+03 & 1.12 \\
  & 1.0000e-05    & 29                 & 190    & - & 3.5874e+03         & 1.48 \\
  & 1.0000e-08    & 31                 & 204    & - & 3.8510e+03         & 1.65 \\
  & 1.0000e-10    & 31                 & 204    & - & 3.8510e+03         & 1.68 \\
\hline 
$\mathrm{EQP}_{\partial_0}^{(3)}$ & 1.0000e-03 & 12 & 11 & 51 & 9.0601e+02 & 3.51 \\
  & 1.0000e-05 & 19 & 18 & 90 & 1.5657e+03 & 3.39 \\
  & 1.0000e-08 & 24 & 23 & 122 & 2.0635e+03 & 3.08 \\
  & 1.0000e-10 & 28 & 27 & 149 & 2.4652e+03 & 2.62 \\

\end{tabular}
\caption{Performance comparison for various methods applied to the bypass problem. The speedup is defined as the
 cost of a particular model divided by the cost of the $\mathrm{HDM}$ method at the same cutoff tolerance ($\epsilon$).}
\label{tab:numexp:bypass:hist_cutoff}
\end{table}

\subsection{Inverse design of an airfoil in inviscid, subsonic flow}
\label{sec:numexp:airfoil}
Next, we consider aerodynamic inverse shape design whereby we aim to recover an RAE2822 airfoil from only
its flow field starting from a NACA0012 airfoil. Let $\Omega\subset\Rbb^d$ ($d=2$) be the region around a NACA0012
airfoil with cord length $L=1$ that extends $8L$ from the leading edge, and consider steady, inviscid, compressible
flow governed by the Euler equations
\begin{equation} \label{eqn:numexp:airfoil:euler}
 \pder{}{x_j}\left(\rho v_j\right) = 0, \quad
\pder{}{x_j}\left(\rho v_iv_j+P\delta_{ij}\right) = 0, \quad
\pder{}{x_j}\left(\left[\rho E+P\right]v_j\right) = 0,
\end{equation}
where $i=1,\dots,d$ and summation is implied over
the repeated index $j=1,\dots,d$. The density of the fluid
$\func{\rho}{\Omega}{\Rbb_{>0}}$, the velocity of
the fluid $\func{v_i}{\Omega}{\Rbb}$ in the $x_i$
direction for $i=1,\dots,d$, and the total energy of the fluid
$\func{E}{\Omega}{\Rbb_{>0}}$ are implicitly defined
as the solution of (\ref{eqn:numexp:airfoil:euler}).
For a calorically ideal fluid, the pressure of the fluid,
$\func{P}{\Omega}{\Rbb_{>0}}$, is related to the energy
via the ideal gas law
\begin{equation}
 P = (\gamma-1)\left(\rho E - \frac{\rho v_i v_i}{2}\right),
\end{equation}
where $\gamma\in\Rbb_{>0}$ is the ratio of specific heats. We stack the conservative variables into a 
state vector $U : \Omega \rightarrow \Rbb^{d+2}$ with $U : x \mapsto (\rho(x),\rho(x)v(x), \rho(x)E(x))$.
A farfield boundary condition is applied to the cutoff surface of the domain and a slip wall condition
($v\cdot n = 0$) is applied on the airfoil surface.

The objective of this problem is to match the RAE2822 flow field, denoted $U_\mathrm{RAE2822} : \Omega \rightarrow \Rbb^{d+2}$
by adjusting the airfoil surface, which is parametrized using a Bezier curve with 18 control points and the deformation is extended
to the rest of the domain using linear elasticity (Section~\ref{eqn:hdm:shapeopt:mesh_coord}). The parameter vector $\mubold$ denotes the collection of Bezier
control points. This leads to the following PDE-constrained shape optimization problem 
\begin{equation}\label{eqn:numexp:airfoil:obj_cont}
  \underset{U,\mubold}{\text{minimize}} ~~\frac{1}{2}\int_{\Omega(\mubold)} \norm{U - U_\mathrm{RAE2822}}^2 ~ dV, \quad
  \text{subject to:} ~~ \Lcal(U;\mubold)=0,
\end{equation}
where $\Lcal$ is the differential operator that includes the Euler equations (\ref{eqn:numexp:airfoil:euler}) and appropriate boundary conditions.
The starting point and optimal solution for this problem are shown in Figure~\ref{fig:numexp:airfoil:soln}. Because of this choice of objective
function, the optimal value of the objective function is zero so we use the first definition of $S_k$ in (\ref{eqn:numexp:sk}) throughout
this section.
\begin{figure}[H]
    \centering
    \begin{minipage}[t]{0.49\textwidth}
        \centering
        \includegraphics[width=\textwidth]{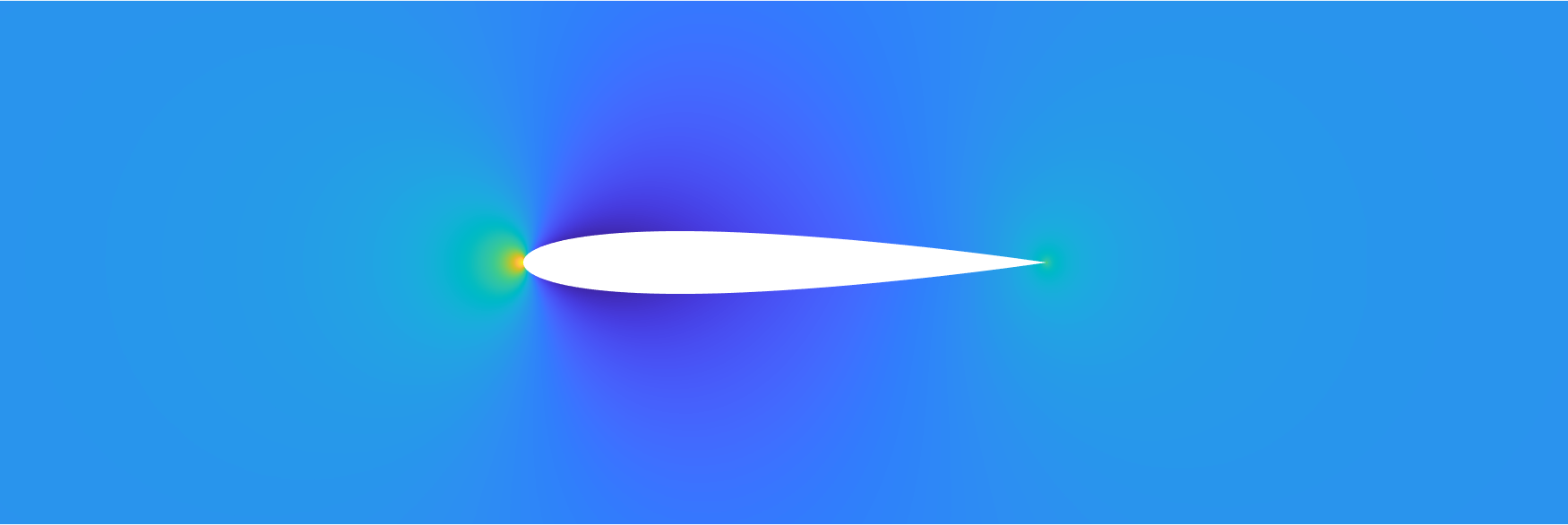}
    \end{minipage}
    \hfill
    \begin{minipage}[t]{0.49\textwidth}
        \centering
        \includegraphics[width=\textwidth]{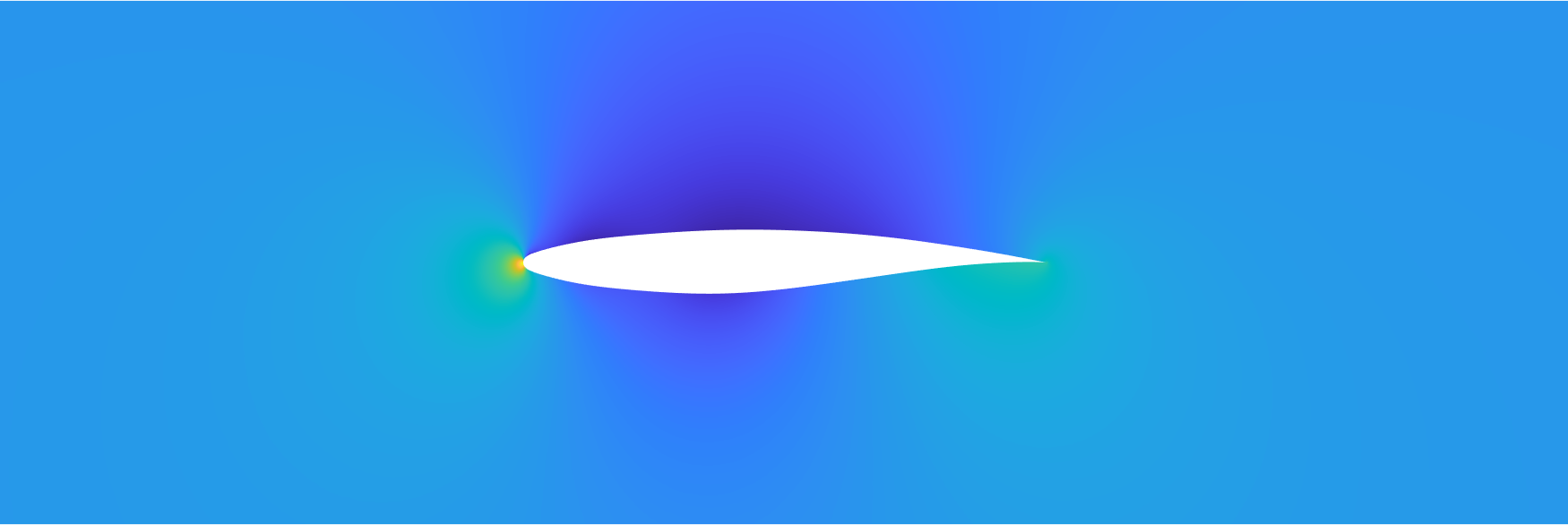}
    \end{minipage}
    \caption{The domain shape and density for the airfoil problem at the starting configuration (\textit{left}) and optimal solution (\textit{right}).}
    \label{fig:numexp:airfoil:soln}
\end{figure}

The governing equations and quantity of interest are discretized using a nodal discontinuous Galerkin method with $\Pcal^2$
triangular elements to yield a discrete optimization problem of the form (\ref{eqn:hdm:opt:full_space}).
A quadratic mesh consisting of $1524$ triangular elements was generated using DistMesh \cite{persson_simple_2004}
and used for all numerical experiments.

At the discrete level, the objective function in (\ref{eqn:numexp:airfoil:obj_cont}) can be written as
\begin{equation}
 j(\ubm, \mubold) = \frac{1}{2}(\ubm-\ubm_\mathrm{RAE2822})^T\Mbm(\ubm-\ubm_\mathrm{RAE2822}),
\end{equation}
where $\Mbm\in\Rbb^{N_\ubm \times N_\ubm}$ is the DG mass matrix and $\ubm_\mathrm{RAE2822}\in\Rbb^{N_\ubm}$ is the
RAE2822 state ($U_\mathrm{RAE2822}$) in algebraic form. Then, the reduced quantity of interest is
\begin{equation}\label{eqn:numexp:airfoil:obj_disc_reduc}
 \hat{j}_\Phibold(\hat\ybm,\mubold) = \frac{1}{2}(\Phibold\hat\ybm-\ubm_\mathrm{RAE2822})^T\Mbm(\Phibold\hat\ybm-\ubm_\mathrm{RAE2822}) = \frac{1}{2} \hat\ybm^T \hat\Mbm \hat\ybm + \hat\bbm^T\hat\ybm + \hat{c},
\end{equation}
where $\hat\Mbm = \Phibold^T\Mbm\Phibold\in\Rbb^{n\times n}$, $\hat\bbm = -\Phibold^T\ubm_\mathrm{RAE2822} \in \Rbb^n$, and $\hat{c} = \frac{1}{2}\norm{\ubm_\mathrm{RAE2822}}^2 \in \Rbb$. These are all small terms that can be precomputed and do not need hyperreduction.
Therefore, in this section, we directly use the form of the QoI in (\ref{eqn:numexp:airfoil:obj_disc_reduc}) and do not include the QoI constraint (because EQP is not responsible for approximating the QoI itself).

Due to the objective function implementation (\ref{eqn:numexp:airfoil:obj_disc_reduc}) and the results of the bypass study
(Section~\ref{sec:numexp:bypass}), we will only consider the $\mathrm{EQP}_{\partial_0}^{(2)}$ with $\hat{\kappa}=10^{-4}$,
$\Delta_0=0.1$, and no basis truncation. From these choices the size of the reduced basis will evolve as $n_k = 20+2k$ (the
20 initial snapshots come from one primal and adjoint solution, and 18 sensitivities at $\mubold_0$). For this problem, the
initial sensitivity snapshots are critical for both the ROM and EQP methods to converge rapidly (Figure \ref{fig:numexp:airfoil:mthd}).
Without the initial sensitivity snapshots, the build up of the reduced basis is too slow and the reduced methods are not competitive
with the HDM. However, with the initial sensitivity snapshots both methods converge to a tight tolerance of at least $10^{-10}$
much faster than using the HDM alone. Furthermore, it is clear than $\mathrm{EQP}_{\partial_0}^{(2)}$ achieves the
best performance unless tolerances below $10^{-10}$ are required.

Figure~\ref{fig:numexp:airfoil:eqp_split} shows cost breakdown of $\mathrm{EQP}_{\partial_0}^{(2)}$ as a function of trust-region iteration broken into the same sources of cost considered in Section~\ref{sec:numexp:bypass} and Figure~\ref{fig:numexp:bypass:eqp_split}. In this case, the cost of computing the EQP weights is roughly constant because
there are relatively few major iterations and more linearly dependent constraints were removed at later iterations
(Remark~\ref{rem:eqp:rm_lin_depend}). Unlike the bypass case, the cost of solving the trust-region subproblem
is noticeably larger than the cost to construct the EQP weights.

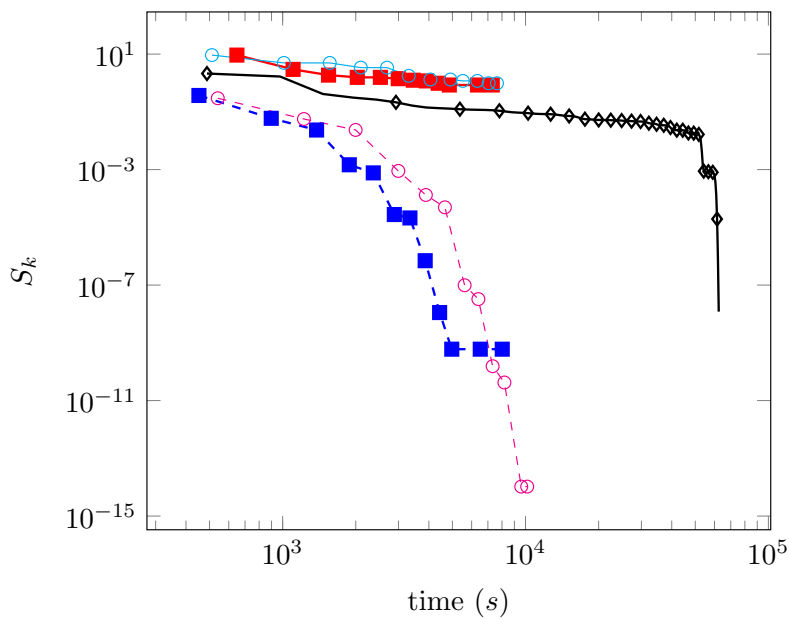
\begin{figure}[H]
    \centering
    \tikzset{every picture/.style={scale=1.1}}
    \begin{tikzpicture}
\begin{axis}[
xmin=0,
ymode=log,
xmode=log,
xlabel={time $(s)$},
ylabel={$S_k$}]
\addplot [mark=square*, mark size=2, mark options={solid}, thick, dashed, blue, mark repeat=1]
coordinates {
( 0.00000000e+00,  9.24206209e+00)
( 4.52190202e+02,  3.69282200e-01)
( 8.97551782e+02,  5.97418131e-02)
( 1.37955127e+03,  2.34921243e-02)
( 1.88142770e+03,  1.46751255e-03)
( 2.36149059e+03,  7.69482824e-04)
( 2.88349008e+03,  2.77659544e-05)
( 3.34213539e+03,  2.11613172e-05)
( 3.86305086e+03,  6.99021224e-07)
( 4.42923522e+03,  1.11849557e-08)
( 4.97008005e+03,  6.02402908e-10)
( 6.51225684e+03,  6.02402908e-10)
( 8.02036662e+03,  6.02402908e-10)};\label{line:airfoil:mthd:eqp_adj_dU0_constr_no_qoi}

\addplot [mark=square*, mark size=2, mark options={solid}, thick, solid, red, mark repeat=1]
coordinates {
( 0.00000000e+00,  9.24206209e+00)
( 6.48179382e+02,  9.24206209e+00)
( 1.10409040e+03,  2.95802620e+00)
( 1.54512554e+03,  1.87268251e+00)
( 2.02936424e+03,  1.56426548e+00)
( 2.52642008e+03,  1.56426548e+00)
( 2.99355470e+03,  1.40931499e+00)
( 3.44977385e+03,  1.24564082e+00)
( 3.90635261e+03,  1.17964605e+00)
( 4.36801518e+03,  9.77846511e-01)
( 4.85522245e+03,  8.51069058e-01)
( 6.34459943e+03,  8.51069058e-01)
( 7.29712809e+03,  8.51069058e-01)};\label{line:airfoil:mthd:eqp_adj_constr_no_qoi}

\addplot [mark=o, mark size=2, mark options={solid}, solid, cyan, mark repeat=1]
coordinates {
( 0.00000000e+00,  9.24206209e+00)
( 5.11434843e+02,  9.24206209e+00)
( 1.01186679e+03,  4.94865419e+00)
( 1.56322165e+03,  4.94865419e+00)
( 2.09829018e+03,  3.40035887e+00)
( 2.69125116e+03,  3.40035887e+00)
( 3.30513702e+03,  1.76838809e+00)
( 4.06000709e+03,  1.31712760e+00)
( 4.90772543e+03,  1.31712760e+00)
( 5.52059024e+03,  1.17929507e+00)
( 6.33437214e+03,  1.17929507e+00)
( 7.05153283e+03,  9.78559763e-01)
( 7.63194863e+03,  9.78559763e-01)};\label{line:airfoil:mthd:rom_adj}

\addplot [mark=o, mark size=2, mark options={solid}, dashed, magenta, mark repeat=1]
coordinates {
( 0.00000000e+00,  9.24206209e+00)
( 5.40617246e+02,  2.95947825e-01)
( 1.22349301e+03,  5.64515624e-02)
( 1.99781263e+03,  2.36381364e-02)
( 2.99415924e+03,  8.98551562e-04)
( 3.88837267e+03,  1.32095815e-04)
( 4.65750687e+03,  4.92426865e-05)
( 5.61517634e+03,  9.85586373e-08)
( 6.39609485e+03,  3.27147962e-08)
( 7.32227903e+03,  1.54856705e-10)
( 8.18643415e+03,  4.22556447e-11)
( 9.60293168e+03,  1.04780096e-14)
( 1.01615437e+04,  1.04780096e-14)};\label{line:airfoil:mthd:rom_adj_dU0}

\addplot [mark=diamond, mark size=2, mark options={solid}, thick, black, mark repeat=5]
coordinates {
( 0.00000000e+00,  9.24206209e+00)
( 4.87990213e+02,  2.12182403e+00)
( 9.75980426e+02,  1.67515604e+00)
( 1.46397064e+03,  4.15887935e-01)
( 1.95196085e+03,  3.11084902e-01)
( 2.43995106e+03,  2.64584223e-01)
( 2.92794128e+03,  2.13794577e-01)
( 3.41593149e+03,  1.64578524e-01)
( 3.90392170e+03,  1.41500539e-01)
( 4.39191192e+03,  1.34266235e-01)
( 4.87990213e+03,  1.29189300e-01)
( 5.36789234e+03,  1.23480856e-01)
( 5.85588256e+03,  1.19637257e-01)
( 6.34387277e+03,  1.17806353e-01)
( 6.83186298e+03,  1.16410872e-01)
( 7.31985319e+03,  1.13804009e-01)
( 7.80784341e+03,  1.09001967e-01)
( 8.29583362e+03,  1.02395133e-01)
( 8.78382383e+03,  9.70111735e-02)
( 9.27181405e+03,  9.42430467e-02)
( 9.75980426e+03,  9.24422268e-02)
( 1.02477945e+04,  8.99622458e-02)
( 1.07357847e+04,  8.69654852e-02)
( 1.12237749e+04,  8.48930781e-02)
( 1.17117651e+04,  8.40828894e-02)
( 1.21997553e+04,  8.36221398e-02)
( 1.26877455e+04,  8.27684924e-02)
( 1.31757358e+04,  8.10168302e-02)
( 1.36637260e+04,  7.81049485e-02)
( 1.41517162e+04,  7.50105331e-02)
( 1.46397064e+04,  7.30746479e-02)
( 1.51276966e+04,  7.19597687e-02)
( 1.56156868e+04,  7.05566929e-02)
( 1.61036770e+04,  6.77420489e-02)
( 1.65916672e+04,  6.31264206e-02)
( 1.70796575e+04,  5.83045123e-02)
( 1.75676477e+04,  5.58331644e-02)
( 1.80556379e+04,  5.50797004e-02)
( 1.85436281e+04,  5.46128379e-02)
( 1.90316183e+04,  5.38293090e-02)
( 1.95196085e+04,  5.29439619e-02)
( 2.00075987e+04,  5.23904324e-02)
( 2.04955889e+04,  5.22164339e-02)
( 2.09835792e+04,  5.21338631e-02)
( 2.14715694e+04,  5.19874021e-02)
( 2.19595596e+04,  5.17251695e-02)
( 2.24475498e+04,  5.13795073e-02)
( 2.29355400e+04,  5.11011057e-02)
( 2.34235302e+04,  5.09276330e-02)
( 2.39115204e+04,  5.07503040e-02)
( 2.43995106e+04,  5.04116203e-02)
( 2.48875009e+04,  4.97968035e-02)
( 2.53754911e+04,  4.89936949e-02)
( 2.58634813e+04,  4.84106442e-02)
( 2.63514715e+04,  4.81602124e-02)
( 2.68394617e+04,  4.80050545e-02)
( 2.73274519e+04,  4.77415792e-02)
( 2.78154421e+04,  4.72927085e-02)
( 2.83034324e+04,  4.67696407e-02)
( 2.87914226e+04,  4.64451698e-02)
( 2.92794128e+04,  4.63132528e-02)
( 2.97674030e+04,  4.62058683e-02)
( 3.02553932e+04,  4.59610031e-02)
( 3.07433834e+04,  4.53989806e-02)
( 3.12313736e+04,  4.42053389e-02)
( 3.17193638e+04,  4.23197613e-02)
( 3.22073541e+04,  4.05387350e-02)
( 3.26953443e+04,  3.97039317e-02)
( 3.31833345e+04,  3.94082073e-02)
( 3.36713247e+04,  3.91115077e-02)
( 3.41593149e+04,  3.84349502e-02)
( 3.46473051e+04,  3.72014357e-02)
( 3.51352953e+04,  3.56103926e-02)
( 3.56232855e+04,  3.46025404e-02)
( 3.61112758e+04,  3.43449174e-02)
( 3.65992660e+04,  3.43017089e-02)
( 3.70872562e+04,  3.42614320e-02)
( 3.75752464e+04,  3.41403820e-02)
( 3.80632366e+04,  3.38514911e-02)
( 3.85512268e+04,  3.31216199e-02)
( 3.90392170e+04,  3.15050222e-02)
( 3.95272073e+04,  2.85982781e-02)
( 4.00151975e+04,  2.53103678e-02)
( 4.05031877e+04,  2.35803169e-02)
( 4.09911779e+04,  2.32587417e-02)
( 4.14791681e+04,  2.32397320e-02)
( 4.19671583e+04,  2.32384472e-02)
( 4.24551485e+04,  2.32358573e-02)
( 4.29431387e+04,  2.32289283e-02)
( 4.34311290e+04,  2.32110149e-02)
( 4.39191192e+04,  2.31643412e-02)
( 4.44071094e+04,  2.30449546e-02)
( 4.48950996e+04,  2.27494174e-02)
( 4.53830898e+04,  2.20793239e-02)
( 4.58710800e+04,  2.08425832e-02)
( 4.63590702e+04,  1.93604208e-02)
( 4.68470604e+04,  1.85139924e-02)
( 4.73350507e+04,  1.83382591e-02)
( 4.78230409e+04,  1.83259140e-02)
( 4.83110311e+04,  1.83244850e-02)
( 4.87990213e+04,  1.83206151e-02)
( 4.92870115e+04,  1.83112481e-02)
( 4.97750017e+04,  1.82860131e-02)
( 5.02629919e+04,  1.82209019e-02)
( 5.07509822e+04,  1.80511531e-02)
( 5.12389724e+04,  1.76172498e-02)
( 5.17269626e+04,  1.65442138e-02)
( 5.22149528e+04,  1.41182956e-02)
( 5.27029430e+04,  9.67125393e-03)
( 5.31909332e+04,  4.41371937e-03)
( 5.36789234e+04,  1.47423616e-03)
( 5.41669136e+04,  8.83231377e-04)
( 5.46549039e+04,  8.46917533e-04)
( 5.51428941e+04,  8.46273416e-04)
( 5.56308843e+04,  8.46250085e-04)
( 5.61188745e+04,  8.46158989e-04)
( 5.66068647e+04,  8.45956292e-04)
( 5.70948549e+04,  8.45390015e-04)
( 5.75828451e+04,  8.43945483e-04)
( 5.80708353e+04,  8.40142798e-04)
( 5.85588256e+04,  8.30324291e-04)
( 5.90468158e+04,  8.05262898e-04)
( 5.95348060e+04,  7.44059955e-04)
( 6.00227962e+04,  6.09435034e-04)
( 6.05107864e+04,  3.77582302e-04)
( 6.09987766e+04,  1.33737650e-04)
( 6.14867668e+04,  1.93909096e-05)
( 6.19747570e+04,  9.06822625e-07)
( 6.24627473e+04,  1.20376011e-08)};\label{line:airfoil:mthd:hdm}

\end{axis}
\end{tikzpicture}
    \caption{Convergence history of the optimization methods in Table~\ref{tab:numexp:var_model} applied to airfoil problem with $\kappa=10^{-4}$ and no basis truncation. Legend:
        HDM (\ref{line:airfoil:mthd:hdm}),
        ROM (\ref{line:airfoil:mthd:rom_adj}),
        $\mathrm{ROM}_{\partial_0}$ (\ref{line:airfoil:mthd:rom_adj_dU0}),
        $\mathrm{EQP}^{(2)}$ (\ref{line:airfoil:mthd:eqp_adj_constr_no_qoi}),
        $\mathrm{EQP}_{\partial_0}^{(2)}$ (\ref{line:airfoil:mthd:eqp_adj_dU0_constr_no_qoi})
    }
    \label{fig:numexp:airfoil:mthd}
\end{figure} 

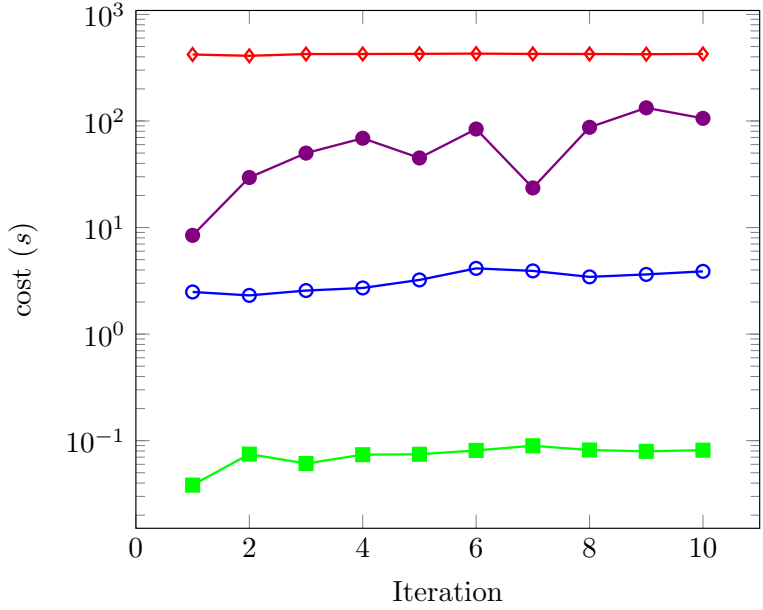
\begin{figure}[H]
\centering
\tikzset{every picture/.style={scale=1.1}}
\begin{tikzpicture}
\begin{axis}[
xmin=0,
ymode=log,
xlabel={Iteration},
ylabel={cost $(s)$}]
\addplot [mark=square*, mark size=2, mark options={solid}, thick, solid, green, mark repeat=1]
coordinates {
( 1.00000000e+00,  3.82564000e-02)
( 2.00000000e+00,  7.48130000e-02)
( 3.00000000e+00,  6.10072000e-02)
( 4.00000000e+00,  7.38680000e-02)
( 5.00000000e+00,  7.45318000e-02)
( 6.00000000e+00,  8.09778000e-02)
( 7.00000000e+00,  8.95530000e-02)
( 8.00000000e+00,  8.17070000e-02)
( 9.00000000e+00,  7.94215000e-02)
( 1.00000000e+01,  8.14176000e-02)};\label{line:airfoil:eqp:basis}

\addplot [mark=o, mark size=2, mark options={solid}, thick, solid, blue, mark repeat=1]
coordinates {
( 1.00000000e+00,  2.48558850e+00)
( 2.00000000e+00,  2.30776260e+00)
( 3.00000000e+00,  2.56228520e+00)
( 4.00000000e+00,  2.71057690e+00)
( 5.00000000e+00,  3.22422360e+00)
( 6.00000000e+00,  4.13790720e+00)
( 7.00000000e+00,  3.91901600e+00)
( 8.00000000e+00,  3.44404850e+00)
( 9.00000000e+00,  3.63225350e+00)
( 1.00000000e+01,  3.88026730e+00)};\label{line:airfoil:eqp:eqp}

\addplot [mark=diamond, mark size=2, mark options={solid}, thick, solid, red, mark repeat=1]
coordinates {
( 1.00000000e+00,  4.20478549e+02)
( 2.00000000e+00,  4.08948073e+02)
( 3.00000000e+00,  4.24801173e+02)
( 4.00000000e+00,  4.25192003e+02)
( 5.00000000e+00,  4.26535594e+02)
( 6.00000000e+00,  4.28696529e+02)
( 7.00000000e+00,  4.25654561e+02)
( 8.00000000e+00,  4.25061208e+02)
( 9.00000000e+00,  4.23281390e+02)
( 1.00000000e+01,  4.25786661e+02)};\label{line:airfoil:eqp:assess}

\addplot [mark=*, mark size=2, mark options={solid}, thick, solid, violet, mark repeat=1]
coordinates {
( 1.00000000e+00,  8.44474530e+00)
( 2.00000000e+00,  2.94757296e+01)
( 3.00000000e+00,  4.99595533e+01)
( 4.00000000e+00,  6.88638806e+01)
( 5.00000000e+00,  4.49934293e+01)
( 6.00000000e+00,  8.40687507e+01)
( 7.00000000e+00,  2.35428686e+01)
( 8.00000000e+00,  8.72451405e+01)
( 9.00000000e+00,  1.32883061e+02)
( 1.00000000e+01,  1.05872132e+02)};\label{line:airfoil:eqp:subprob}

\end{axis}
\end{tikzpicture}
\caption{Cost breakdown at each major iteration for $\mathrm{EQP}_{\partial_0}^{(2)}$ applied to the airfoil problem. Legend:     
Build $\Phibold_k$ (i) (\ref{line:airfoil:eqp:basis}),
compute the EQP weights (ii) (\ref{line:airfoil:eqp:eqp}),
solve the TR subproblem (iii) (\ref{line:airfoil:eqp:subprob}), and 
compute snapshots by solving the HDM (iv) (\ref{line:airfoil:eqp:assess}).
}
\label{fig:numexp:airfoil:eqp_split}
\end{figure}

Finally, we compare three methods---HDM, $\mathrm{ROM}_{\partial_0}$, $\mathrm{EQP}_{\partial_0}^{(2)}$---in terms of the
computational cost required to achieve a given cutoff tolerance $S_k$ (Table \ref{tab:numexp:airfoil:hist_cutoff}). Because the HDM
uses pseudo-transient continuation to compute the CFD solution at each iteration, whereas the $\mathrm{ROM}_{\partial_0}$ and
$\mathrm{EQP}_{\partial_0}^{(2)}$ methods directly use Newton's method with initial guesses coming from the snapshots
(Remark~\ref{rem:ptc_vs_newt}), both methods provide speedup relative to the HDM method over an order of magnitude.
For this problem,
the speedup of the $\mathrm{ROM}_{\partial_0}$ method is between $9.7-14$ and the $\mathrm{EQP}_{\partial_0}^{(2)}$
method is $12.5-18.8$. Similar to the bypass case, the best speedup comes from the looser tolerances.
Figure~\ref{fig:numexp:airfoil:samp_msh} shows several samples meshes produced throughout the
$\mathrm{EQP}_{\partial_0}^{(2)}$ algorithm; as expected, the sampled elements are mainly distributed around the airfoil.

\begin{table}[H]
\centering
\begin{tabular}{c|c|c c c c c} 
\hline
Method & $S_k$ & \# HDM & \# ROM &\# EQP & cost(s) & Speed-up\\ [0.5ex] 
\hline
HDM & 1.0000e-03 & 114 & - & - & 5.4567e+04 & - \\
  & 1.0000e-05 & 130 & - & - & 6.2002e+04 & - \\
  & 1.0000e-07 & 131 & - & - & 6.2463e+04 & - \\
\hline 
$\mathrm{ROM}_{\partial_0}$ & 1.0000e-03 & 6 & 109 & - & 3.9055e+03 & 13.97 \\
  &  1.0000e-05 & 9 & 186 & - & 6.4132e+03 & 9.67 \\
  &  1.0000e-07 & 9 & 186 & - & 6.4132e+03 & 9.74 \\
\hline 
$\mathrm{EQP}_{\partial_0}^{(2)}$ & 1.0000e-03 & 7 & 6 & 126 & 2.8997e+03 & 18.82 \\
  & 1.0000e-05 & 10 & 9 & 216 & 4.4455e+03 & 13.95 \\
  & 1.0000e-07 & 11 & 10 & 258 & 4.9863e+03 & 12.53 \\

\end{tabular}
\caption{Comparison of the performance based on cutoff tolerance}
\label{tab:numexp:airfoil:hist_cutoff}
\end{table}

\begin{figure}[H]
\centering
\begin{minipage}[t]{0.329\textwidth}
\centering
\includegraphics[width=\textwidth]{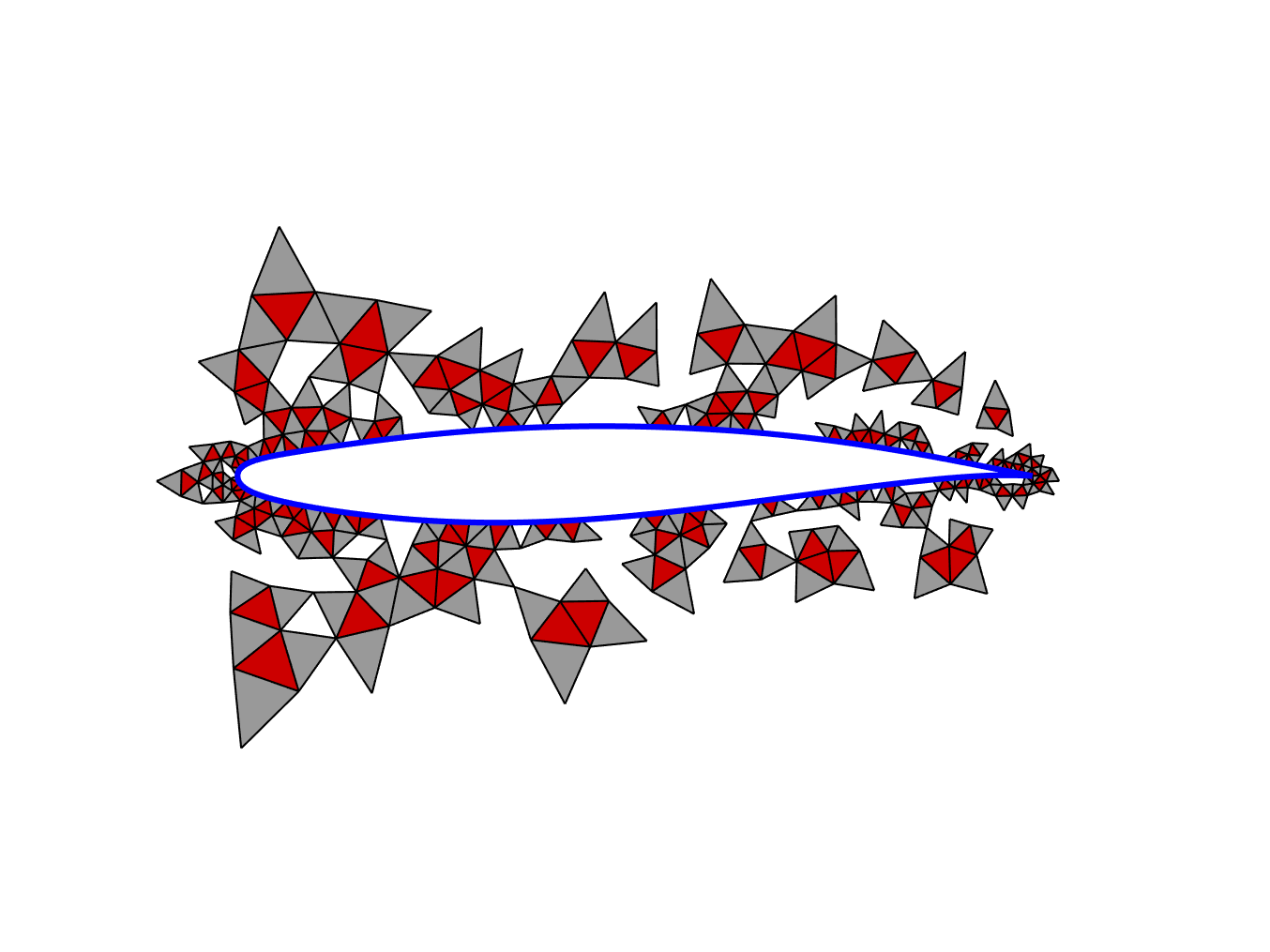}
\end{minipage}
\hfill
\begin{minipage}[t]{0.329\textwidth}
\centering
\includegraphics[width=\textwidth]{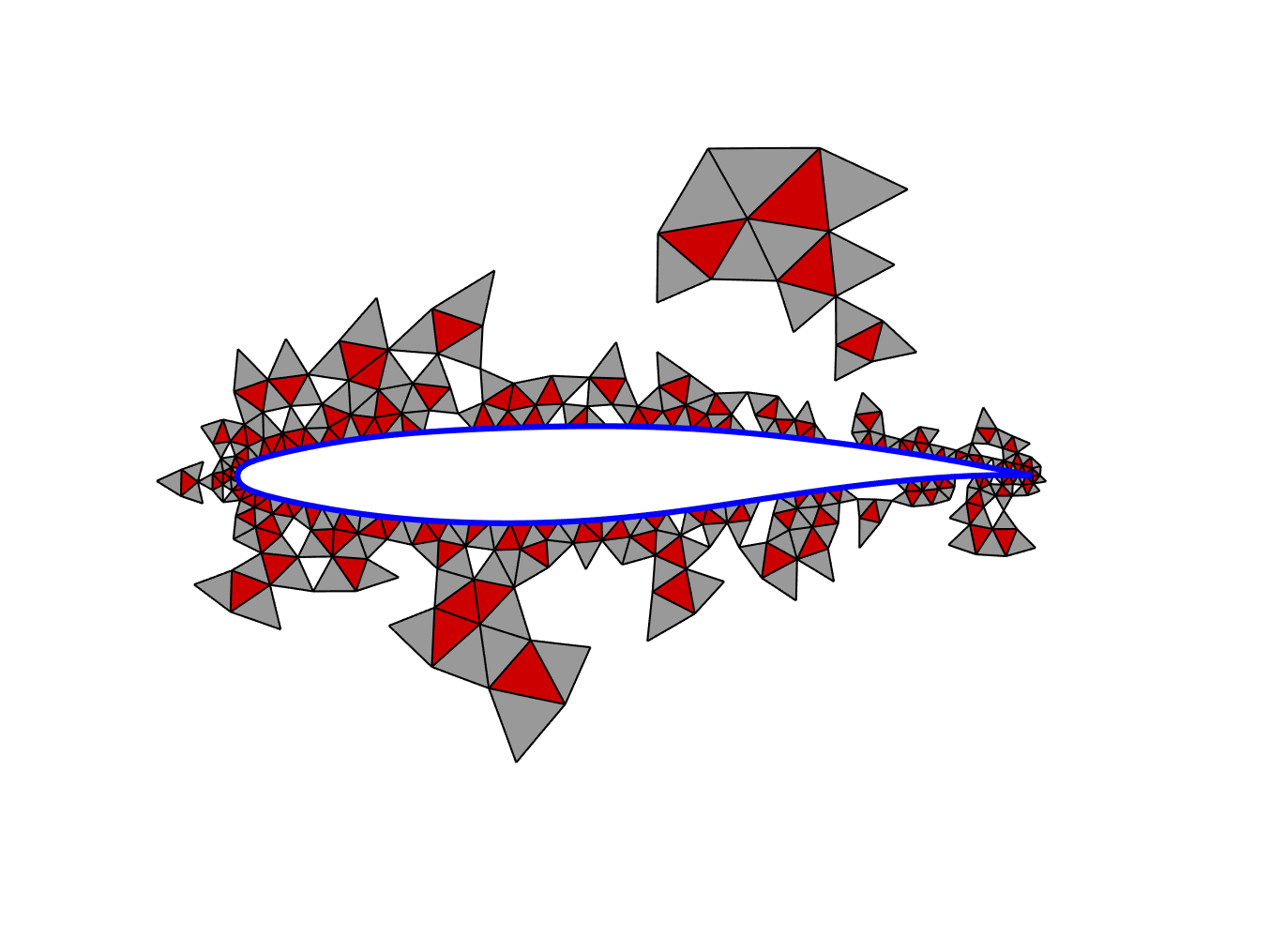}
\end{minipage}
\hfill
\begin{minipage}[t]{0.329\textwidth}
\centering
\includegraphics[width=\textwidth]{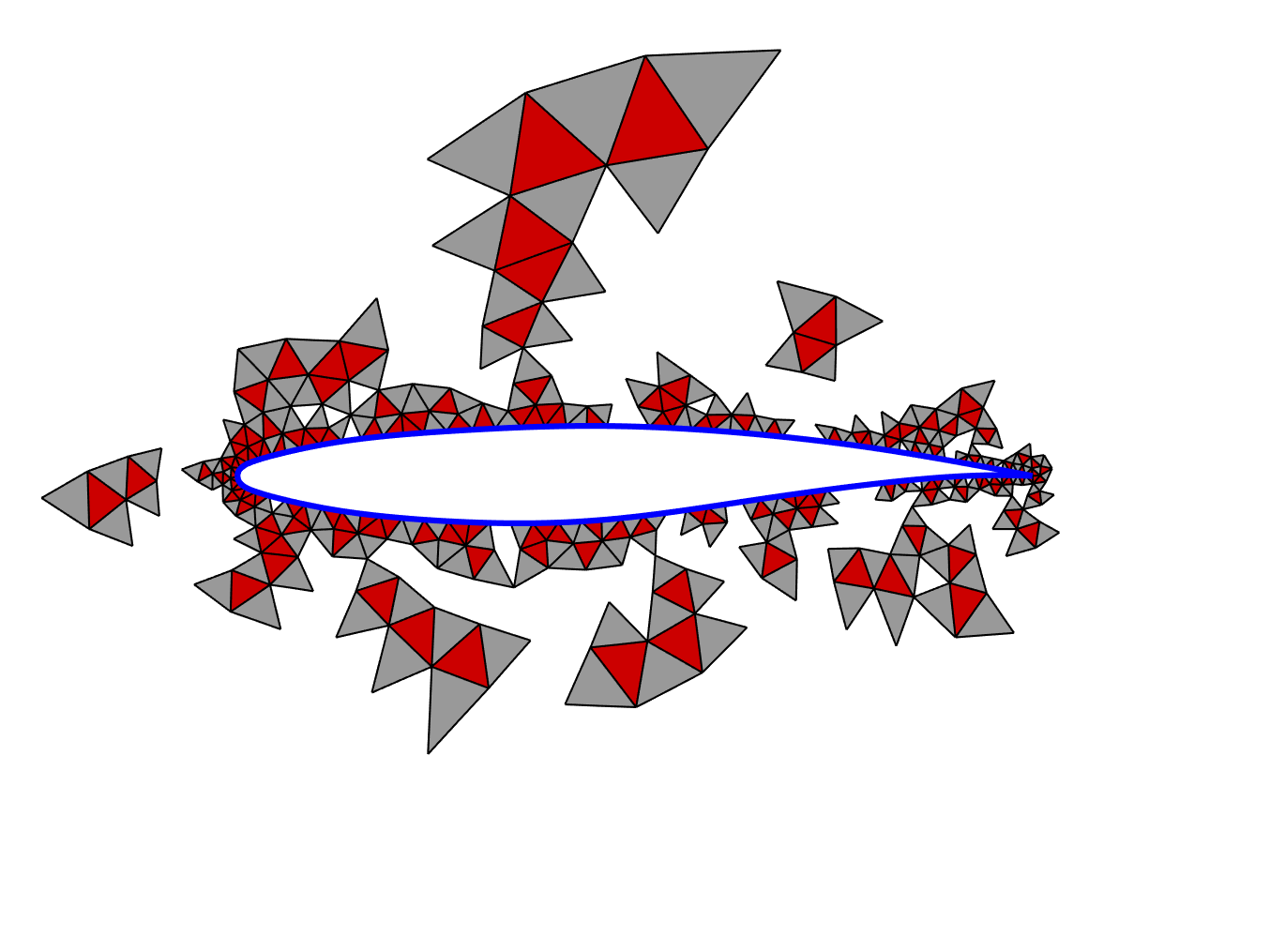}
\end{minipage}
\caption{Domain shape and sample mesh at the beginning of major iterations $k=1,2,3$ (\textit{left-to-right}) of the $\mathrm{EQP}_{\partial_0}^{(2)}$ method for the airfoil problem. The red elements correspond to nonzero weights in $\rhobold_k$, although the
gray elements are also included in the sample mesh since the solution on these elements (at least on their faces that neighbor
the red elements) are required to compute the DG residual contribution in the red elements.}
\label{fig:numexp:airfoil:samp_msh}
\end{figure}

\section{Conclusion}
\label{sec:concl}
In this work, we introduce a numerical method to efficiently solve optimization problems governed by large-scale nonlinear systems
using projection-based reduced-order models accelerated with hyperreduction (empirical quadrature) and globalized with a
trust-region method based on inexact gradient evaluations and asymptotic error bounds \cite{kouri_trust-region_2013}. The proposed method is globally
convergent because hyperreduced models are constructed at each trust-region center precisely to satisfy the conditions
for global convergence. In addition to the primary contribution of the EQP/TR method, this work also contributes:
(i) an empirical quadrature procedure with additional constraints tailored to the optimization setting,
(ii) an approach to accelerate physics-based mesh motion (linear elasticity, springs system, etc) using linear model reduction, and
(iii) global convergence theory for the proposed EQP/TR method.

Two fluid shape optimization problems are employed to verify global convergence of the method and
demonstrate the efficiency of the method; speedups over 18 are attained relative to standard optimization approaches
(even when accounting for all sources of computational cost, i.e., HDM evaluations to gather snapshots and assess trust-region
steps, reduced basis construction, linear program solves to construct EQP weights, trust-region subproblems solves, etc).
Several studies were performed to understand the sensitivity of the method relative to algorithmic parameters, e.g.,
the gradient tolerance, the basis truncation size, and the EQP constraints to include. We found mild sensitivity
with respect to the gradient tolerance, although a moderate value of $\hat\kappa/\kappa=10^{-4}$ proved effective
throughout. Furthermore, there seemed to be little benefit to truncating the snapshot matrix and including
all optimization-based EQP constraints was most effective. Currently, the dominant cost of the proposed
method comes from the HDM evaluations. This suggests further
investigation into hyperreduced models with better predictive capabilities could reduce the frequency
in which the model is reconstructed and further improve the computational efficiency of the method.
Other interesting research directions include the extension to unsteady problems, extension to quadrature-based
(instead of element-based) EQP \cite{du_adaptive_2021,dua2022efficient} for integration with higher order finite element
discretizations, and using the methodology to solve relevant optimization problems.

\section*{Acknowledgments}
This material is based upon work supported by the Air Force Office of
Scientific Research (AFOSR) under award numbers FA9550-20-1-0236
and FA9550-22-1-0004. The content of this publication does not necessarily
reflect the position or policy of any of these supporters, and no official endorsement
should be inferred.

\appendix
\section{Regularity and boundedness assumptions}
\label{sec:appendix_A}
We begin by stating a series of regularity and boundedness assumptions on both the HDM and hyperreduced model. These
assumptions were introduced in previous work \cite{zahr_efficient_2019} and will be used to derive residual-based error estimates.
\begin{assume}\label{assum:hdm}
Consider any open, bounded subset $\Ucal\subset\Rbb^{N_\ubm}$. We assume the HDM
 residual function in (\ref{eqn:hdm:res_eqn}) and quantity of interest in (\ref{eqn:hdm:qoi_elem}) satisfy the following:
\begin{enumerate}[label=\textbf{(AH\arabic*)}]
 \item\label{AH1} $\rbm$ is continuously differentiable with respect to both arguments on the domain $\Ucal\times\Dcal$.
 \item\label{AH2} $j$ is continuously differentiable with respect to both arguments on the domain $\Ucal\times\Dcal$.
 \item\label{AH3} $j$ is Lipschitz continuous with respect to its first argument on the domain $\Ucal\times\Dcal$.
 \item\label{AH4} The Jacobian matrix
 \begin{equation}
  \pder{\rbm}{\ubm} : \Rbb^{N_\ubm}\times\Rbb^{N_\mubold} \rightarrow \Rbb^{N_\ubm\times N_\ubm}, \qquad
  \pder{\rbm}{\ubm} : (\ubm,\mubold) \mapsto \pder{\rbm}{\ubm}(\ubm,\mubold)
 \end{equation}
 is Lipschitz continuous with respect to its first argument on the domain $\Ucal\times\Dcal$.
 \item\label{AH5} The state derivative
 \begin{equation}
  \pder{f}{\ubm} : \Rbb^{N_\ubm}\times\Rbb^{N_\mubold} \rightarrow \Rbb^{1\times N_\ubm}, \qquad
  \pder{f}{\ubm} : (\ubm,\mubold) \mapsto \pder{f}{\ubm}(\ubm,\mubold)
 \end{equation}
 is Lipschitz continuous with respect to its first argument on the domain $\Ucal\times \Dcal$.
 \item\label{AH6} The parameter Jacobian matrix
 \begin{equation}
  \pder{\rbm}{\mubold} : \Rbb^{N_\ubm}\times\Rbb^{N_\mubold} \rightarrow \Rbb^{N_\ubm\times N_\mubold}, \qquad
  \pder{\rbm}{\mubold} : (\ubm,\mubold) \mapsto \pder{\rbm}{\mubold}(\ubm,\mubold)
 \end{equation}
 is Lipschitz continuous with respect to its first argument on the domain $\Ucal\times\Dcal$.
 \item\label{AH7} The parameter derivative
 \begin{equation}
  \pder{f}{\mubold} : \Rbb^{N_\ubm}\times\Rbb^{N_\mubold} \rightarrow \Rbb^{1\times N_\mubold}, \qquad
  \pder{f}{\mubold} : (\ubm,\mubold) \mapsto \pder{f}{\mubold}(\ubm,\mubold)
 \end{equation}
 is Lipschitz continuous with respect to its first argument on the domain $\Ucal\times \Dcal$.
 \item\label{AH8} The matrix function
 \begin{equation}
  \Dbm : \Rbb^{N_\ubm}\times\Rbb^{N_\ubm}\times\Rbb^{N_\mubold} \rightarrow \Rbb^{N_\ubm\times N_\ubm}, \qquad
  \Dbm : (\ubm_1,\ubm_2,\zbm) \mapsto \int_0^1 \pder{\rbm}{\ubm}(\ubm_2 + t(\ubm_1-\ubm_2),\mubold) \, dt
 \end{equation}
 is invertible with bounded inverse on the domain $\Ucal\times\Ucal\times\Dcal$.
 \item\label{AH9} For any $\mubold\in\Dcal$, there is a unique solution $\ubm^\star$ satisfying
 $\rbm(\ubm^\star,\mubold) = \zerobold$ and the set of solutions
 $\left\{ \ubm \in \Rbb^{N_\ubm} \suchthat \rbm(\ubm,\mubold) = \zerobold, \forall \mubold\in\Dcal\right\}$
 is a bounded set.
\end{enumerate}
\end{assume}

\begin{assume}\label{assum:eqp}
Consider any open, bounded subset $\Ycal\subset\Rbb^n$. For any full-rank
reduced basis $\Phibold\in\Rbb^{N_\ubm\times n}$, we assume the hyperreduced
residual function in (\ref{eqn:eqp:res_elem}) and quantity of interest in (\ref{eqn:eqp:qoi_elem}) satisfy the following: for any $\rhobold\in\Rcal$,
\begin{enumerate}[label=\textbf{(AR\arabic*)}]
 \item\label{AR1}$\tilde\rbm_\Phibold(\,\cdot\,,\,\cdot\,;\rhobold)$ is continuously differentiable with respect to both
   arguments on the domain $\Ycal\times\Dcal$.
 \item\label{AR2}$\tilde{j}_\Phibold(\,\cdot\,,\,\cdot\,;\rhobold)$ is continuously differentiable with respect to both arguments on the domain $\Ycal\times\Dcal$.
 \item\label{AR3} $\tilde{j}_\Phibold(\,\cdot\,,\,\cdot\,;\rhobold)$ is Lipschitz continuous with respect to its first argument on the domain $\Ycal\times\Dcal$.
 \item\label{AR4} The Jacobian matrix
 \begin{equation}
  \pder{\tilde\rbm}{\tilde\ybm}(\,\cdot\,,\,\cdot\,;\rhobold) : \Rbb^{n}\times\Rbb^{N_\mubold} \rightarrow \Rbb^{n\times n}, \qquad
  \pder{\tilde\rbm}{\tilde\ybm} : (\tilde\ybm,\mubold;\rhobold) \mapsto \pder{\tilde\rbm}{\tilde\ybm}(\tilde\ybm,\mubold;\rhobold)
 \end{equation}
 is Lipschitz continuous with respect to its first argument on the domain $\Ycal\times\Dcal$.
 \item\label{AR5} The state derivative
 \begin{equation}
  \pder{\tilde{j}_\Phibold}{\tilde\ybm}(\,\cdot\,,\,\cdot\,;\rhobold) : \Rbb^{n}\times\Rbb^{N_\mubold} \rightarrow \Rbb^{1\times n}, \qquad
  \pder{\tilde{j}_\Phibold}{\tilde\ybm}(\,\cdot\,,\,\cdot\,;\rhobold) : (\tilde\ybm,\mubold;\rhobold) \mapsto \pder{\tilde{j}_\Phibold}{\tilde\ybm}(\tilde\ybm,\mubold;\rhobold)
 \end{equation}
 is Lipschitz continuous with respect to its first argument on the domain $\Ycal\times \Dcal$.
 \item\label{AR7} The parameter Jacobian matrix
 \begin{equation}
  \pder{\tilde\rbm}{\mubold} : \Rbb^{n}\times\Rbb^{N_\mubold} \rightarrow \Rbb^{n\times N_\mubold}, \qquad
  \pder{\tilde\rbm}{\mubold} : (\tilde\ybm,\mubold) \mapsto \pder{\tilde\rbm}{\mubold}(\tilde\ybm,\mubold)
 \end{equation}
 is Lipschitz continuous with respect to its first argument on the domain $\Ycal\times\Dcal$.
 \item\label{AR8}The parameter derivative
 \begin{equation}
  \pder{\tilde{j}_\Phibold}{\mubold}(\,\cdot\,,\,\cdot\,;\rhobold) : \Rbb^{n}\times\Rbb^{N_\mubold} \rightarrow \Rbb^{1\times N_\mubold}, \qquad
  \pder{\tilde{j}_\Phibold}{\mubold} : (\tilde\ybm,\mubold;\rhobold) \mapsto \pder{\tilde{j}_\Phibold}{\mubold}(\tilde\ybm,\mubold;\rhobold)
 \end{equation}
 is Lipschitz continuous with respect to its first argument on the domain $\Ycal\times \Dcal$.
 \item\label{AR9} The matrix function
 \begin{equation}
  \tilde\Dbm_\Phibold(\,\cdot\,,\,\cdot\,,\,\cdot\,;\rhobold) : \Rbb^n\times\Rbb^n\times\Rbb^{N_\mubold} \rightarrow \Rbb^{n \times n}, \qquad
  \tilde\Dbm_\Phibold : (\tilde\ybm_1,\tilde\ybm_2,\zbm;\rhobold) \mapsto \int_0^1 \pder{\tilde\rbm_\Phibold}{\tilde\ybm}(\tilde\ybm_2 + t(\tilde\ybm_1-\tilde\ybm_2),\mubold;\rhobold) \, dt
 \end{equation}
 is invertible with bounded inverse on the domain $\Ycal\times\Ycal\times\Dcal$.
 \item For any $\mubold\in\Dcal$, there is a unique solution $\tilde\ybm^\star$ satisfying
 $\tilde\rbm_\Phibold(\tilde\ybm^\star,\mubold;\rhobold) = \zerobold$, and the set of solutions
 $ \left\{ \ybm \in \Rbb^n \suchthat \tilde\rbm_\Phibold(\ybm,\mubold;\rhobold) = \zerobold, \forall \mubold\in\Dcal\right\}$
 is a bounded set. 
\end{enumerate}
\end{assume}

\begin{remark}
\ref{AR1}-\ref{AR8} follow directly from Assumption~\ref{assum:hdm} in the case where $\rhobold=\onebold$ because $\hat\rbm_\Phibold(\hat\ybm,\mubold) = \tilde\rbm_\Phibold(\hat\ybm,\mubold;\onebold)$
and the relationship between $\rbm$ and $\hat\rbm$ in (\ref{eqn:rom:res}).
\end{remark}


From these assumptions, the following residual-based estimates on the primal state, adjoint state, and output hold for each model.
These lemmas will be used to establish residual-based error estimates between the HDM and hyperreduced outputs
(Theorem~\ref{the:qoi_grad_errbnd} and its corollaries).
\begin{lemma}\label{lem:hdm_diff}
Under Assumptions~\ref{assum:hdm}-\ref{assum:eqp}, for any $\ubm\in\Ucal$, $\lambdabold\in\Lambda$, and $\mubold\in\Dcal$, there exist constants
$\kappa,\tau,\omega>0$ such that
\begin{equation}
 \norm{\ubm^\star(\mubold)-\ubm} \leq \kappa \norm{\rbm(\ubm,\mubold)}, \qquad
 \norm{\lambdabold^\star(\mubold)-\lambdabold} \leq \tau \norm{\rbm(\ubm,\mubold)} + \omega \norm{\rbm^\lambda(\lambdabold,\ubm,\mubold)},
\end{equation}
where $\Ucal\subset\Rbb^{N_\ubm}$, $\Lambda\subset\Rbb^{N_\ubm}$, and $\Dcal\subset\Rbb^{N_\mubold}$ are bounded
subsets. Furthermore, there exists constants $\kappa',\tau',\omega'>0$ such that
\begin{equation}
\begin{aligned}
 |j(\ubm^\star(\mubold),\mubold) - j(\ubm,\mubold)| &\leq \kappa' \norm{\rbm(\ubm,\mubold)}, \\
 \norm{\gbm^\lambdabold(\lambdabold^\star(\mubold),\ubm^\star(\mubold),\mubold)-\gbm^\lambdabold(\lambdabold,\ubm,\mubold)} &\leq \tau' \norm{\rbm(\ubm,\mubold)} + \omega' \norm{\rbm^\lambda(\lambdabold,\ubm,\mubold)}.
\end{aligned}
\end{equation}
\begin{proof}
Proposition A.1 and A.2 of \cite{zahr_efficient_2019}.
\end{proof}
\end{lemma}

\begin{lemma}\label{lem:eqp_diff}
Under Assumptions~\ref{assum:hdm}-\ref{assum:eqp}, for any $\tilde\ybm\in\Ycal$, $\tilde\zbm\in\Zcal$, $\rhobold\in\Rcal$, and $\mubold\in\Dcal$, there
exist constants $\kappa,\tau,\omega>0$ such that
\begin{equation}
 \norm{\tilde\ybm_\Phibold^\star(\mubold;\rhobold)-\tilde\ybm} \leq \kappa \norm{\tilde\rbm_\Phibold(\tilde\ybm,\mubold;\rhobold)}, \qquad
 \norm{\tilde\lambdabold_\Phibold^\star(\mubold;\rhobold)-\tilde\zbm} \leq \tau \norm{\tilde\rbm_\Phibold(\tilde\ybm,\mubold;\rhobold)} + \omega \norm{\tilde\rbm_\Phibold^\lambda(\tilde\zbm,\tilde\ybm,\mubold;\rhobold)},
\end{equation}
where $\Ycal\subset\Rbb^n$, $\Zcal\subset\Rbb^n$, and $\Dcal\subset\Rbb^{N_\mubold}$ are bounded
subsets. 
Furthermore, there exists constants $\kappa',\tau',\omega'>0$ such that
\begin{equation}
\begin{aligned}
 |\tilde{j}_\Phibold(\tilde\ybm_\Phibold^\star(\mubold;\rhobold),\mubold;\rhobold) - \tilde{j}_\Phibold(\tilde\ybm,\mubold;\rhobold)| &\leq \kappa' \norm{\tilde\rbm_\Phibold(\tilde\ybm,\mubold;\rhobold)}, \\
 \norm{\tilde\gbm_\Phibold^\lambdabold(\tilde\lambdabold_\Phibold^\star(\mubold;\rhobold),\tilde\ybm_\Phibold^\star(\mubold;\rhobold),\mubold;\rhobold)-\tilde\gbm_\Phibold^\lambdabold(\tilde\zbm,\tilde\ybm,\mubold;\rhobold)} &\leq \tau' \norm{\tilde\rbm_\Phibold(\tilde\ybm,\mubold;\rhobold)} + \omega' \norm{\tilde\rbm_\Phibold^\lambda(\tilde\zbm,\tilde\ybm,\mubold;\rhobold)}.
\end{aligned}
\end{equation}
\begin{proof}
Proposition A.1 and A.2 of \cite{zahr_efficient_2019}.
\end{proof}
\end{lemma}


\section{Proof of residual-based output error estimates}
\label{sec:appendix_B}
We prove the Theorem~\ref{the:qoi_grad_errbnd} and Corollaries~\ref{cor:qoi_errbnd}-\ref{cor:qoi_grad_errbnd}. For notational brevity, we suppress the subscript $\Phibold$ and EQP weights $\rhobold$ for all functions, and we drop off the input argument of the solution terms $A^\star(\mubold)$.

\begin{proof}[Proof of Theorem~\ref{the:qoi_grad_errbnd}]
First, we expand the quantities of interest using their definitions in terms of $j$ and $\tj$ as
\begin{equation}
 \abs{f(\mubold) - \tilde{f}(\mubold)} = \abs{j(\ubm^\star, \mubold)-\tj(\tybm^\star,\mubold)}.
\end{equation}
Then, we add and subtract two terms, $\hj(\hybm, \mubold)$ and $\tj(\hybm, \mubold)$, and use the triangle inequality to obtain
\begin{equation}
  \abs{f(\mubold) - \tilde{f}(\mubold)} 
  \leq 
  \abs{j(\ubm^\star, \mubold)-\hj(\hybm^\star, \mubold)} 
  + \abs{\tj(\hybm^\star,\mubold)-\tj(\tybm^\star,\mubold)} 
  + \abs{\hj(\hybm^\star, \mubold)-\tj(\hybm^\star, \mubold)}.
\end{equation}
From Lemma~\ref{lem:hdm_diff} and \ref{lem:eqp_diff}, the first two terms can be written in terms of the corresponding residuals
\begin{equation}
  \abs{f(\mubold) - \tilde{f}(\mubold)}
  \leq
  c_1 \norm{\rbm(\Phibold \hybm^\star, \mubold)}
  + c_2 \norm{\trbm(\hybm^\star, \mubold)}
  + \abs{\hj(\hybm^\star, \mubold)-\tj(\hybm^\star, \mubold)},
\end{equation}
where $c_1, c_2 > 0$ are constants, which is the desired result in (\ref{eqn:eqp:qoi_errbnd}).

Similarly, we expand the gradient error in terms of the $\gbm^\lambda$ and $\tilde\gbm^\lambda$ operators,
add and subtract $\hglam(\hlam^\star, \hybm^\star, \mubold)$ and $\tglam(\hlam^\star, \hybm^\star, \mubold)$,
and use the triangle inequality to obtain
\begin{equation}
\begin{aligned}
\norm{\nabla f(\mubold) - \nabla \tilde{f}(\mubold)}
\leq
&\norm{
    \glam(\lambdabold^\star, \ubm^\star, \mubold)-\hglam(\hlam^\star, \hybm^\star, \mubold)
} +
\norm{
  \tglam(\hlam^\star, \hybm^\star, \mubold)-\tglam(\tlam^\star, \tybm^\star, \mubold)
} +\\
&\norm{
    \hglam(\hlam^\star, \hybm^\star, \mubold)-\tglam(\hlam^\star, \hybm^\star, \mubold)
}
\end{aligned}
\end{equation}
From Lemma~\ref{lem:hdm_diff} and \ref{lem:eqp_diff}, the first two terms can be written in terms of the corresponding primal residuals and adjoint residuals as
\begin{equation}
  \begin{aligned}
   \norm{\nabla f(\mubold) - \nabla\tf(\mubold)} \leq
    & c_1'\norm{\rbm(\Phibold \hybm^\star,\mubold)} + c_2' \norm{\rlam(\Phibold\hlam^\star,\Phibold\hybm^\star,\mubold)} +\\
    & c_3'\norm{\trbm(\hybm^\star,\mubold)}+c_4'\norm{\trlam(\hlam^\star,\hybm^\star,\mubold)}+\\
    & \norm{\hglam(\hlam^\star,\hybm^\star,\mubold)-\tglam(\hlam^\star,\hybm^\star,\mubold)}
  \end{aligned}
\end{equation}
where $c_1', c_2', c_3', c_4' > 0$ are constants, which is the desired result in (\ref{eqn:eqp:qoi_grad_errbnd}).
\end{proof}

\begin{proof}[Proof of Corollary~\ref{cor:qoi_errbnd}]
Because $\ubm^\star \in \mathrm{Ran}~\Phibold_k$, the ROM will recover the exact solution ($\ubm^\star = \Phibold\hat\ybm^\star$)
and therefore the first term in (\ref{eqn:eqp:qoi_errbnd}) will vanish, which gives
\begin{equation}
    \abs{f(\mubold) - \tilde{f}(\mubold)}
      \leq c_2 \norm{\trbm(\hybm^\star, \mubold)} + \norm{\hj(\hybm^\star, \mubold)-\tj(\hybm^\star, \mubold)}.
\end{equation}
Furthermore, because $\rhobold$ is the solution of (\ref{eqn:eqp:linprog}) with
$\Ccal_{\Phibold,\Xibold,\deltabold}\subset\Ccal_{\Phibold,\Xibold,\delta_\mathtt{rp}}^\mathtt{rp}\cap\Ccal_{\Phibold,\Xibold,\delta_\mathtt{q}}^\mathtt{q}$, we have
\begin{equation}
 \norm{\tilde\rbm(\hat\ybm^\star,\mubold)} \leq \delta_\mathtt{rp}, \qquad
 \abs{\hat{j}(\hat\ybm^\star,\mubold)-\tilde{j}(\hat\ybm^\star,\mubold)} \leq \delta_\mathtt{q},
\end{equation}
which follows directly from (\ref{eqn:eqp:rescon})-(\ref{eqn:eqp:rescon_simple}) and leads to the desired result in (\ref{eqn:eqp:qoi_errbnd2}).
\end{proof}

\begin{proof}[Proof of Corollary~\ref{cor:qoi_grad_errbnd}]
Because $\ubm^\star \in \mathrm{Ran}~\Phibold_k$ and $\lambdabold^\star\in\mathrm{Ran}~\Phibold_k$, the ROM will
recover the exact primal and adjoint solutions ($\ubm^\star = \Phibold\hat\ybm^\star$, $\lambdabold^\star = \Phibold\hat\lambdabold^\star$) and therefore the first two terms in (\ref{eqn:eqp:qoi_grad_errbnd}) will vanish, which gives
\begin{equation}
 \norm{\nabla f(\mubold) - \nabla\tilde{f}(\mubold)} \leq
 c_3'\norm{\tilde\rbm(\hat\ybm^\star,\mubold)} + 
 c_4'\norm{\tilde\rbm^\lambda(\hat\lambdabold^\star,\hat\ybm^\star,\mubold)} +
 \norm{\hat\gbm^\lambda(\hat\lambdabold^\star,\hat\ybm^\star,\mubold) - \tilde\gbm^\lambda(\hat\lambdabold^\star,\hat\ybm^\star,\mubold)}.
\end{equation}
Furthermore, because $\rhobold$ is the solution of (\ref{eqn:eqp:linprog}) with
$\Ccal_{\Phibold,\Xibold,\deltabold}\subset\Ccal_{\Phibold,\Xibold,\delta_\mathtt{rp}}^\mathtt{rp}\cap\Ccal_{\Phibold,\Xibold,\delta_\mathtt{ra}}^\mathtt{ra}\cap\Ccal_{\Phibold,\Xibold,\delta_\mathtt{ga}}^\mathtt{ga}$, we have
\begin{equation}
 \norm{\tilde\rbm(\hat\ybm^\star,\mubold)} \leq \delta_\mathtt{rp}, \qquad
 \norm{\tilde\rbm^\lambda(\hat\lambdabold^\star,\hat\ybm^\star,\mubold)} \leq \delta_\mathtt{ra}, \qquad
 \norm{\hat\gbm^\lambda(\hat\lambdabold^\star,\hat\ybm^\star,\mubold)-\tilde\gbm^\lambda(\hat\lambdabold^\star,\hat\ybm^\star,\mubold)} \leq \delta_\mathtt{ga},
\end{equation}
which follows directly from (\ref{eqn:eqp:rescon})-(\ref{eqn:eqp:rescon_simple}) and leads to the desired result in (\ref{eqn:eqp:qoi_grad_errbnd2}).
\end{proof}

\bibliographystyle{plain}
\bibliography{biblio}

\end{document}